\documentclass[11pt]{article}%
\usepackage{amssymb,amsmath,amsfonts,amsthm,array,bm,bbm,color,arydshln,stmaryrd,enumerate}%
\usepackage[colorlinks,linkcolor=red,citecolor=blue,urlcolor=blue]{hyperref}

\usepackage{graphicx}
\providecommand{\U}[1]{\protect\rule{.1in}{.1in}}

\numberwithin{equation}{section}

\newtheorem{theorem}{Theorem}[section]
\newtheorem{lemma}[theorem]{Lemma}
\newtheorem{corollary}[theorem]{Corollary}
\newtheorem{proposition}[theorem]{Proposition}
\newtheorem{remark}[theorem]{Remark}

\newtheorem{definition}[theorem]{Definition}

\newtheorem{assumption}[theorem]{Assumption}
\newtheorem{conjecture}[theorem]{Conjecture}
\newtheorem{system}{System}

\def\cA{\mathcal{A}}
\def\cC{\mathcal{C}}
\def\cE{\mathcal{E}}
\def\cF{\mathcal{F}}
\def\cH{\mathcal{H}}
\def\cI{\mathcal{I}}
\def\L{\mathcal{L}}
\def\cP{\mathcal{P}}
\def\cR{\mathcal{R}}
\def\cQ{\mathcal{Q}}
\def\cS{\mathcal{S}}

\def\E{\mathbb{E}}
\def\F{\mathbb{F}}
\def\N{\mathbb{N}}
\def\P{\mathbb{P}}
\def\Q{\mathbb{Q}}
\def\R{\mathbb{R}}
\def\T{\mathbb{T}}
\def\Z{\mathbb{Z}}

\def\<{\langle}
\def\>{\rangle}
\def\d{{\rm d}}
\def\div{{\rm div \,}}
\def\curl{{\rm curl}}
\def\Tr{{\rm Tr \,}}

\def\eps{\varepsilon}
\def\loc{{\rm loc}}

\newcommand{\red}[1]{{\color{red} #1}}

\begin{document}

\title{Weak well-posedness by transport noise for a class of\\ 2D fluid dynamics equations}
\author{Lucio Galeati\thanks{Email: lucio.galeati@epfl.ch. Research Group Math-AMCV, EPFL, 1015 Lausanne, Switzerland.}
\quad Dejun Luo\thanks{Email: luodj@amss.ac.cn. Key Laboratory of RCSDS,
Academy of Mathematics and Systems Science, Chinese Academy of Sciences,
Beijing 100190, China, and School of Mathematical Sciences, University of
Chinese Academy of Sciences, Beijing 100049, China. } }
\maketitle

\vspace{-15pt}

\begin{abstract}
A fundamental open problem in fluid dynamics is whether solutions to $2$D Euler equations with $(L^1_x\cap L^p_x)$-valued vorticity are unique, for some $p\in [1,\infty)$.
A related question, more probabilistic in flavour, is whether one can find a physically relevant noise regularizing the PDE.
We present some substantial advances towards a resolution of the latter, by establishing well-posedness in law for solutions with $(L^1_x\cap L^2_x)$-valued vorticity and finite kinetic energy, for a general class of stochastic 2D fluid dynamical equations;
the noise is spatially rough and of Kraichnan type and we allow the presence of a deterministic forcing $f$.
This class includes as primary examples logarithmically regularized 2D Euler and hypodissipative 2D Navier--Stokes equations. In the first case, our result solves the open problem posed by Flandoli in \cite{Fla}. In the latter case, for well-chosen forcing $f$, the corresponding deterministic PDE without noise has recently been shown in \cite{AlbCol} to be ill-posed; consequently, the addition of noise truly improves the solution theory for such PDE.
\end{abstract}



\section{Introduction}\label{sec-intro}

Consider the incompressible Euler equations on $\R^2$ in vorticity form:
\begin{equation}\label{eq:intro-euler}
\begin{cases}
	\partial_t \omega + u\cdot\nabla\omega = f, \\
	\nabla \cdot u = 0, \quad \curl\, u = \omega.
\end{cases}\end{equation}
Here $u$ and $\omega$ are the unknowns, representing respectively the velocity and vorticity of a perfect fluid; instead the time-dependent forcing term $f$ and an initial configuration $u\vert_{t=0}=u_0$, $\omega\vert_{t=0}=\omega_0$ are prescribed.
It is well-known that, under the above conditions, $u$ can be recovered from $\omega$ by the Biot--Savart law, namely $u=\curl^{-1}\omega=K\ast \omega$, so that \eqref{eq:intro-euler} can be regarded as an \emph{active scalar} equation with unknown $\omega$ solely (we refer to Section \ref{subsec:notation} for an explanation of relevant notations, standard facts and function spaces adopted in this paper).

Let us quickly recall some known results for the PDE \eqref{eq:intro-euler}, in the absence of forcing:\footnote{All of them can be easily generalized in the presence of suitable $f$, see e.g. \cite[Theorem 1.0.3]{ABCDGJK} for Yudovich's result with forcing $f\in L^1_t (L^1_x\cap L^\infty_x)$.}

\begin{itemize}
\item If $\omega_0\in L^1_x\cap L^\infty_x\cap \dot H^{-1}_x$ (so that $u_0 = \curl^{-1}\omega_0\in L^2_x$), in the celebrated work \cite{Yudovich63}, Yudovich established existence and uniqueness of weak solutions $\omega\in L^1_t(L^1_x\cap L^\infty_x \cap \dot H^{-1}_x)$.
Moreover, regularity of initial data is propagated, namely if $\omega_0\in C^k_x$, the same will be true for any $t\geq 0$; see the monographs \cite{MajBer02,MarPul94} for an overview.
\item If $\omega_0\in L^1_x\cap L^p_x$ for some $p<\infty$, then it is possible to construct solutions $\omega\in C^0_t (L^1_x\cap L^p_x)$ by several different approximations schemes, as observed by DiPerna and Majda in \cite{DiPMaj}; the result have then been refined by Schochet \cite{Schochet} and Delort \cite{Delort} to allow $\omega_0$ to be a signed measure ``with a sign preference''.
Uniqueness here is a longstanding open problem.
\item For $\omega_0$ below $L^1_x$, examples of non-uniqueness are now known.
Grotto and Pappalettera \cite{GroPap} revisit a classical example from Marchioro and Pulvirenti \cite{MarPul94} to show non-uniqueness of point vortex solutions (namely measure-valued objects of the form $\omega_t=\sum_{j=1}^n a_j \delta_{x^j_t}$) due to finite time collapse of three vortices into one; instead Brué and Colombo \cite{BruCol} use convex integration techniques to construct infinitely many solutions of finite energy, with vorticity in the Lorentz space $L^{1,\infty}_x$.
\end{itemize}

The problem of uniqueness for $L^p_x$-valued vorticity has received considerable attention in the last years, with several different attempts at disproving it.
On one hand, Bressan et al. \cite{BreMur, BreShe} have set forth a (numerically assisted) strategy towards non-uniqueness of the unforced 2D Euler, with vorticity in $L^p_{\rm loc}$, based on a symmetry breaking argument; this program however is not completely rigorous yet.
On the other, Vishik successfully showed in \cite{Vishik1, Vishik2} that, for any $p\in (1,\infty)$, one can find a carefully chosen forcing $f\in L^1_t (L^1_x\cap L^p_x)$ such that there are infinitely many weak solutions $\omega\in C^0_t(L^1_x\cap L^p_x)$ with zero initial condition, establishing sharpness of Yudovich's result in the forced case; see also the revisitation of this result from \cite{ABCDGJK}.

Understanding better $2$D Euler solutions with $L^p_x$-valued vorticity and the mechanisms under which non-uniqueness may arise is of fundamental importance in order to get a better insight on their \emph{turbulent behaviour}, which might not be displayed in Yudovich solutions which are quite regular; for instance the review \cite{Boffetta} predicts $2$D turbulence to display a $1/3$-H\"older velocity field $u$.
One possible way to forcefully exhibit turbulence in the equation is by the introduction of an exhogenous \emph{noise component}, which at the same time could be thought of as arising endogenously by the action of the turbulent small scales of the fluid.
There is by now a general consensus in the literature that the correct noise to consider for this task should be of \emph{Stratonovich gradient} type; the corresponding stochastic PDE (written directly only for $\omega$) is thus given by
\begin{equation}\label{eq:intro-stochastic-euler}
	\d \omega + (\curl^{-1} \omega)\cdot\nabla\omega \, \d t+ \circ \d W\cdot \nabla\omega
	= f\, \d t.
\end{equation}
See e.g. \cite{BrCaFl,Holm2015, Memin, FP21} for several alternative derivations of \eqref{eq:intro-stochastic-euler}, as well as the more detailed discussions in Section \ref{subsec:literature}.
Here $W(t,x)$ is a space-time Gaussian noise, Brownian in time and coloured in space\footnote{It is well known that in order to give meaning to the SPDE \eqref{eq:intro-stochastic-euler} $W$ cannot be taken too rough, in particular it cannot be white in space; rather, it should be roughly speaking a \emph{trace class noise}. See Section \ref{sec:preliminaries} for a more detailed explanation.}, and the symbol $\circ\d $ denotes Stratonovich integration in time (see Section \ref{sec:preliminaries} for more details), which is the most physically reasonable one (instead of It\^o) due to the Wong-Zakai principle.
It is also natural (as well as mathematically convenient) to take $W$ to be spatially divergence-free, so that the SPDE \eqref{eq:intro-stochastic-euler} shares many properties of its deterministic counterpart \eqref{eq:intro-euler}: a priori estimates, Casimir invariants, Lagrangian representations, etc. Indeed this allows to derive several analogues of the aforementioned results for \eqref{eq:intro-euler}, see e.g. \cite{BrFlMa} for a stochastic Yudovich theorem, \cite{BrzMau} for Delort's result and \cite{CrFlHo} for the Beale-Kato-Majda blow-up criterion.

In view of the above discussion for \eqref{eq:intro-euler}, a closely related more probabilistic question is then: can the introduction of the noise $W$ as in \eqref{eq:intro-stochastic-euler} restores uniqueness of solutions with $L^p_x$-valued vorticity, possibly also in the presence of an external forcing $f$?

This problem now falls in the growing field of \emph{regularization by noise} phenomena, see \cite{FlaBook, Gess} for overviews.
The hope that a sufficiently active, yet physically meaningful noise $W$ should improve the solution theory for an ODE or PDE, besides being by now understood in numerous examples, is grounded in the idea that often non-uniqueness (as well as formation of singularity) should be a very rare event, only happening at some special configurations.
In other words, non-uniqueness should be the exception, rather than the standard;
if one accepts this philosophy, then the noise should play the role of immediately pushing the system out of such configurations, restoring well-posedness.
A much more sophisticated question, outlined in the general regularization by noise program by Flandoli in \cite{FlaBook}, is then whether one can devise \emph{selection criteria} for physically meaningful solutions by performing a \emph{vanishing noise limit}. This is a much harder problem, which so far has not been properly understood even in the ODE case.
At the same time, in recent years convex integration techniques have been successfully applied to fluid dynamics SPDEs as well, making researchers doubt whether regularization by noise is even possible here; see \cite{BecHof} for a deeper discussion adopting this, equivalently plausible, perspective.

Unfortunately, $2$D Euler equations are still too hard to be tackled with our current techniques, and in this paper we are not able to solve the problem as formulated above.
We can however get ``very close to it'' by establishing such result for two other classes of PDEs, related to $2$D Euler; the price we have to pay is that these are not fundamental equations of fluid dynamics anymore.
The first SPDE we will consider is a stochastic analogue of the so called \emph{logEuler equations}; it differs from \eqref{eq:intro-euler} by an additional logarithmic regularization on $\curl^{-1}$.
Its interest as a test candidate for regularization by noise results for inviscid fluids comes from the open problem proposed by Flandoli in \cite{Fla}; therein the operator $T_\gamma$ is actually replaced by the nicer $(-\Delta)^{-\eps}$ with $\eps>0$, which we can treat by the same means, cf. Remark \ref{rem:flandoli}.
%

\begin{system}[$2$D logEuler]\label{system:logEuler}
Let $\gamma>1/2$ and denote by $T_\gamma$ the Fourier multiplier with symbol $\log^{-\gamma} (e+|\nabla|)$; then the $2$D logarithmically regularized Euler equation is given by
\begin{equation}\label{eq:intro-logeuler}
	\d \omega + (T_\gamma \curl^{-1} \omega)\cdot\nabla\omega \, \d t+ \circ \d W\cdot\nabla\omega = f\, \d t.
\end{equation}
\end{system}

The (deterministic) logEuler equations have already been studied by several researchers in the literature, as they fall in the general class of inviscid models introduced in \cite{ChCoWu}.
Compared to classical Euler, they present a slightly richer behaviour for initial vorticities in the borderline critical spaces $H^1_x\cap \dot H^{-1}_x$, see \cite{ChaeWu, DongLi, Kwon}; however for solutions with $L^p_x$-valued vorticity, their uniqueness is as open as in Euler.

The second SPDE we will deal with are the stochastic 2D hypodissipative Navier-Stokes equations, where the Biot--Savart kernel is not regularized, rather we add a fractionally dissipative term on the r.h.s. of \eqref{eq:intro-stochastic-euler}.

\begin{system}[$2$D hypodissipative Navier--Stokes]\label{system:hypoNS}
Let $\beta\in (0,2)$ and denote by $\Lambda^\beta$ the fractional Laplacian of order $\beta$, namely the Fourier multiplier with symbol $|\nabla|^\beta$; then the stochastic $2$D hypodissipative Navier--Stokes equation in vorticity form is given by
\begin{equation}\label{eq:intro-hypoNS}
	\d \omega + (\curl^{-1}\omega)\cdot\nabla\omega \, \d t+ \circ \d W \cdot\nabla\omega
	= (f- \Lambda^\beta \omega)\, \d t.
\end{equation}
\end{system}

Clearly, we are mostly interested in the regime where $\beta$ is very small, so that the regularization coming from $\Lambda^\beta$ is not that strong, and the evolution of \eqref{eq:intro-hypoNS} should resemble the one of \eqref{eq:intro-stochastic-euler}.
Our interest in System \ref{system:hypoNS} comes from a recent result by Albritton and Colombo \cite{AlbCol}, who established non-uniqueness of solutions to the forced, deterministic PDE for any $\beta\in (0,2)$, see Proposition \ref{prop:pde-nonuniquess} below for a precise statement; the proof is based on a readaption of the aforementioned unstable background method by Vishik \cite{Vishik1, Vishik2}. In particular, this equation is now a natural test field to understand if the introduction of noise as in \eqref{eq:intro-hypoNS} may affect this nonuniqueness scenario; we answer positively, by showing that a suitably chosen $W$ indeed restores uniqueness in law, even in the presence of the forcing $f$ considered in \cite{AlbCol}.

The next statement loosely summarizes the main findings of this paper; we refer to Definition \ref{defn:solution-nonlinear-SPDE} below for the meaning of probabilistically weak solutions to \eqref{eq:intro-logeuler} and \eqref{eq:intro-hypoNS}, and $\mathcal{L}(\omega)$ in the following result stands for the law of solution process.

\begin{theorem}\label{thm:main-theorem}
Let $f\in L^1_t \big(L^1_x\cap L^2_x\cap \dot H^{-1}_x \big)$ be a deterministic forcing term.
The following assertions hold.
\begin{itemize}
	\item ($2$D logEuler)
	There exists a noise $W$, with paths in $C^0_t L^2_\loc$, such that for any $\omega_0\in L^1_x\cap L^2_x\cap \dot{H}^{-1}_x$, global weak existence and uniqueness in law hold for System \ref{system:logEuler}, in the class of solutions $\omega$ satisfying
	\begin{align*}
	\P \mbox{-a.s.}\quad \omega\in L^2_t \big(L^1_x\cap L^2_x\cap \dot{H}^{-1}_x \big).
	\end{align*}
	Furthermore, the solution map $\omega_0\mapsto \mathcal{L}(\omega)$ is continuous with respect to weak convergence of probability measures.
	Finally, if $\gamma>3/4$, then the noise has trajectories in $C^0_t C^0_\loc$.
	
	\item ($2$D Hypodissipative Navier--Stokes)
	There exists a noise $W$ such that for any $\omega_0\in L^1_x\cap L^2_x\cap \dot{H}^{-1}_x$, global weak existence and uniqueness in law hold for System \ref{system:hypoNS}, in the class of solutions $\omega$ satisfying
	\begin{align*}
	\P \mbox{-a.s.}\quad \omega\in L^2_t \big(L^1_x\cap H^{\beta/2}_x\cap \dot{H}^{-1}_x \big).
	\end{align*}
	Furthermore, the solution map $\omega_0\mapsto \mathcal{L}(\omega)$ is continuous with respect to weak convergence of probability measures.
	Finally, the noise has paths in $C^0_t C^{\beta/2-\eps}_\loc$ for any $\eps>0$.
\end{itemize}
\end{theorem}

Theorem \ref{thm:main-theorem} will be a consequence of Theorems \ref{thm:main-logeuler}-\ref{thm:main-hypoNS}; the proofs will be presented in Section \ref{subsec:uniqueness-nonlinear}.
The noise $W$ in the statement is constructed in a rather explicit manner, in the sense that an exact formula for its covariance function $Q$ (as defined in Section \ref{subsec:structure-noise}) will be given (cf. equations \eqref{eq:covariance-logeuler}-\eqref{eq:covariance-hypoNS} therein); $W$ is fairly rough in space (e.g. non differentiable) and of \textit{Kraichnan type}, see Section \ref{subsec:structure-noise} for its precise description.
Our main interest is in the study of Systems \ref{system:logEuler}-\ref{system:hypoNS} when they are close to the original SPDE \eqref{eq:intro-stochastic-euler} (in particular, $\beta\ll 1$); let us remark however that Theorem \ref{thm:main-theorem} nicely displays how, when the regularizing terms $T_\gamma$ and $\Lambda^\beta$ become stronger, the roughness of the noise can be accordingly relaxed.

Theorem \ref{thm:main-theorem} illustrates the regularizing effects of Stratonovich transport noise of Kraichnan type on $2$D fluid dynamics PDEs.
In order to compare it to the nonuniqueness result from \cite{AlbCol}, we need to verify that all solutions in consideration belong to the same regularity class; for this reason, keeping track of the rather explicit construction from \cite{AlbCol}, we will prove in Section \ref{sec:nonuniqueness-deterministic} the following.

\begin{proposition}\label{prop:pde-nonuniquess}
Let $\beta\in (0,1)$. There exists a forcing $f\in L^1_t \big(L^1_x\cap L^2_x\cap \dot H^{-1}_x \big)$ such that the deterministic 2D hypodissipative Navier--Stokes equations
\begin{equation}\label{eq:hypodissipative-NS-deterministic}
\partial_t \omega + (\curl^{-1} \omega)\cdot\nabla \omega = -\Lambda^\beta\omega + f, \quad u=\curl^{-1}\omega
\end{equation}
admit two distinct weak solutions $\bar\omega$, $\omega$, starting with the same initial datum $\omega_0\equiv 0$, both belonging to $L^2_t \big( L^1_x\cap \dot{H}^{-1}_x \cap H^{\beta/2}_x \big)$.
\end{proposition}

Comparing Theorem \ref{thm:main-theorem} and Proposition \ref{prop:pde-nonuniquess}, we see that the addition of rough transport noise to \eqref{eq:hypodissipative-NS-deterministic} truly improves the solution theory of 2D hypodissipative Navier--Stokes equations.

After the completion of this work, the outstanding paper \cite{CogMau} appeared. Therein, the authors are able to establish strong existence and pathwise uniqueness for the true $2$D Euler equations, for instance for vorticities in $L^2_x$ and rough Kraichnan noise with paths in $C^0_t C^{1/2-\eps}_{\loc}$. The proof is based on a completely different approach from ours, not making use of of Girsanov transform; see Remark \ref{rem:apriori-estimates-nonlinear} for a deeper discussion.

\subsection{Strategy of proof} \label{subsec:strategy-proof}

Let us briefly describe here the main ideas behind Theorem \ref{thm:main-theorem}.
In order to treat at once both Systems \ref{system:logEuler} and \ref{system:hypoNS}, while also trying to extract useful information on the original problem \eqref{eq:intro-stochastic-euler}, it will be convenient to consider the following more general abstract stochastic fluid model:
\begin{equation}\label{eq:intro-abstract-spde}
	\d \omega + (\cR \curl^{-1} \omega)\cdot\nabla\omega \, \d t+ \circ \d W\cdot\nabla\omega
	= (f- \nu\Lambda^\beta \omega)\, \d t.
\end{equation}
Here we denote by $\cR$ a general Fourier multiplier associated to some bounded map $r:\R^2\to \R$ and $\nu$ is a nonnegative parameter.
System \ref{system:logEuler} is recovered by taking $\nu=0$ and $\cR=T_\gamma$, System \ref{system:hypoNS} by $\nu=1$ and $\cR=I$ and \eqref{eq:intro-stochastic-euler} by $\nu=0$ and $\cR=I$.

The idea behind the proof is based on the exposition presented by Flandoli in \cite{Fla}; let us also mention that, according to \cite{Fla,CrFlMa}, this approach is originally due to P. Malliavin.

Let $\cQ$ be the covariance operator associated to the noise $W$ (cf. Section \ref{subsec:structure-noise}), so that one can write $W=\cQ^{1/2} \Xi$ where $\Xi$ is a cylindrical noise on $L^2_x$.
Then, at least formally, one can collect terms so to write the SPDE \eqref{eq:intro-abstract-spde} as
\begin{align*}
	\d \omega + \circ \d \cQ^{1/2} \bigg( \Xi_\cdot + \int_0^\cdot \cQ^{-1/2}\, \cR\, \curl^{-1} \omega_s\, \d s\bigg) \cdot\nabla \omega = (f- \nu\Lambda^\beta \omega)\, \d t.
\end{align*}
In particular, if we define a new process $\tilde\Xi$ by
\begin{align*}
	\tilde \Xi_t:= \Xi_t + \int_0^t \cQ^{-1/2}\, \cR\, \curl^{-1} \omega_s\, \d s
\end{align*}
and correspondigly set $\tilde W:=\cQ^{1/2} \tilde \Xi$, then the above SPDE is nothing but
\begin{equation}\label{eq:intro-linear-spde-0}
	\d \omega + \circ \d\tilde W \cdot\nabla \omega = (f- \nu\Lambda^\beta \omega)\, \d t.
\end{equation}
If the noise $W$ is ``sufficiently active'', we may hope to reabsorb the nonlinearity into $W$; in other words, $\tilde W$ should display the same behaviour as the original $W$, the nonlinearity just being accounted as a not too relevant perturbation.
Pushing this heuristic a step further, we may regard \eqref{eq:intro-linear-spde-0} as a \emph{linear} equation driven by $\tilde{W}$, whose well-posedness theory should be much easier to establish compared to the original nonlinear \eqref{eq:intro-abstract-spde}.

The passage from \eqref{eq:intro-abstract-spde} to \eqref{eq:intro-linear-spde-0} can be made  mathematically rigorous by means of \emph{Girsanov's theorem}.
In particular, under suitable assumptions on the solution $\omega$ to \eqref{eq:intro-abstract-spde} in consideration, one can define a change of measure $\d \Q= M_T\, \d \P$, for $M$ given by the exponential martingale
\begin{align*}
	M_t= \exp\bigg(-\int_0^t \big\<\cQ^{-1/2}\, \cR\, \curl^{-1} \omega_s, \d\Xi_s\big\> - \frac12 \int_0^t \big\| \cQ^{-1/2}\, \cR\, \curl^{-1} \omega_s \big\|_{L^2_x}^2\,\d s \bigg),\quad t\le T
\end{align*}
and correspondingly show that, under the new measure $\Q$, the process $\tilde{W}$ has the same law on $[0,T]$ as the original $W$.
Said otherwise, if the hypothesis of Girsanov's theorem are verified, then the solution $\omega$ to the nonlinear SPDE \eqref{eq:intro-abstract-spde} under $\P$ is also a solution to the linear one under $\Q$.
The next step is then to show that, under suitable conditions on the noise $W$, the linear SPDE \eqref{eq:intro-linear-spde-0} enjoys pathwise uniqueness of solutions;
combining these two ingredients overall allows us to deduce the uniqueness in law of weak solutions to the original nonlinear equation \eqref{eq:intro-abstract-spde}.

Although apparently simple, the above program is actually very challenging in practice, as the two main requirements are in competition with one another. Specifically:
\begin{itemize}
\item Pathwise uniqueness for the linear, transport-type SPDE \eqref{eq:intro-linear-spde-0} is usually based on either solving the underlying SDE, or commutator estimates, see \cite{Fla} for a detailed discussion; both arguments require $W=\cQ^{1/2} \Xi$ to be fairly regular, which can be equivalently formulated as a condition on the smoothing properties of operator $\cQ$.
Other concepts of uniqueness (called \emph{Wiener uniqueness} in \cite{Fla}), coming from the works by Le Jan and Raimond \cite{LejRai1, LejRai2} are available under milder assumptions on $\cQ$, but do not coordinate well with Girsanov transform techniques.

\item A rigorous justification of the passage from \eqref{eq:intro-abstract-spde} to \eqref{eq:intro-linear-spde-0} based on Girsanov usually requires to verify Novikov's condition, which amounts to
\begin{equation}\label{Novikov-criterion}
	\E\bigg[ \exp\bigg(\frac12 \int_0^T \big\| \cQ^{-1/2}\, \cR\, \curl^{-1} \omega_s \big\|_{L^2_x}^2\,\d s\bigg)\bigg]<+\infty.
\end{equation}
Even neglecting the exponential integrability requirement, condition \eqref{Novikov-criterion} requires at least to know that $\cQ^{-1/2}\, \cR\, \curl^{-1} \omega \in L^2_t L^2_x$.
In particular, either $\cQ^{-1/2}$ or the operator $\cR\, \curl^{-1}$ coming from the nonlinearity, should have nice behaviour.
\end{itemize}
Let us also stress that, even just to make sense of the SPDE \eqref{eq:intro-abstract-spde}, one needs the operator $\cQ$ to have certain compactness properties (see Section \ref{subsec:structure-noise}), so that $\cQ^{1/2}$ and $\cQ^{-1/2}$ cannot even be both bounded on $L^2_x$.
In general, there is a delicate balance and nontrivial trade-off between the regularity properties of the main players (the noise, the nonlinearity and the solution), which have to be verified in each case in order to carry out the strategy.
This is the reason why we are currently able to close the arguments for Systems \ref{system:logEuler} and \ref{system:hypoNS}, but not for the genuine $2$D Euler one; we leave it as an open problem, see Conjecture \ref{conj:euler} in Section \ref{sec:open}.

In order to solve the issues with the verification of Novikov's criterion \eqref{Novikov-criterion}, we will invoke a nice (yet quite unnoticed) result by Ferrario \cite{Ferrario}, which in turn is originally based on Liptser and Shiryaev \cite[Chapter 7]{LipShi}; this provides a much easier-to-check condition, not only guaranteeing the validity of Girsanov's transform, but also the equivalence (on finite intervals $[0,T]$) of the laws $\Q$ and $\P$; see Section \ref{subsec:uniqueness-nonlinear} for more details.
In particular, this implies that the laws of the unique solution $\omega$ to the nonlinear SPDE \eqref{eq:intro-abstract-spde} and of $\tilde\omega$ solution to \eqref{eq:intro-linear-spde-0} are equivalent; moreover $\omega$ and $\tilde\omega$ share the same pathwise properties, which is particularly interesting since $\tilde\omega$ actually solves a Kraichnan type SPDE, which is a basic toy model of turbulence.

The other key ingredient needed in the proof of Theorem \ref{thm:main-theorem} is the pathwise uniqueness of the linear equation \eqref{eq:intro-linear-spde-0}, under regularity conditions on $\cQ$ as mild as possible.
This is an SPDE of interest on its own, as in the case $f=0=\nu$, it is a particular (incompressible) case of the aforementioned Kraichnan model of turbulence \cite{Kraichnan} (see also Section \ref{subsec:literature} for more references).
It then makes sense to study it in higher dimension and to introduce the following auxiliary linear system on $\R^d$, $d\ge 2$ ($d=1$ being trivial for divergence-free objects):
\begin{equation}\label{eq:intro-linear-spde}
	\d \rho + \circ \d W\cdot\nabla \rho = [f - \nu\Lambda^\beta \rho]\, \d t.
\end{equation}
Here $W$ is a divergence free Gaussian noise on $\R^d$, homogeneous and isotropic, with paths in $C^0_t L^2_\loc$; see Assumption \ref{ass:covariance-basic} in Section \ref{subsec:structure-noise} below for the precise conditions.

\begin{theorem}\label{thm:main-linear-intro}
Let $p\in [1,\infty]$, $\nu \geq 0$, $\beta\in (0,2)$, $f\in L^1_t (L^1_x\cap L^p_x)$ and $W$ as described above.
Then for any deterministic $\rho_0\in L^1_x\cap L^p_x$, there exists a global probabilistically strong, analytically weak solution to \eqref{eq:intro-linear-spde}, satisfying $\P$-a.s.
\begin{equation}\label{eq:intro-linear-apriori}
	\sup_{s\in [0,t]} \| \rho_s \|_{L^1_x\cap L^p_x}
	\leq \| \rho_0\|_{L^1_x\cap L^p_x} + \int_0^t \| f_s\|_{L^1_x\cap L^p_x}\, \d s
	\quad \forall\, t\geq 0.
\end{equation}
Moreover pathwise uniqueness and uniqueness in law hold among all solutions to \eqref{eq:intro-linear-spde} satisfying $\rho\in L^1_t L^1_x$ $\P$-a.s.
Finally, given two solutions $\rho^i$, $i=1,\,2$ defined on the same probability space and driven by the same $W$, but associated to different data $(\rho_0^i, f^i)$, $\P$-a.s. it holds
\begin{equation}\label{eq:intro-linear-stability}
	\sup_{s\in [0,t]} \| \rho^1_s -\rho^2_s \|_{L^1_x\cap L^p_x}
	\leq \| \rho^1_0-\rho^2_0\|_{L^1_x\cap L^p_x} + \int_0^t \| f^1_s-f^2_s\|_{L^1_x\cap L^p_x}\, \d s
	\quad \forall\, t>0.
\end{equation}
\end{theorem}

Theorem \ref{thm:main-linear-intro} will be a subcase of Theorem \ref{thm:main-linear-extended}, which will be proved in Section \ref{subsec:uniqueness-linear}.
Assumption \ref{ass:covariance-basic} is roughly the same condition under which we are able to give meaning to the SPDE \eqref{eq:intro-linear-spde} to begin with, thus we do not expect it to be further improvable; we have not however studied the case of non incompressible $W$, which we leave for future investigations, for which such regularity may be no longer sufficient.
The proof relies on a non-standard argument, based on looking at the evolution of the mean energy spectrum $a_t(\xi):=\E[|\hat\rho(\xi)|^2]$, which can be shown to solve an auxiliary system of infinitely many coupled ODEs; notable precursors for the idea of the proof are the works \cite{BBF,Gal}, see Section \ref{subsec:literature} for more references.
In the setting of the linear SPDE \eqref{eq:intro-linear-spde}, we can even allow the forcing term $f$ to be stochastic, see Theorem \ref{thm:main-linear-extended} for the precise statement; moreover, in the proof of pathwise uniqueness, we only need the forcing $f$ to be in the class $L^1_t L^1_x$.

In this paper we always work on the whole space $\R^d$, which makes the proofs of Theorems \ref{thm:main-theorem}-\ref{thm:main-linear-intro} considerably more difficult than on the torus $\T^d$; indeed we have to handle with several technical difficulties, ranging from the singular properties of the Biot--Savart kernel, to the lack of compact embeddings between function spaces and even of Mercer type theorems for the covariance function $Q$ of the noise $W$ (which we essentially need to reprove, see Section \ref{subsec:structure-noise}).
At the same time, all the results carry over seamlessly to the torus $\T^d$; in particular, in the setting of regularized 2D Euler equation on $\T^2$, our results settle the open problem formulated by Flandoli in \cite{Fla}, see Remarks \ref{rem:readaptation-torus}-\ref{rem:flandoli} for more details on both statements.

Let us finish this section by shortly describing the structure of the rest of the paper.
We conclude the introduction by recalling some related literature in Section \ref{subsec:literature} and the relevant notations in Section \ref{subsec:notation}.
In Section \ref{sec:preliminaries}, we present some preliminary results needed in the rest of the paper; in particular, Section \ref{subsec:structure-noise} comprises a detailed description of the noise $W$ we consider, while Section \ref{subsec:solution-concepts} discusses the relevant notions of solutions for our SPDEs.
Section \ref{sec:wellposedness-linear} is devoted to the proof of Theorem \ref{thm:main-linear-intro} and consists of two parts, dealing respectively with strong existence and pathwise uniqueness of solutions.
Theorem \ref{thm:main-theorem} will be proved in Section \ref{sec:wellposedness-nonlinear}, similarly splitting the problem into treating weak existence and uniqueness in law separately.
In Section 5 we discuss the non-uniqueness result from Proposition \ref{prop:pde-nonuniquess}, and we conclude the paper with a discussion of potential future research in Section \ref{sec:open}.
Finally, we collect some auxiliary results in the Appendices \ref{app:useful}-\ref{app:pathwise-regularity}.

\subsection{Relations with the existing literature}\label{subsec:literature}

We discuss here more in details several results in the literature connected to ours; we separate them in distinct thematic blocks.

\emph{Weak solutions to $2$D Euler equations.}
DiPerna and Majda first constructed Young-type solutions to $2$D Euler equations in \cite{DiPMaj}, building on the theory introduced by the same authors in \cite{DiPMaj2}; such solutions have vorticity in $\mathcal{M}\cap \dot H^{-1}_x$, $\mathcal{M}$ denoting the space of Radon measures.
They also observed that, for $\omega_0\in L^p_x$ with $p\in (1,\infty)$, these are in fact classical weak solutions.
They can be obtained by three different relevant approximation schemes: approximation by exact smooth solutions, vanishing viscosity from $2$D Navier--Stokes, and vortex-blob approximations.
However, as their uniqueness is still open, one cannot infer in general that these schemes give rise to the same solutions.
Nonetheless, weak solutions to $2$D Euler equations keep receiving considerable attention in the literature and many physical properties have been established for them, most notably their \emph{renormalizability} in the sense of DiPerna-Lions \cite{DiPLio} and preservation of kinetic energy and enstrophy: see the work by Lopes Filho et al. \cite{LoMaNu} and then the one by Crippa et al. \cite{CriSpi,CiCrSp2}.
Finally, Ciampa, Crippa and Spirito \cite{CiCrSp1} recently established quantitative rates of convergence in the strong $L^p_x$-norms for solutions obtained by the vanishing viscosity limit of Navier--Stokes.

\emph{Stratonovich transport noise in fluids.}
An idealized description of the effects of small, possibly turbulent, fluid scales by means of a Brownian-in-time, coloured-in-space noise $W$ was first proposed by Kraichnan in the context of passive scalar turbulence, see the works \cite{Kraichnan, Kra2}, as well as the reviews \cite{CGHKP, FaGaVe}.
The noise $W$ in consideration here belongs to the class of measure preserving, isotropic Brownian flows as defined in \cite{Kunita}, which have been studied in detail by Le Jan and Raimond \cite{LejRai1, LejRai2}; the authors therein identify three relevant regimes for the existence and properties of the corresponding flow, depending on the compressibility of the noise (as encoded by some parameters in the covariance function), see \cite[p. 1312]{LejRai2} and a similar classification given by \cite[p. 639]{EVan}.
In particular, in this paper we work exclusively with divergence free noise, which is \emph{turbulent without hitting} according to the classification from \cite{LejRai2}.
Early mathematical attempts to incorporate transport noise in nonlinear fluid dynamics equations go back to Inoue and Funaki \cite{InoFun}; their idea, based on a variational principle, has been later substantially revisited by Holm \cite{Holm2015} and expanded to several classes of equations, see \cite{DrHo, CHLN} and the references therein.
A different, more Lagrangian-flavoured approach was proposed in \cite{BrCaFl}, where Brz\'ezniak et al. derived stochastic Navier--Stokes equations with transport noise by computing the stochastic material derivative along the trajectories of fluid particles.
Yet another model for turbulent fluid dynamical equations, incorporating transport noise, has been proposed by M\'emin in \cite{Memin}.
In all of the above cases, the adoption of Stratonovich noise is physically justified  by the Wong--Zakai principle, which is known to hold also for SPDEs, see e.g. \cite{BreFla}.
A heuristic argument for the introduction of transport noise, based on a separation of scales, can be found in \cite{FlaLuo}, which has been made rigorous by Flandoli and Pappalettera \cite{FP21, FP22} using methods of stochastic model reduction.

\emph{Non-uniqueness issues in fluid dynamics.}
In recent years, the convex integration tehcniques have successfully constructed non-unique weak solutions to various fluid equations, most notably 3D Euler and Navier-Stokes equations; see \cite{DeLSz13, BuckDeLeIS, BuckVic} and the references therein.
The same techniques apply to $2$D Euler in velocity form, giving rise to infinitely many wild solutions $u$ with suitable H\"older regularity, see \cite{BuckShkVic, Novack, MenSze, Mengual} for several examples. However, so far convex integration has not been able to show non-uniqueness of solutions with vorticity in $L^p_x$, the closest attempt being the aforementioned \cite{BruCol} reaching the Lorentz class $L^{1,\infty}_x$.
In another direction, these techniques have been readapted in the stochastic setting to show the similar pathological behaviour and multiplicity of solutions in the presence of additive noise \cite{HZZa,HZZb}, to the point where even non-unique ergodicity holds \cite{HZZc}, and for sufficiently regular Stratonovich transport noise \cite{HoLaPa, Pappalettera}.
It should be mentioned that convex integration techniques typically yield very weak solutions, for instance outside the Leray-Hopf class, which are thus usually considered non-physical.
In the context of forced $2$D Euler equations in vorticity form, non-uniqueness of solutions with $L^1_x\cap L^p_x$ vorticity has been established by Vishik \cite{Vishik1, Vishik2}, by applying ideas from spectral analysis and linear instability; this construction has been later revisited by De Lellis and his group in \cite{ABCDGJK} and has recently led to astonishing developments in \cite{AlBrCo}, where non-uniqueness of Leray-Hopf solutions to forced $3$D Navier-Stokes equations is shown.
An alternative approach to show nonuniqueness for the unforced $2$D Euler equation is the one by Bressan et al. \cite{BreMur, BreShe}, also borrowing ideas from the constructions of self-similar and spiral solutions by Elling \cite{Elling1, Elling2}. See also \cite[Section 1.3]{ABCDGJK} for further references on this very rich topic.

\emph{Logarithmically regularized Euler equations.}
The paper \cite{ChCoWu} introduced a general class of active inviscid transport equations, comprising both $2$D Euler and SQG equations, by allowing a general Fourier multiplier in front of the velocity term $u$ in \eqref{eq:intro-euler}.
Among them, the logEuler equations have then attracted considerable attention, due to their slightly different behaviour in the critical spaces $\dot H^1_x \cap \dot H^{-1}_x$, compared to standard Euler.
Indeed, although the space $\dot H^1_x$ has the same scaling behaviour as $L^\infty_x$, it does not embed therein and so Yudovich theory does not apply.
Chae and Wu \cite{ChaeWu} first studied the logEuler equations with $\gamma>1/2$, obtaining local well-posedness in $H^1_x$; later \cite{DongLi} established global well-posedness in the critical space $(\dot H^1\cap \dot H^{-1})(\R^2)$ for $\gamma \ge 3/2$.
On the other hand, for the $2$D Euler equations (formally corresponding to $\gamma=0$) Bourgain and Li \cite{BouLi} established strong ill-posedness in $H^1_x$, in the sense that the $H^1_x$-norm of Yudovich solutions can become instantaneously infinite at positive times.
The result was later revisited by Elgindi and Jeong \cite{ElgJeo}, who proved a similar result on $\T^2$, making the construction local.
The gap between the available results for $\gamma=0$ and $\gamma>1/2$ was recently filled by Kwon \cite{Kwon}, showing that the logEuler equation is ill-posed in $\dot H^1_x\cap \dot H^{-1}_x$ (in the sense of instantaneous infinite norm inflation as above) for $\gamma\in (0,1/2]$.

\emph{Regularization by transport noise in PDEs.}
The regularizing effect of Stratonovich transport noise in PDEs was first observed in the seminal work \cite{FlGuPr2010}, showing that simple transport noise (in the sense that $W(t,x)=W(t)$) drastically improves the solution theory of linear transport equations with H\"older continuous drifts;
several other results for linear PDEs have then been established, like prevention of blow-up by singularity formation \cite{FedFla} or infinite stretching \cite{FlMaNe}.
However, it was already noticed in \cite[Section 6.2]{FlGuPr2010} that simple noise is not as useful in regularizing nonlinear PDEs, where noises with richer space-dependence should be adopted.
This is now also confirmed by the negative results of \cite{HoLaPa, Pappalettera}, in which smooth transport noise can be reabsorbed by a flow transformation; the resulting random PDE is then amenable to the application of convex integration schemes, yielding the same pathological behaviour as in the deterministic setting.
Concerning positive results, \cite{FlGuPr2011} proved that a sufficiently nondegenerate space-time noise prevents the collapse of point vortex system of 2D Euler equations; see \cite{DeFlVi14} for related results on system of Vlasov-Poisson point charges and \cite{LuoSaa} for the point vortex system of modified SQG equations.
In another direction, we have shown in the recent works \cite{FlaLuo, FlGaLu21} that high mode transport noise can suppress blow-up of solutions with high probability; they cover several PDEs of interest, most notably $3$D Navier-Stokes, Keller-Segel and $2$D Kuramoto-Sivashinsky equations on the torus.
The results are based on a scaling limit technique introduced by the first author in \cite{Gal} and have been recently applied also to Tao's averaged version of Navier-Stokes equations \cite{Lange} and reaction-diffusion systems \cite{Agresti}.
The same scaling limit argument applies to stochastic $2$D inviscid fluids, see \cite{FlGaLu21a,Luo21}, although therein it cannot be exploited to infer any regularization by noise property.
Another class of positive results, similar in spirit to the ones in this paper, is based on the use of Girsanov transform to reduce the nonlinear SPDE to a simpler one.
The closest to our setting, especially in the case of the logEuler equations, is the work \cite{BBF} by Barbato et al., where uniqueness in law for a $3$D Leray-$\alpha$ model is shown;
the proof of pathwise uniqueness for the linear SPDE presented therein, based on the analysis of the energy spectrum, has then been applied in the case of general scalar transport equations in \cite[Section 4.2]{Gal}.
Other applications in the literature, outside the realm of SPDEs, mainly concern stochastic dyadic model of turbulence, see \cite{BFM} and \cite{Bianchi} for the more complicated tree model;
remarkably, in this case even anomalous dissipation of energy can be established, see \cite{BFM11}.
Finally, inspired by \cite{Gal, FlGaLu21a, Luo21}, the second author and Wang \cite{LuoWang} considered a stochastic inviscid dyadic model with more complicated random perturbation than that in \cite{BFM}; the new noise ensures also uniqueness in law of weak solutions and, under a suitable scaling limit, convergence to the associated viscous model.

\subsection{Notations and conventions}\label{subsec:notation}

We write $x\cdot y$ for the scalar product of $x,y\in \R^d$, $|x|$ for the Euclidean norm.
The notation $a\lesssim b$ means that there exists some constant $c>0$ such that $a\leq c b$;
$a\sim b$ if $a \lesssim b$ and $b\lesssim a$.
Given some family of parameters $\lambda$, we write $a\lesssim_\lambda b$ to stress the dependence of the hidden constant $c=c(\lambda)$.

We denote by $C^\infty_b(\R^d;\R^m)$ the space of bounded, infinitely differentiable functions $f:\R^d\to\R^m$ with bounded derivatives of all orders; whenever $d$, $m$ are clear from the context, we will just write $C^\infty_b$ for short.
Similarly, $C^\infty_c(\R^d;\R^m)=C^\infty_c$ is the space of smooth, compactly supported functions and $\cS(\R^d;\R^m)=\cS$ that of Schwartz functions; $\cS'$ denotes its dual, the space of tempered distributions.
For $p\in [1,\infty]$, $L^p(\R^d;\R^m)=L^p_x$ denote the usual Lebesgue spaces; given $p$, we denote by $p'=p/(p-1)$ its conjugate exponent, with the usual conventions for $p\in \{1,\infty\}$.
For $s\in \R$, $\cC^s_x=\cC^s_x(\R^d;\R^m)=B^s_{\infty,\infty}(\R^d;\R^m)$ stand for the classical inhomogeneous Besov--H\"older spaces, while for $n\in\N$ and $p\in [1,\infty]$, $W^{n,p}_x=W^{n,p}(\R^d;\R^m)$ denote classical Sobolev spaces.

We adopt the convention that the Fourier transform of $f$ is defined as
\begin{equation*}
	(\cF f)(\xi) = \hat{f}(\xi)
	:= (2\pi)^{-d/2} \int f(x) e^{-ix\cdot \xi}\, \d x
\end{equation*}
In this way, for real functions $f$, the antitransform is given by
$(\cF^{-1} f)(x) = \check f(x) = \overline{\hat{f}(x)}$
and standard properties like Parseval identity read as
\begin{equation}\label{eq:properties-fourier}
	\widehat{f\ast g} = (2\pi)^{d/2} \hat{f}\, \hat{g},\quad
	\| \hat{f} \|_{L^2} = \| f\|_{L^2},
\end{equation}
similarly for the antitransform. Among other things, this implies
$\| f\ast g\|_{L^2} = (2\pi)^{d/2} \| \hat{f} \hat g\|_{L^2}$.

Given a sufficiently regular function $r:\R^d\to\R$, the associated Fourier symbol, or Fourier multiplier, is the operator $\mathcal{R}$ given by
$\mathcal{R} \varphi = \cF^{-1}( r \cF\varphi)$.
A protoypical example are the fractional operators $(1-\Delta)^{s/2}$, corresponding to $r(\xi)=(1+|\xi|^2)^{s/2}$, for $s\in\R$.
The associated inhomogeneous fractional Sobolev spaces are given by $H^s_x := (1-\Delta)^{-s/2}(L^2_x)$, consisting of functions $f:\R^d \to \R^m$ such that
\begin{equation}\label{eq:sobolev-norm}
	\|f \|_{H^s_x} := \bigg(\int_{\R^d} (1+|\xi|^2)^{s} |\hat f(\xi)|^2\,\d\xi \bigg)^{1/2} < \infty.
\end{equation}
We use $\dot H^s_x= (-\Delta)^{-s/2} (L^2_x)$ to denote their homogeneous counterparts, with $(1-\Delta)^{-s/2}$ replaced by $(-\Delta)^{-s/2}$; although the latter are not necessarily Banach spaces, we will keep using the norm notation $\| \cdot\|_{\dot H^s}$, which is defined similarly to \eqref{eq:sobolev-norm}.

With a slight abuse, whenever clear, we use $\langle \cdot,\cdot\rangle$ both to indicate the dual pairing between any of the spaces introduced above (e.g. between $\cS$ and $\cS'$, or $L^p_x$ and $L^{p'}_x$, or $H^s_x$ and $H^{-s}_x$), keeping as a reference the duality of $L^2_x$ with itself.
Whenever $E$ and $F$ are function spaces in the list above, their intersection $E\cap F$ will be endowed with $\| \cdot\|_{E\cap F}:= \| \cdot\|_E+\| \cdot\|_F$, which makes it a Banach space if $E$ and $F$ are so (but we will also allow $F=\dot{H}^s$).

It is worth giving several details on two classes of Fourier multipliers which will be used frequently in the paper.

The first one is the fractional Laplacian $\Lambda^\beta=(-\Delta)^{\beta/2}$, for $\beta\in (0,2)$.
Besides being the Fourier multiplier associated to $\xi\mapsto |\xi|^\beta$, it has an integral representation in the real space variable: there exists a constant $C_{d,\beta}$ such that
\begin{equation}\label{eq:fractional-laplacian1}
	\Lambda^\beta f (x)
	= C_{d,\beta}\, {\rm P.v.} \int_{\R^d} \frac{f(x)-f(y)}{|x-y|^{d+\beta}}\, \d y
\end{equation}
whenever $f$ is smooth enough; as a consequence, for any $f$, $g$ regular enough, it holds
\begin{equation}\label{eq:fractional-laplacian2}
	\langle \Lambda^\beta f, g\rangle
	= C_{d,\beta} \int_{\R^d\times \R^d} \frac{(f(x)-f(y)) (g(x)-g(y))}{|x-y|^{d+\beta}}\, \d x \d y.
\end{equation}
$\Lambda^\beta$ are fairly well-behaved linear operators, e.g. they map $C^\infty_b$ into itself, $\cC^{s+\beta}_x$ (resp. $H^{s+\beta}_x$) into $\cC^s_x$ (resp. $H^s_x$) for any $s\in\R$.

The second one is the Biot--Savart operator $\curl^{-1}$, which corresponds on $\R^2$ to the Fourier symbol $-\nabla^\perp(-\Delta)^{-1}$;
it admits similarly a real space representation as $\curl^{-1} f = K\ast f$, where $K(x)=C_2\, x^\perp/|x|^2$, for a suitable constant $C_2$.
Here $x^\perp:=(-x_2,x_1)^T$, namely the counterclockwise 90 degrees rotation of $x$, similarly for $\nabla^\perp$.
We warn the reader that on the full space $\R^2$, compared to bounded domains such as $\T^2$, this operator presents some nontrivial technical issues, due to the spectrum being continuous and $\hat K$ (resp. $\xi^\perp/|\xi|^2$) being singular in $0$.
Specifically, $\curl^{-1}$ maps $\dot{H}^{-1}_x$ in $L^2_x$ and $L^1_x\cap L^{2+\eps}_x$ in $L^\infty_x$ for all $\eps>0$, but in general is not well-defined from $L^2_x$ to itself; see Remarks 1.0.2-1.0.5 from \cite{ABCDGJK} for a deeper discussion.

We might sometimes need to work with local function spaces, without knowledge on the behaviour of the function at infinity.
For instance, we denote by $L^p_\loc=L^p_\loc(\R^d;\R^m)$ the space of functions $f:\R^d\to\R^m$ such that $f\,\varphi\in L^p_x$ for all $\varphi\in C^\infty_c$;
similarly for $C^\infty_\loc$, $H^s_\loc$ and so on.
Such spaces can be endowed with a countable family of seminorms, yielding a metric and a Fr\'echet topology, see Appendix \ref{app:useful} for such a construction for $H^s_\loc$.

Whenever dealing with functions defined w.r.t. a time parameter $t\in\R_{\geq 0}=[0,+\infty)$, for simplicity the associated norms will always be computed only on finite intervals $[0,T]$, although $T$ can be chosen arbitrarily large.
For example, given a Banach space $E$ and $q\in [1,\infty]$ we denote by $L^q_t E = L^q(0,T;E)$ the space of Bochner-measurable function $f:[0,T]\to E$ such that
\begin{equation*}
	\| f\|_{L^q_t E}:= \Big( \int_0^T \| f_t\|_E^q\, \d t\Big)^{1/q}<\infty
\end{equation*}
with the usual convention on the essential supremum norm for $q=\infty$.
With a slight abuse, we extend the definition to other ``norm-like'' quantities such as $\| \cdot\|_{\dot{H}_x}$.
For more general metric vector spaces $(E,d_E)$ we denote by $C^0_t E=C([0,T];E)$ the space of continuous, $E$-valued functions, with supremum distance $\| f\|_{C^0_t E}:=\sup_{t\in [0,T]} d_E( f_t, 0)$;
for $\gamma\in (0,1)$, $C^\gamma_t E=C^\gamma([0,T];E)$ instead is the class of $\gamma$-H\"older continuous functions, with
\begin{equation*}
	\llbracket f \rrbracket_{C^\gamma_t E}
	:= \sup_{0\leq s<t\leq T} \frac{d_E(f_t,f_s)}{|t-s|^\gamma}, \quad
	\| f\|_{C^\gamma_t E}:= \| f\|_{C^0_t E} + \llbracket f \rrbracket_{C^\gamma_t E}.
\end{equation*}

Whenever dealing with a filtered probability space $(\Theta,\mathbb F, \mathbb F_t, \P)$, we will assume that $\F$ and $\F_t$ satisfy the usual assumptions.
Given  metric vector space $E$ and $p\in [1,\infty]$, an $E$-valued random variable $Y$ belongs to $L^p_\Theta E= L^p(\Theta,\F,\P;E)$ if $d_E(Y,0)$ belongs to the classical $L^p$-space; $L^p_\Theta E$ is a complete metric space with distance
\begin{equation*}
	d_{L^p_\Theta E} (Y,Z)=\E[d_E(Y,Z)^p]^{1/p}
\end{equation*}
where $\E$ denotes expectation w.r.t. $\P$; it is Banach if $E$ is Banach.
An $E$-valued stochastic process $X$ is $\F_t$-adapted if $X_t$ is $\F_t$-measurable for any $t\geq 0$;
it is progressively measurable if for any $t>0$, the restriction $X:\Theta\times [0,t]\to E$ is measurable w.r.t. $(\mathbb F_t\otimes \mathcal B([0,t]); \mathcal B(E))$.
Whenever $E$ embeds into $\cS'$, we will say that $X$ is $\P$-a.s. weakly continuous if for any $\varphi \in C_c^\infty$, the real-valued process $t\to \<X_t, \varphi\>$ has $\P$-a.s. continuous paths; we warn the reader to keep in mind that, for Banach space $E$, this does not imply that $X$ is $\P$-a.s. continuous w.r.t. the weak topology induced by the topological dual $E'$.

All the previous definitions of metric and Banach spaces can be concatenated, with the convention of reading the expression from left to right.
For instance one can define $L^q_t(L^1_x\cap L^p_x)=L^q([0,T]; L^1\cap L^p(\R^d;\R^m))$, similarly $L^k_\Theta C^\gamma_t H^s_\loc$, $C^0_t L^k_\Theta \dot H^{-1}_x$ and so on.

Finally, whenever referring to a smooth deterministic function $f$, say e.g. $f:\R_{\geq 0} \times \R^d\to \R^m$, we will mean that $f$ has all the desired boundedness, differentiability or local support properties needed, in any of its variables.
Similarly, whenever talking about a smooth process $X$, say $X:\Theta\times\R_+\times\R^d\to\R^m$, we will impose that it has similar properties in the variables $(t,x)$, while possibly being uniformly bounded in $\Theta$ in all relevant norms.

\medskip

\noindent \textbf{Acknowledgements.}
LG is supported by the SNSF Grant 182565 and by the Swiss State Secretariat for Education, Research and Innovation (SERI) under contract number MB22.00034. DL is grateful to the financial supports of the National Key R\&D Program of China (No. 2020YFA0712700), the National Natural Science Foundation of China (Nos. 11931004, 12090010, 12090014), and the Youth Innovation Promotion Association, CAS (Y2021002).

LG will be constantly indebted to David Barbato for making him interested in this line of research; part of this project dates back to LG's master thesis.
Both authors are thankful to Franco Flandoli for his generosity, constant supervision throughout the years and many stimulating discussions.
They also warmly thank Maria Colombo for very useful discussions and clarifications on the results from \cite{AlbCol, ABCDGJK}, Michele Dolce for the content of Remark \ref{rem:michele} and Francesco de Vecchi for some helpful suggestions concerning Appendix \ref{app:pathwise-regularity}.

\section{Preliminaries}\label{sec:preliminaries}

This section consists of two parts. We first describe in detail the structure of the noise $W$, starting from the Fourier transform of its covariance function $Q$; we shall find series expansions of $Q$ and $W$, and discuss the relation between their regularity properties, which will play important role in the sequel. We stress that many of these results are classical on compact domains, but we have not found proper references on the full space. In the second part, as our noise is rough in space variable, we explain the connection between Stratonovich and It\^o formulations, and consequently introduce the relevant solution concepts for the SPDEs studied in this paper.

\subsection{Structure of the noise}\label{subsec:structure-noise}

Since we will only consider divergence free $W$, we set ourselves on $\R^d$ with $d\geq 2$.

Let $(\Theta,\mathbb F, \P)$ be a filtered probability space, on which we are given a centered Gaussian vector field $W=W(t,x)$, possibly taking values in spaces of distributions;
we refer to the monographs \cite{Bog,DaPZab} for standard facts on these objects.
We will assume $W$ to be Brownian in time, coloured and divergence-free in space;
we additionally impose it to be \emph{homogeneous} in space, so that its law is entirely determined by the associated covariance function $Q:\R^d \to \R^{d\times d}$, formally given by
\begin{equation*}
	\E[W(t,x)\otimes W(s,y)]= (t\wedge s) Q(x-y)
	\quad \forall\, t,s\in \R_{\geq 0},\, x,y\in \R^d.
\end{equation*}
Throughout the paper, we will always impose the following structural condition on $Q$; it encompasses a general class of noises, including the incompressible Kraichnan noise of turbulence (corresponding to $g(\xi)=(1+|\xi|^2)^{-(d+\gamma)/2}$ for some $\gamma\in (0,2)$).

\begin{assumption}\label{ass:covariance-basic}
The covariance $Q$ has Fourier transform $\hat{Q}$ given by
\begin{equation}\label{eq:covariance-basic}
	\hat Q (\xi)
	:=  g(\xi) P_\xi
	= g(\xi) \bigg(I_d- \frac{\xi\otimes \xi}{|\xi|^2} \bigg)
\end{equation}
for a continuous, non-negative radial function $g(\xi)=G(|\xi|)$ satisfying $g\in (L^1\cap L^\infty) (\R^d)$.
\end{assumption}

The matrix $P_\xi$ appearing on the r.h.s. of \eqref{eq:covariance-basic} is the projection on the orthogonal of $\xi$ (with the convention that $P_0=0$), which enforces the fact that the noise $W$ is divergence free.
Instead, the radiality of $g$ enforces the \emph{isotropy} of the noise, which can be equivalently expressed by the property that, for any rotation matrix $R\in \R^{d\times d}$, it holds
\begin{equation}\label{eq:isotropy-covariance}
	R\, Q(x) R^T = Q(R x), \quad R\, \hat{Q}(\xi) R^T = \hat Q(R \xi).
\end{equation}
Also observe that, since $g$ is real-valued, $Q$ is symmetric, namely $Q(x)=Q(-x)$.

It follows from Assumption \ref{ass:covariance-basic}, together with well-known properties of the Fourier antitransform of $L^1_x \cap L^2_x$ functions, that $Q$ is bounded, uniformly continuous and in $L^2_x$.
The covariance of $W$ can be alternatively thought of as a symmetric operator $\cQ$, acting on smooth $\R^d$-valued functions $f$ by
\begin{equation*}
	(\cQ f)(x) := (Q\ast f)(x)= \int_{\R^d} Q(x-y) f(y)\,\d y.
\end{equation*}
The operator $\cQ$ can be equivalently defined as the Fourier multiplier $\cQ v= (2\pi)^{d/2} \mathcal{F}^{-1}(\hat Q \hat v)$ (the additional factor $(2\pi)^{d/2}$ being due to our conventions for $\cF$, cf. the properties \eqref{eq:properties-fourier}); in this case, it can be thought of as the composition of two multipliers, one corresponding to the scalar $(2\pi)^{d/2} g(\xi)$, the other to the matrix $P_\xi$, the latter being the classical Leray-Helmholtz projector $\Pi$. In other words
\begin{equation}\label{eq:covariance-projector}
	\cQ = \cF^{-1}\Big( (2\pi)^{d/2} g\ \hat\cdot\Big) \circ \Pi.
\end{equation}
In the following we will often say that $W$ is a $\cQ$-Brownian motion (resp. $Q$-Brownian motion), to mean that it has covariance operator $\cQ$ (resp. covariance kernel $Q$) satisfying Assumption \ref{ass:covariance-basic}.
There is no ambiguity in this, since $Q$ and $\cQ$ are in a 1-1 correspondence.

In view of Assumption \ref{ass:covariance-basic}, the square root operator $\cQ^{1/2}$ is also well defined, with associated Fourier multiplier $(2\pi)^{d/4} g^{1/2}(\xi) P_\xi$, which now belongs to $L^2_x\cap L^\infty_x$.
Antitransforming again, this implies that $\cQ^{1/2}$ is also a convolutional operator, associated to an $L^2_x$-integrable function $\mathfrak{g}= (2\pi)^{-d/2} \mathcal{F}^{-1} ( (2\pi)^{d/4} g^{1/2} P_\xi) = (2\pi)^{-d/4} \mathcal{F}^{-1} ( g^{1/2} P_\xi)$.

The next lemma collects some basic properties of the operators $\cQ^{1/2}$.
\begin{lemma}\label{lem:basic-properties-covariance}
The following assertions hold:
\begin{itemize}
\item[i)] $\cQ^{1/2}$ maps $L^2_x$ fields into uniformly continuous, bounded, divergence free fields and
\begin{align*}
	\| \cQ^{1/2} f \|_{L^\infty} \lesssim \| f\|_{L^2} \| g\|_{L^1}^{1/2};
\end{align*}
\item[ii)] $\cQ^{1/2}$ maps $\R^d$-valued finite measures into divergence free fields in $L^2_x$ and
\begin{align*}
	\| \cQ^{1/2} \mu \|_{L^2} \lesssim \| \mu\|_{TV} \| g\|_{L^1}^{1/2},
\end{align*}
where $\| \mu\|_{TV}$ denotes the total variation norm of $\mu$.
\end{itemize}
\end{lemma}

\begin{proof}
Both statements follow from the representation of $\cQ^{1/2}$ via the convolution with $\mathfrak{g}$, which is an $L^2_x$ matrix-valued function with $\| \mathfrak{g} \|_{L^2_x} = \|\hat{\mathfrak{g}} \|_{L^2_x} \sim \| g\|_{L^1_x}^{1/2}$, and standard properties of convolution via Young inequalities.
\end{proof}

As standard when dealing with Gaussian variables in infinite dimensional spaces, we define the associated Cameron-Martin space as $\cH:=\cQ^{1/2}(L^2(\R^d;\R^d))$; it is a Hilbert space endowed with the scalar product
\begin{align*}
	\langle \varphi,\psi\rangle_\cH := \langle \cQ^{-1/2} \varphi, \cQ^{-1/2} \psi \rangle_{L^2_x}.
\end{align*}
Observe that, by Lemma \ref{lem:basic-properties-covariance}, $\cH$ is a space of continuous, bounded vector fields and that convergence in $\cH$ implies convergence in the uniform topology. Furthermore, combining items i) and ii) of Lemma \ref{lem:basic-properties-covariance}, we can deduce that $\cQ \mu \in \cH$ for any finite vector-valued measure $\mu$. In particular, for any $y\in \R^d$, the map
\begin{align*}
	Q_y (x) :=Q(x-y) = (\cQ \delta_y)(x)
\end{align*}
is an element of $\cH$,\footnote{Technically, since $\delta_y$ in only scalar valued, $Q_y$ is $\R^{d\times d}$-valued, while elements of $\cH$ are $\R^d$-valued. More precisely $Q_y$ can be identified with $d$ distinct elements of $\cH$, given by $Q_y^i= \cQ (\delta_y e_i)$, where $(e_1,\ldots, e_d)$ is the standard basis of $\R^d$. With a slight abuse, we will write $Q_y\in \cH$ whenever the interpretation is clear.}
with the property that for any $\varphi\in \cH$ it holds
\begin{equation}\label{eq:property-Q_y}
	\langle Q_y, \varphi\rangle_{\cH}
	= \langle \cQ^{1/2}(\delta_y), \cQ^{-1/2}\varphi\rangle_{L^2}
	= \langle \delta_y, \varphi\rangle = \varphi(y)
\end{equation}
where we used the facts that $\cQ^{1/2}$ is symmetric and $\cH\hookrightarrow C_b$, so that all pairings are rigorous.
With these preparations, we can prove rigorously a tensorized representation of $Q$, which in the setting of integral operators on compact domains is standardly associated to Mercer's theorem \cite{Mercer}.

\begin{lemma}\label{lem:covariance-series-representation}
Let $\{\sigma_k\}_{k\in\N}$ be any complete orthonormal system (CONS) of $\cH$ made of smooth, divergence free functions. Then it holds
\begin{equation}\label{eq:covariance-series-representation}
	Q(x-y)=\sum_{k\in\N} \sigma_k(x) \otimes \sigma_k(y)
\end{equation}
where the series converges absolutely, uniformly on compact sets.
\end{lemma}

\begin{proof}
Since $\{\sigma_k\}_k$ is a CONS of $\cH$ and $Q_y\in\cH$, for any fixed $y\in\R^d$ it holds
\begin{align*}
	Q_y(\cdot)
	= \sum_{k\in\N} \langle Q_y, \sigma_k \rangle_{\cH}\, \sigma_k (\cdot)
	= \sum_{k\in\N} \sigma_k(y) \otimes \sigma_k(\cdot)
\end{align*}
where the series converges in $\cH$ and thus in the uniform topology; we can then evaluate in $x$, which together with the definition of $Q_y$ yields the pointwise relation \eqref{eq:covariance-series-representation}.
To upgrade the convergence to be also locally uniform in the variable $x$, first observe that by taking $x=y$ in \eqref{eq:covariance-series-representation} we find
\begin{align*}
	\Tr (Q(0))
	= \sum_k \Tr (\sigma_k (x) \otimes \sigma_k(x))
	= \sum_k |\sigma_k(x)|^2 \quad \forall\, x\in \R^d.
\end{align*}
This implies that the increasing sequence $S_n(x) = \sum_{k\leq n} |\sigma_k(x)|^2$ is converging pointwise to the constant $S_\infty = \Tr (Q(0) )$, which by Dini's theorem implies that the sequence is also converging uniformly on compact sets.
For any $|x|,\,|y|\leq R$ it then holds
\begin{align*}
	\sum_{k\geq n} |\sigma_k(x)\otimes \sigma_k(y)|
	\leq \bigg(\sum_{k\geq n} |\sigma_k(x)|^2\bigg)^{1/2}\, \bigg( \sum_{k\geq n} |\sigma_k(y)|^2 \bigg)^{1/2}
	\leq \sup_{|z|\leq R}  |S_\infty - S_n(z)|
\end{align*}
which finally implies absolute convergence for \eqref{eq:covariance-series-representation}, uniformly on compact sets.
\end{proof}
Thanks to the isotropy property \eqref{eq:isotropy-covariance}, we can compute $Q(0)$ explicitly.

\begin{lemma}\label{lem:Q_0}
Under Assumption \ref{ass:covariance-basic}, it holds
\begin{equation}\label{eq:Q_0}
	Q(0)=2\kappa I_d, \quad \kappa:= (2\pi)^{-\frac{d}{2}}\, \frac{d-1}{2d}\, \| g\|_{L^1};
\end{equation}
moreover $Q$ satisfies $Q(x) \leq Q(0)$, in the sense of symmetric matrices, for all $x\in\R^d$.
\end{lemma}

\begin{proof}
Applying \eqref{eq:isotropy-covariance} for $x=0$, it holds $R Q(0) R^T = Q(0)$, so that $Q(0) = 2\kappa I_d$ for some $\kappa\in \R$.
The exact value can be computed by taking the trace on both sides, which gives
\begin{align*}
	\kappa = \frac{1}{2d} \Tr (Q(0)))
	=  (2\pi)^{-\frac{d}{2}}\, \frac{1}{2d}\, \Tr \bigg( \int_{\R^d} g(\xi) P_\xi\, \d \xi\bigg)
	= (2\pi)^{-\frac{d}{2}}\, \frac{d-1}{2d} \int_{\R^d} g(\xi)\, \d \xi
\end{align*}
where we used the fact that $\Tr (P_\xi)=d-1$ for all $\xi\in\R^d \setminus \{0\}$; \eqref{eq:Q_0} thus follows.
The second claim comes from \eqref{eq:covariance-series-representation} and the inequality
$2 \sigma(x)\otimes \sigma(y) \leq \sigma(x)\otimes \sigma(x) + \sigma(y)\otimes \sigma(y)$.
\end{proof}
The importance of formula \eqref{eq:covariance-series-representation} comes from the fact that correspondingly, for any such CONS $\{\sigma_k\}_k$ of $\cH$, the noise $W$ admits a series representation as
\begin{equation}\label{eq:noise-series-representation}
	W(t,x) = \sum_{k\in\N} \sigma_k(x) B^k_t,
\end{equation}
where $\{B^k \}_k$ is a family of independent standard Brownian motions defined by
\begin{equation}\label{eq:chaos-expansion-finite-noises}
	B^k_t := \frac{\langle W_t, \cQ^{-1/2} \sigma_k\rangle}{\|\sigma_k \|_{L^2_x}}.
\end{equation}

The next proposition collects some important information on the spatial regularity of $W$, in terms of sufficient conditions on its covariance $Q$, and justifies the representation \eqref{eq:noise-series-representation};
since the proofs are a bit technical and fairly standard, we postpone them to Appendix \ref{app:pathwise-regularity}.
Let us stress that the regime $\alpha\in (0,2)$ in Point c) below is exactly the one associated to the Kraichnan noise.

\begin{proposition}\label{prop:pathwise-regularity-noise}
Let $Q$ be satisfying Assumption \ref{ass:covariance-basic}, $\{\sigma_k\}_k$ be an associated CONS of $\cH$ and $W$ be a $\cQ$-Brownian motion.
The following hold:
\begin{itemize}
\item[a)] $\P$-a.s. $W\in C^0_t L^2_\loc$ and the series \eqref{eq:noise-series-representation} converges $\P$-a.s. therein;
\item[b)] if additionally there exists $\gamma>3/4$ such that
\begin{equation}\label{eq:noise-logarithmic-regularity}
	g(\xi) \lesssim (1+|\xi|)^{-d} \log^{-2\gamma}(e+ |\xi|)
	\quad \forall\, \xi\in\R^d,
\end{equation}
then $\P$-a.s. $W\in C^0_t C^0_\loc$ and the series \eqref{eq:noise-series-representation} converges $\P$-a.s. therein;
\item[c)] finally, if there exists $\alpha\in (0,1)$ such that
\begin{equation}\label{eq:noise-holder-regularity}
	g(\xi) \lesssim (1+|\xi|)^{-d-2\alpha}
	\quad \forall\, \xi\in\R^d ,
\end{equation}
then $\P$-a.s. $W\in C^0_t C^{\alpha-\eps}_\loc$ for all $\eps>0$ and the series \eqref{eq:noise-series-representation} converges $\P$-a.s. therein.
\end{itemize}
\end{proposition}

The representations \eqref{eq:covariance-series-representation} and \eqref{eq:noise-series-representation}, sometimes referred to as the \emph{chaos expansion} associated to $W$, are extremely useful for practical computations and we will invoke them quite often throughout the paper.

All the noises appearing in our main results will be spatially rough, satisfying some of the conditions from Proposition \ref{prop:pathwise-regularity-noise} at best.
However, in order to rigorously set up a priori estimates and construct solutions, we will often first work with smooth noise $W$.
To this end, it is desirable to have at hand the following result, relating the spatial regularity of $Q$ to summability of (derivatives of) $\{\sigma_k\}_k$ and spatial regularity of $W$.

\begin{proposition}\label{prop:regularity-Q-sigma}
Let $Q$ satisfy Assumption \ref{ass:covariance-basic} and $\{\sigma_k\}_k$ be a CONS of $\cH$ made of smooth functions.
Then for any $n\in\N$ it holds
\begin{align*}
	Q\in W^{2n,\infty}_x \quad \Longleftrightarrow \quad \sup_{x\in\R^d} \sum_{k\in\N} |D^n_x \sigma_k(x)|^2<\infty;
\end{align*}
moreover under the above condition, $\P$-a.s. $W\in C^0_t C^{n-\eps}_\loc$ for all $\eps>0$.
\end{proposition}

\begin{proof}
Let us first prove the statement for $n=1$; to this end, let us employ the shortcut notation $a^{\otimes 2}=a\otimes a$ for $a\in\R^d$.

First assume $Q\in W^{2,\infty}_x$, then for any $x,\, y\in\R^d$ it holds
\begin{align*}
	\sum_k |\sigma_k(x)-\sigma_k(y)|^2
	& = \Tr  \bigg( \sum_k (\sigma_k(x)-\sigma_k(y))^{\otimes 2}\bigg)\\
	& = \Tr \bigg(\sum_k \big[ \sigma_k(x)^{\otimes 2} + \sigma_k(y)^{\otimes 2} - 2 \sigma_k(x)\otimes \sigma_k(y)\big] \bigg)\\
	& = 2 \big( \Tr Q(0) - \Tr Q(x-y) \big)
\end{align*}
where the last passage comes from \eqref{eq:covariance-series-representation}.
Since $Q$ is even, $D_x Q(0)=0$;
therefore under the assumption that $Q\in W^{2,\infty}_x$, by Taylor expansion it holds $|Q(0)-Q(z)|\lesssim |z|^2$ for all $z\in \R^d$.
Overall we find
\begin{align*}
	\sup_{x\neq y} \sum_k \frac{|\sigma_k(x)-\sigma_k(y)|^2}{|x-y|^2} \lesssim 1
\end{align*}
and the conclusion now follows by taking $x= y + \eps e_i$, where $\{e_i\}_{i=1}^d$ are the canonical elements of $\R^d$, and taking $\eps\to 0$.

Conversely, if the family $\{D_x \sigma_k\}_k$ satisfies the aforementioned summability, then for any $i,j\in \{1, \ldots, d\}$ it holds
\begin{align*}
	\partial_{ij}^2 Q(x-y)
	= - \partial_{x_i} \partial_{y_j} Q(x-y)
	= - \sum_{k\in\N} \partial_i\sigma_k(x)\otimes \partial_j\sigma_k(y)
\end{align*}
so that by Cauchy's inequality
\begin{align*}
	\sup_{x,y} |\partial^2_{ij} Q(x-y)|
	& \leq \sup_{x,y} \sum_{k\in\N} |\partial_i \sigma_k(x)|\, |\partial_j \sigma_k(y)|\\
	& \leq \sup_x \bigg( \sum_{k\in\N} |\partial_i \sigma_k(x)|^2\bigg)^{1/2}\,
	 \sup_y \bigg( \sum_{k\in\N} |\partial_j \sigma_k(y)|^2\bigg)^{1/2}
\end{align*}
which implies uniform boundedness of $D^2_x Q$.

The claim on the regularity of $W$ then follows from arguing similarly to the proof of Point c) of Proposition \ref{prop:pathwise-regularity-noise}, since it holds
\begin{equation*}
	\E	\big[ |W_t(x)-W_t(y)|^2 \big] = 2t \big[ \Tr Q(0)-\Tr Q(x-y)\big]\lesssim t |x-y|^2.
\end{equation*}
The case of general $n\in\N$ follows from an iteration argument. Indeed, observe that if $Q$ is the covariance function associated to $W$, then $\tilde Q = -\partial_i^2 Q$ is the covariance function associated to $\tilde W = \partial_i W$ and moreover
\begin{align*}
	\tilde Q(x-y) = -\partial_{x_i}\partial_{y_i} Q(x-y)
	= \sum_k \partial_i\sigma_k(x)\otimes \partial_i \sigma_k(y)
	=: \sum_k \tilde\sigma_k(x)\otimes \tilde\sigma_k(y);
\end{align*}
therefore applying the statement with $(Q,\{\sigma_k\},W)$ replaced by $(\tilde Q, \{\tilde \sigma_k\}, \tilde W)$ we can inductively pass from $n$ to $n+1$ and conclude the proof.
\end{proof}

\begin{remark}\label{rem:smooth-noise}
The proof above is inspired by \cite[Remark 4]{CogFla}, which already covered one implication for $n=1$.
Based on Proposition \ref{prop:regularity-Q-sigma}, we will say that $W$ is a smooth noise if the associated covariance $Q\in C^\infty_b$, so that the statement applies for all $n\in\N$.

A sufficient condition, expressed in terms of the associated $g$ given by \eqref{eq:covariance-basic}, is that
\begin{equation*}
	\int_{\R^d} g(\xi) (1+|\xi|^m)\, \d \xi<\infty \quad \forall\, m\in\N
\end{equation*}
as can be readily checked using properties of the Fourier transform.
\end{remark}

Suppose now we are given a filtered probability space $(\Theta,\F,\F_t,\P)$ and $W$ defined on it;
we will say that $W$ is an $\F_t$-Brownian motion with covariance $Q$ if, in addition to the above requirements, it is $\F_t$-adapted and $W_t-W_s$ is independent of $\F_s$ for all $s<t$.
We conclude this section with some basic facts concerning stochastic integration w.r.t. $W$.

\begin{lemma}\label{lem:stoch-integr-basic}
For any $\mathbb F_t$-predictable, smooth process $f:\Omega\times \R_{\geq 0}\times \R^d\to \R^d$, the stochastic integral $M_t := \int_0^t \langle f_s, \d W_s\rangle$ is a well-defined continuous $\R$-valued $\mathbb F_t$-martingale satisfying
\begin{equation*}
	[M]_t = \int_0^t \| \cQ^{1/2} f_s\|_{L^2_x}^2\, \d s, \quad
	\E\bigg[ \sup_{t\in [0,T]} |M_t|^p \bigg]
	\lesssim_p   \E\bigg[ \Big( \int_0^T \|  \cQ^{1/2} f_s\|_{L^2_x}^2\, \d s \Big)^{\frac{p}{2}} \bigg]
\end{equation*}
for any $p\in [1,\infty)$. Similarly, $N_t :=\int_0^t f_s\cdot \d W_s$ is a well-defined continuous $L^2_x$-valued $\mathbb F_t$-martingale satisfying
\begin{equation*}
	[N]_t = 2\kappa \int_0^t \| f_s\|_{L^2_x}^2\, \d s, \quad
	\E\bigg[ \sup_{t\in [0,T]} \|N_t \|_{L^2_x}^p \bigg]
	\lesssim_p  \kappa^{\frac{p}{2}}\, \E\bigg[ \Big( \int_0^T \| f_s\|_{L^2_x}^2\, \d s \Big)^{\frac{p}{2}} \bigg].
\end{equation*}
By standard density and localization arguments, the definition of $M$ (resp. $N$) as a continuous $\R$-valued (resp. $L^2$-valued) $\mathbb F_t$-local martingale extends to all $\mathbb F_t$-predictable processes $f$ satisfying $\P$-a.s.
$\int_0^T \| \cQ^{1/2} f_s \|_{L^2_x}^2 \,\d s <\infty$
(resp. $\int_0^T \| f_s \|_{L^2_x}^2\, \d s <\infty$) for all $T<\infty$.
\end{lemma}

\begin{proof}
The estimates for $M$ are classical and can be found for instance in \cite[Theorem 4.12]{DaPZab}; we only focus on the identity for $[N]_t$, as the moment bounds then readily follow from an application of Burkholder-Davis-Gundy (BDG) inequality.
By the representation \eqref{eq:noise-series-representation} and classical rules of stochastic calculus on Hilbert spaces, it holds
\begin{align*}
	[N]_t = \bigg[\sum_k \int_0^\cdot \sigma_k\cdot f_s\, \d B^k_s \bigg]_t
	= \sum_k \bigg[\int_0^\cdot \sigma_k\cdot f_s\, \d B^k_s \bigg]_t
	= \int_0^t \sum_k \| \sigma_k\cdot f_s\|_{L^2_x}^2\, \d s;
\end{align*}
for any fixed $s$, by \eqref{eq:covariance-series-representation} and translation invariance, we have
\begin{align*}
	\sum_k \| \sigma_k\cdot f_s\|_{L^2_x}^2
	= \sum_k \int_{\R^d} f_s(x)^T \sigma_k(x)\otimes \sigma_k(x) f_s(x)\, \d x
	= \int_{\R^d} f_s(x)^T Q(0)  f_s(x)\, \d x.
\end{align*}
Combined with \eqref{eq:Q_0}, this yields the claim.
Finally, the density and localization arguments are classical, see \cite[Section 4.2]{DaPZab} for more details.
\end{proof}

\begin{remark}\label{rem:stoch-integr-basic}
In view of Lemma \ref{lem:stoch-integr-basic} and Lemma \ref{lem:basic-properties-covariance}-ii), $\int_0^\cdot \langle f_s, \d W_s\rangle$ is a well defined local martingale for any measure-valued process $f$ such that $\P$-a.s. $\int_0^T \| f_s\|_{TV}^2\, \d s <\infty$ for all $T>0$.
\end{remark}

\subsection{Solution concepts and equations in It\^o form}\label{subsec:solution-concepts}

As already mentioned in the introduction, in order to study our SPDEs, it is convenient to rewrite them in the corresponding It\^o form.
A general class of equations, which includes those of our interest but also several useful approximations of them, is given by
\begin{equation}\label{eq:abstract-transport-stratonovich}
	\d \rho + b \cdot\nabla \rho\, \d t + \circ \d W\cdot\nabla \rho = (f- A \rho)\, \d t
\end{equation}
where $A$ is a deterministic self-adjoint positive operator on its domain $D(A)$, such that $A$ maps $\cS(\R^d)$ into itself (think of $A= \nu \Lambda^\beta - \eps \Delta$), while $f$ and $b$ are predictable stochastic processes, taking values in suitable Lebesgue spaces, $b$ divergence free; $W$ is a $Q$-Brownian motion, $Q$ satisfying Assumption \ref{ass:covariance-basic}.
In particular, we can always find a collection $\{\sigma_k\}_k$ of smooth, divergence free vector fields such that \eqref{eq:noise-series-representation} holds, so that the Stratonovich term can be (formally) interpreted in integral form as
\begin{equation*}
	\int_0^t \circ \d W_s\cdot\nabla \rho_s
	= \sum_{k\in\N} \int_0^t \sigma_k\cdot\nabla \rho_s \circ \d B^k_s.
\end{equation*}
If we assumed $\rho$ and $W$ to be smooth for the moment, then standard calculus rules relating It\^o and Stratonovich integrals would suggest the It\^o-Stratonovich corrector to be
\begin{align*}
	\frac{1}{2}\sum_k \d [ \sigma_k\cdot\nabla \rho, B^k]_t
	& = -\frac{1}{2} \sum_{k,j} \sigma_k\cdot\nabla (\sigma_j\cdot \nabla \rho_t)\, \d [B^k,B^j]_t\\
	& = -\frac{1}{2} \sum_k \nabla \cdot (\sigma_k\otimes \sigma_k \nabla \rho_t)\, \d t
	= -\frac{1}{2} \nabla\cdot (Q(0)\nabla \rho_t)\, \d t
	= -\kappa\Delta \rho_t\, \d t,
\end{align*}
where in the last passage we applied \eqref{eq:Q_0};
thus the Stratonovich SPDE \eqref{eq:abstract-transport-stratonovich} should have equivalent It\^o form given by
\begin{equation}\label{eq:abstract-transport-ito}
	\d \rho + b \cdot\nabla \rho\, \d t + \d W\cdot\nabla \rho = (f- A \rho + \kappa \Delta \rho)\, \d t.
\end{equation}
However, since in general both $W$ and $\rho$ will not be smooth and the sum in \eqref{eq:noise-series-representation} can be genuinely infinite dimensional, it is far from obvious that the equations \eqref{eq:abstract-transport-stratonovich} and \eqref{eq:abstract-transport-ito} are equivalent.
Remarkably, we have not found many rigorous results concerning this equivalence; this is because many authors either directly work with fairly regular noise (e.g. given by a \textit{finite} sum of $\sigma_k$), or justify the passage from Stratonovich to It\^o more heuristically and then work directly with the latter.
The only reference we are aware of, providing general criteria for equivalence of It\^o and Stratonovich SPDEs with infinite dimensional transport noise, is \cite[Section 2.3]{Goodair}, which however does not fit our setting.

To provide rigorous statements, we need to introduce some conventions and definitions.

Recall that a real-valued process $(X_t)_{t\in [0,T]}$ is a continuous semimartingale if it admits the canonical decomposition $X=X_0+V+M$, where $V$ is a continuous process with paths of bounded variation and $M$ is a local martingale on $[0,T]$.
Let $\tau$ be a $\P$-a.s. strictly positive stopping time; we will say that a sequence of semimartingales $\{X^n\}_n$ converge to $X$ on $[0,\tau)$ if there exists a common increasing sequence of stopping times $\{\tau^m\}_m$, $\tau^m\uparrow \tau$, such that
\begin{equation*}
	\lim_{n\to\infty} \E\bigg[|X^n_0-X_0|^2 + \Big( \int_0^{\tau_m} | \d V^n_s-\d V_s| \Big)^2 + [ M^n-M]_{\tau^m}\bigg]=0
	\quad \forall\, m\in \N.
\end{equation*}
In other words, $X^n\to X$ on $[0,\tau)$ if $X^n_0\to X_0$ in $L^2_\Theta$ and, for any $m\in \N$, the sequence of stopped semimartingales $\{X^n_{\cdot\wedge \tau_m}-X^n_0\}_n$ is such that $X^n_{\cdot\wedge \tau_m}-X^n_0\to X_{\cdot\wedge \tau_m}-X_0$ in the $\mathcal{H}^2$-norm, as defined in \cite[Chapter V.2]{Protter}.
Let us point out that in particular, this implies that $X^n$ converge to $X$ in probability, uniformly on compact sets on $[0,\tau)$.

We can now rigorously define weak solutions to the Stratonovich SPDE \eqref{eq:abstract-transport-stratonovich}.

\begin{definition}\label{defn:solution-abstract-stratonovich}
Let $(\Theta,\mathbb F, \mathbb F_t, \P)$ be a filtered probability space, $W$ be an $\mathbb F_t$-Brownian noise with covariance $Q$ satisfying Assumption \ref{eq:covariance-basic}.
Let $p\in [1,\infty]$, $p'$ denotes its conjugate, $f$ and $b$ be progressively measurable processes such that $\P$-a.s. $f\in L^1_tL^1_\loc$, $b\in L^1_t L^{p'}_\loc$ and $b$ is divergence free.
Let $\rho_0\in L^1_\loc$ be $\mathbb{F}_0$-measurable.
A progressively measurable $L^p_x$-valued process $\rho$, defined up to some stopping time $\tau$, is a \emph{Stratonovich weak solution} to \eqref{eq:abstract-transport-stratonovich} if the following hold:
\begin{itemize}
\item[a)] $\rho$ is weakly continuous and $\P$-a.s. it holds $\rho\in L^\infty([0,\tau); L^p_x)$;
\item[b)] for any $\varphi\in C_c^\infty$, $t\mapsto \langle \varphi, \rho_t\rangle$ is a continuous semimartingale on $[0,\tau)$;
\item[c)] for any $\varphi\in C^\infty_c$, it holds
\begin{equation}\label{eq:solution-abstract-stratonovich}\begin{split}
	\< \rho_\cdot,\varphi\> - \< \rho_0,\varphi\>
	& - \int_0^\cdot \big( \<\rho_s, b_s\cdot \nabla\varphi -A \varphi\>  + \< f_s, \varphi\> \big)\, \d s\\
	& = \lim_{n\to\infty} \sum_{k\leq n} \int_0^\cdot \<\sigma_k\cdot\nabla\varphi, \rho_s\> \circ \d B^k_s
\end{split}\end{equation}
where convergence holds in the sense of semimartingales on $[0,\tau)$.
\end{itemize}
\end{definition}

\begin{remark}\label{rem:solution-abstract-stratonovich}
The passage from \eqref{eq:abstract-transport-stratonovich} to \eqref{eq:solution-abstract-stratonovich} is formally justified by integration by parts, using the fact that both $b$ and $\sigma_k$ are divergence free vector fields.
Note that Definition \ref{defn:solution-abstract-stratonovich} is meaningful. Indeed, under condition a), $\P$-a.s. $b\,\rho\in L^1 ([0,\tau);L^1_\loc)$ and $A\varphi\in \cS$, so that the integral on the l.h.s. of \eqref{eq:solution-abstract-stratonovich} is meaningful in the Lebesgue sense and defines a bounded variation process on $[0,\tau)$.
On the other hand, by construction $\sigma_k\cdot\nabla\varphi\in C^\infty_c$ for all $k\in\N$, so that by condition b) the process $t\mapsto \< \sigma_k\cdot\nabla\varphi, \rho_t \>$ is a continuous semimartingale; the same thus holds for the Stratonovich integral $\int_0^\cdot \langle \sigma_k\cdot\nabla\varphi,\rho_s\rangle\circ \d B^k_s$, for any fixed $k$ (thus also for any finite sum of them as on the r.h.s. of \eqref{eq:solution-abstract-stratonovich}).
\end{remark}

\begin{definition}\label{defn:solution-abstract-ito}
Let $(\Theta,\mathbb F, \mathbb F_t, \P)$, $W$, $p$, $p'$, $f$, $b$ and $\rho_0$ be as in Definition \ref{defn:solution-abstract-stratonovich}.
A progressively measurable $L^p_x$-valued process $\rho$, defined up to some stopping time $\tau$, is an \emph{It\^o weak solution} to \eqref{eq:abstract-transport-stratonovich} if the following hold:
\begin{itemize}
\item[a)] $\rho$ is weakly continuous and $\P$-a.s. it holds $\rho\in L^\infty([0,\tau); L^p_x)$;
\item[b)] for any $\varphi\in C^\infty_c$, $\P$-a.s. it holds
\begin{equation}\label{eq:solution-abstract-ito} \begin{split}
	\< \rho_{\cdot\wedge \tau},\varphi\> = \< \rho_0,\varphi\>
	& + \int_0^{\cdot\wedge\tau} \big( \<\rho_s, b_s\cdot \nabla\varphi -A \varphi + \kappa\Delta\varphi\>  + \< f_s, \varphi\> \big)\, \d s \\
	& + \int_0^{\cdot\wedge\tau} \<\rho_s \nabla\varphi, \d W_s\>.
\end{split}\end{equation}
\end{itemize}
\end{definition}

\begin{remark}\label{rem:solution-abstract-ito}
Definition \ref{defn:solution-abstract-ito} is meaningful; indeed, as in Remark \ref{rem:solution-abstract-stratonovich}, the first integral in \eqref{eq:solution-abstract-ito} is well defined in the Lebesgue sense and defines a continuous process of bounded variation on $[0,\tau)$; instead for the stochastic integral, we observe that by our assumption $\rho\nabla\varphi\in L^\infty([0,\tau);L^1_x)$ and thus conclude that $\int_0^{\cdot\wedge \tau} \< \rho_s\nabla\varphi, \d W_s \>$ is a local martingale by Remark \ref{rem:stoch-integr-basic}.
\end{remark}

We can finally relate the two concepts of solutions introduced above.

\begin{proposition}\label{prop:equivalence-ito-stratonovich}
Let $(\Theta,\mathbb F, \mathbb F_t, \P)$, $W$, $p$, $p'$, $f$, $b$ and $\rho_0$ be as in Definition \ref{defn:solution-abstract-stratonovich}, $\rho$ be a progressively measurable $L^p_x$-valued process, defined up to some stopping time $\tau$.
Suppose that at least one of the following conditions holds:
\begin{itemize}
\item[i)] $\P$-a.s. $\rho\in L^1 \big([0,\tau); W^{1,1}_\loc \big)$;
\item[ii)] $Q\in W^{2,\infty}_x$.
\end{itemize}
Then $\rho$ satisfies Definition \ref{defn:solution-abstract-stratonovich} if and only if it satisfies Definition \ref{defn:solution-abstract-ito}.
\end{proposition}

We postpone the proof to Appendix \ref{app:equiv-ito-strat}, in order not to interrupt the flow of the presentation.

\begin{remark}
We expect that one can interpolate between the two sufficient conditions from Proposition \ref{prop:equivalence-ito-stratonovich}; namely, that in order for Definitions \ref{defn:solution-abstract-stratonovich} and \ref{defn:solution-abstract-ito} to be equivalent, it suffices to exhibit some $\theta\in [0,1]$ such that $\P$-a.s. $\rho\in L^1 \big([0,\tau);W^{\theta,1}_\loc \big)$ and $Q\in W^{2(1-\theta),\infty}_x$. We leave this problem for future investigations.
\end{remark}

For the SPDEs of our interest in general neither condition \textit{i}) nor \textit{ii}) from Proposition \ref{prop:equivalence-ito-stratonovich} will be met.
Observe however that, in order to be meaningful, Definition \ref{defn:solution-abstract-ito} does not impose any of them; in other words, Definition \ref{defn:solution-abstract-ito} is often weaker than Definition \ref{defn:solution-abstract-stratonovich}.

For this reason, whenever referring to weak solutions to either System \ref{system:logEuler}, System \ref{system:hypoNS} or SPDEs \eqref{eq:intro-abstract-spde} and \eqref{eq:intro-linear-spde}, from now on we will adopt the convention that, unless specified otherwise, these will be solutions in the \emph{It\^o sense}; namely, $\circ \d W\cdot\nabla \rho$ will be systematically interpreted as $\d W\cdot\nabla \rho - \kappa \Delta \rho\, \d t $.

We now specialize the general solution concept from Definition \ref{defn:solution-abstract-ito} to the specific cases of interest \eqref{eq:intro-abstract-spde} and \eqref{eq:intro-linear-spde}, where it is natural to impose some slightly different, more physical, constraints on the solutions.
Below we always assume the existence of an underlying filtered probability space $(\Theta,\F,\F_t,\P)$, on which there exists an $\F_t$-Brownian noise with covariance function $Q$ satisfying Assumption \ref{eq:covariance-basic}.

\begin{definition}\label{defn:linear-weak-solution}
Let $p\in [1,\infty]$, $f$ and $\rho_0$ be as in Definition \ref{defn:solution-abstract-stratonovich}, $\tau$ be a stopping time.
An $\mathbb F_t$-progressively measurable $L^p_x$-valued process $\rho$, defined on $[0,\tau)$, is a solution to SPDE \eqref{eq:intro-linear-spde} if it is weakly continuous, $\P$-a.s. it holds $\rho\in L^2([0,\tau);L^p_x)$ and for any $\varphi\in C^\infty_c$, $\P$-a.s. on $[0,\tau)$ it holds
\begin{equation}\label{eq:defn-linear-weak-solution}
	\< \rho_\cdot,\varphi\>
	= \< \rho_0,\varphi\> + \int_0^\cdot \<\rho_s \nabla\varphi,\d W_s\> + \int_0^\cdot \big[ \big\langle \rho_s, (\kappa \Delta -\nu \Lambda^\beta) \varphi \big\rangle + \< f_s,\varphi \>\big]\, \d s.
\end{equation}
\end{definition}

\begin{remark}\label{rem:defn-linear-weak-solution}
Similarly to Remark \ref{rem:solution-abstract-ito}, we see that the integrals in \eqref{eq:defn-linear-weak-solution} are meaningful, respectively in the It\^o and in the Lebesgue sense.
Depending on the precise value of $p$ and additional information on the solution $\rho$, the set of test functions $\varphi$ such that identity \eqref{eq:defn-linear-weak-solution} holds can be expanded by classical density arguments.
For instance for $p=1$, one can extend this to any $\varphi\in C^\infty_b$;
indeed, it still holds $\Lambda^\beta \varphi, \Delta\varphi\in C^\infty_b$ and so all pairings appearing in \eqref{eq:defn-linear-weak-solution} are meaningful.
\end{remark}

Next we restrict ourselves to $d=2$; recall the nonlinear SPDE we are interested in, generalizing both System \ref{system:logEuler} and \ref{system:hypoNS}, given by
\begin{equation}\label{eq:SPDE-nonlinear-strat}
	\d \omega + (\cR\, \curl^{-1}\omega)\cdot\nabla\omega \, \d t+ \circ \d W\cdot\nabla\omega
	= (f- \nu\Lambda^\beta \omega)\, \d t.
\end{equation}
We will mostly work with deterministic data $\omega_0$ and $f$, but let us give a slightly more general definition.

\begin{definition}\label{defn:solution-nonlinear-SPDE}
Let $f$ be a progressively measurable process with paths in $L^1_t L^1_\loc$, $\cR$ be a Fourier multiplier associated to $r\in L^\infty_x$; let $\omega_0$ be $\F_0$-measurable with $\P$-a.s. values in $L^1_x\cap L^2_x \cap \dot H^{-1}_x$, $\tau$ be a stopping time.
A progressively measurable process $\omega$, with values in $L^1_x\cap L^2_x\cap \dot H^{-1}_x$, is a weak solution to \eqref{eq:SPDE-nonlinear-strat} on $[0,\tau)$ if it is weakly continuous, $\P$-a.s. $\omega\in L^2 \big([0,\tau); L^1_x	\cap L^2_x \cap \dot{H}^{-1}_x \big)$, and for any $\varphi\in C_c^\infty$, $\P$-a.s. on $[0,\tau)$ it holds
\begin{equation}\label{eq:solution-nonlinear-SPDE}\begin{split}
	\< \omega_{t\wedge \tau},\varphi\>
	& = \< \omega_0,\varphi\> + \int_0^{t\wedge \tau} \<\omega_s \nabla\varphi,\d W_s\> + \int_0^{t\wedge \tau} \<\omega_s, (\cR\, \curl^{-1}\omega_s)\cdot\nabla\varphi\>\,\d s \\
  	& \quad + \int_0^{t\wedge \tau} \big[ \big\langle\omega_s, (\kappa \Delta -\nu \Lambda^\beta) \varphi \big\rangle + \< f_s,\varphi \>\big]\, \d s.
\end{split}\end{equation}
\end{definition}

\begin{remark}
Roughly speaking, Definition \ref{defn:solution-nonlinear-SPDE} coincides with Definition \ref{defn:solution-abstract-ito} for $b=\cR\,\curl^{-1}\omega$; we are however specifying that we work exclusively with solutions with \emph{finite energy and enstrophy}, so that $\curl^{-1} \omega$ is a well defined element of $L^2_x$ at all times.
Properties of Fourier multipliers $\cR$ then readily imply that $\cR\, \curl^{-1}\omega \in L^2([0,\tau);L^2_x)$, thus arguing as before one can verify that all integrals appearing in \eqref{eq:solution-nonlinear-SPDE} are meaningful.
\end{remark}

Finally, all classical concepts such as \emph{strong solutions}, \emph{pathwise uniqueness} and \emph{uniqueness in law} can be readapted in the context of Definitions \ref{defn:linear-weak-solution} and \ref{defn:solution-nonlinear-SPDE}.
In particular, we will say that a solution is \emph{strong} if it is adapted to the filtration generated by all the data of the problem, which in this case can be e.g. $(\rho_0, W, f, b)$.

\section{Proof of Theorem \ref{thm:main-linear-intro}}\label{sec:wellposedness-linear}

This section is devoted to the well-posedness of the linear SPDE \eqref{eq:intro-linear-spde} on $\R^d$, $d\geq 2$.
We start by showing strong existence in Section \ref{subsec:existence-linear}, by standard compactness arguments, and then pass to prove pathwise uniqueness in Section \ref{subsec:uniqueness-linear}, by studying the evolution of energy spectrum in frequency space.

\subsection{A priori estimates and strong existence}\label{subsec:existence-linear}

Throughout this section, we fix a filtered probability space $(\Theta,\mathbb F,\mathbb F_t, \P)$ and assume that the noise $W$ is $\mathbb F_t$-adapted.

In the linear setting, the proof of strong existence is based on standard approximations and compactness arguments.
To this end, we start by considering a smooth noise $W$, in the sense of Remark \ref{rem:smooth-noise}, so that well-posedness of the SPDE \eqref{eq:intro-linear-spde} with smooth coefficients is classical, cf. \cite[Chapter 6]{Kunita}.
Our first aim is to derive a priori estimates for solutions, which will then allow us to relax the regularity assumptions on the noise and the coefficients. In the following result we allow the presence of a smooth drift part $b$.

\begin{proposition}\label{prop:apriori-estimates-linear}
Let $d\geq 2$, $\rho_0:\R^d\to\R$ be a deterministic smooth function, $f:\Theta\times \R_{\geq 0}\times \R^d\to \R$, $b:\Theta\times \R_{\geq 0} \times \R^d\to \R^d$ be smooth predictable processes;
additionally assume that $b$ is spatially divergence free and that $W$ is smooth.
For any $\nu\geq 0$, $\eps>0$ and $\beta\in (0,2)$, consider the unique strong smooth solution $\rho$ to
\begin{equation}\label{eq:linear-spde-apriori}
	\d \rho + [ b\, \d t + \circ \d W]\cdot\nabla \rho
	= \big[ f + \nu \eps \Delta \rho- \nu \Lambda^\beta \rho\big] \d t  .
\end{equation}
Then for any $p\in [1,\infty]$ we have the $\P$-a.s. pathwise estimate
\begin{equation}\label{eq:apriori-linear}
	\| \rho_t\|_{L^p_x}
	\leq \| \rho_0\|_{L^p_x} + \int_0^t \| f_s\|_{L^p_x} \,\d s
	\quad \forall\, t\geq 0.
\end{equation}
\end{proposition}

\begin{proof}

Let us first consider the case $\nu=0$; here well-posedness of \eqref{eq:linear-spde-apriori} comes from the method of stochastic characteristics, as illustrated in \cite[Section 6.1]{Kunita}.
In particular, by applying Theorem 6.1.9 therein for  $(F^1,\ldots, F^d)=\int_0^t b_s\, \d s + W_t$, $F^{d+1}=0$ and $F^{d+2}= \int_0^t f_s\, \d s$, we can deduce that
\begin{equation}\label{eq:linear-estimates-representation}
	\rho_t(X_t^x) = \rho_0(x) + \int_0^t f_s(X_s^x)\,\d s,
\end{equation}
where $X_t^x$ is the stochastic flow associated to the Stratonovich SDE
\begin{equation}\label{eq:SDE}
	\d X_t^x = b_t(X_t^x)\,\d t + \sum_k \sigma_k(X_t^x)\circ \d B^k_t, \quad X^x_0=x.
\end{equation}
Estimate \eqref{eq:apriori-linear} then follows by taking the $L^p_x$-norm on both sides of \eqref{eq:linear-estimates-representation} and using the fact that $\P$-a.s. the flow $x\mapsto X^x_t$ preserves the Lebesgue measure for all $t\geq 0$, which comes from the divergence free assumption on $b$ and $\sigma_k$. More generally, in this setting the Jacobian $J_t^x = \det (D_x X^x_t)$ is given by the formula
\begin{equation*}
	J_t^x = \exp\bigg(\int_0^t \div b_r(X_r^x)\, \d r + \sum_{k\in \N} \int_0^t \sigma_k(X_r^x) \circ \d B^k_r \bigg);
\end{equation*}
this can be established by applying \cite[Lemma 4.3.1]{Kunita} for $\pi\equiv 1$, or alternatively going through the same argument as in \cite[Lemma 3.1]{Zhang2010}, based on time mollification and the Wong--Zakai theorem.

We now pass to the case $\nu>0$; here, although well-posedness of \eqref{eq:linear-spde-apriori} is standard, we are not aware of a precise reference in the literature.
Thus let us quickly mention three alternative ways to establish it:
i) use monotonicity arguments, like those first developed in \cite{Pardoux}, see \cite{LiuRoc} for a general exposition and \cite{FlGaLu21,RoShZh} in the case of transport noise;
ii) perform a flow transformation to reduce the SPDE to a PDE with random coefficients and show uniqueness of the latter, in the style of \cite[Lemma 6.2.3]{Kunita};
iii) identify Stratonovich noise with a geometric rough path, and treat the SPDE pathwise by unbounded rough drivers techniques as in \cite{DGHT}.

Smoothness of solutions can then be established by classical arguments, either by exploiting the associated mild formulation and the regularizing properties of the heat kernel, or iteratively looking at the equations satisfied by the partial derivatives $\partial^\alpha \rho$ for $\alpha\in \N^d$;
similarly, one can establish their decay at infinity, e.g. by looking at the equation satisfied by $x\mapsto (1+|x|^2)^m \rho_t(x)$ and using monotonicity arguments to get $L^2_x$-bounds.
Overall, we can and will assume that our solution $\rho$ take values $H^s_x$ for all $s$, so that it belongs to $L^\infty_x$, and decays polynomially fast at infinity, so that it belongs to $L^1_x$.

Let us now compute $\d \int h(\rho)\, \d x$ for a smooth convex function $h:\R\to\R$ such that $h(0)=0$; by the previous observations, $h(\rho_t)\in L^1_x$ as well.
By the divergence-free assumption on $b$, it holds
\begin{equation*}
	\langle b\cdot\nabla \rho, h'(\rho)\rangle
	= \langle b, \nabla ( h(\rho))\rangle
	= -\langle \div b, h(\rho)\rangle = 0,
\end{equation*}
similarly for $b$ replaced by $\sigma_k$; together with the Stratonovich chain rule, this implies
\begin{equation*}
	\d \int_{\R^d} h(\rho_t(x))\, \d x
	= \langle h'(\rho_t), \circ \d \rho_t\rangle
	= \langle h'(\rho_t), f_t - \nu \Lambda^\beta \rho_t + \nu\eps \Delta \rho_t\rangle\, \d t.
\end{equation*}
By convexity of $h$, it holds $(h'(v)-h'(z))(v-z)\geq 0$, so that by \eqref{eq:fractional-laplacian2} we find
\begin{equation*}
	-\nu \langle \Lambda^\beta \rho_t, h'(\rho_t) \rangle
	= -\nu\, C_{d,\beta} \int_{\R^{2d}} \frac{[\rho_t(x)-\rho_t(y)] [h'(\rho_t(x))-h'(\rho_t(y)) ]}{|x-y|^{d+\beta}}\, \d x \d y \leq 0;
\end{equation*}
moreover using $h''\geq 0$, by integration by parts
\begin{equation*}
	\langle h'(\rho_t), \Delta \rho_t \rangle = -\int_{\R^d} h''(\rho_t(x)) | \nabla \rho_t(x)|^2 \d x \leq 0,
\end{equation*}
so that overall we obtain
\begin{equation*}
	\frac{\d}{\d t} \int_{\R^d} h(\rho_t(x))\, \d x \leq \langle f_t, h'(\rho_t)\rangle.
\end{equation*}
By approximating the convex functions $x\mapsto |x|^p$, $p\in [1,\infty)$, by smooth convex functions $h$, we then obtain the same relation in this case, in particular
\begin{equation*}
	\frac{\d}{\d t} \| \rho_t\|_{L^p_x}^p
	\leq p \int_{\R^d} |f_t(x)|\, |\rho_t(x)|^{p-1}\,\d x
	\leq p \| f_t\|_{L^p_x} \|\rho_t\|_{L^p_x}^{1-\frac{1}{p}}.
\end{equation*}
This implies
\begin{equation*}
	\frac{\d}{\d t} \| \rho_t\|_{L^p_x}
	= \frac{1}{p} \| \rho_t\|_{L^p_x}^{\frac{1}{p}-1} \frac{\d }{\d t}\| \rho_t\|_{L^p_x}^p \leq \| f_t\|_{L^p_x}
\end{equation*}
which upon integrating on $[0,t]$ readily yields \eqref{eq:apriori-linear} whenever $p\in [1,\infty)$.

The case $p=\infty$ can be obtained by passing to the limits on both sides of \eqref{eq:apriori-linear} as $p\to\infty$, using dominated convergence and Lemma \ref{lem:Lp-norms-appendix} in Appendix \ref{app:useful}, which guarantees that
$\| \rho_t\|_{L^\infty_x}=\lim_{p\to\infty} \| \rho\|_{L^p_x}$, similarly for $f$.
\end{proof}

\begin{remark}\label{rem:energy-estim-linear-fractional-laplacian}
In the above, we never used the regularizing properties of the fractional Laplacian $\Lambda^\beta$.
It is clear however that, in the special case $p=2$ and $\nu>0$, going through the same computations, one can arrive at the $\P$-a.s. estimate
\begin{equation}\label{eq:energy-estim-linear-fractional-laplacian}
	\| \rho_t\|_{L^2_x}^2  + 2\nu \int_0^t \| \rho_s\|_{\dot H^{\beta/2}_x}^2\, \d s
	\leq \bigg( \| \rho_0\|_{L^2_x} + \int_0^t \|f_s\|_{L^2_x}\, \d s \bigg)^2 \quad \forall\, t\geq 0.
\end{equation}
\end{remark}

With Proposition \ref{prop:apriori-estimates-linear} at hand, we are now ready to establish probabilistically strong existence of analytically weak solutions for the SPDE \eqref{eq:intro-linear-spde}.

\begin{proposition}\label{prop:strong-existence-linear}
Let $p\in[1,\infty]$, $\nu\geq 0$ and $\beta\in (0,2)$; let $W$ be an $\F_t$-Brownian motion with covariance $Q$ satisfying Assumption \ref{ass:covariance-basic} and let $f:\Theta\times \R_{\geq 0}\times \R^d \to\R$ be an $\mathbb F_t$-predictable $L^p_x$-valued process such that
\begin{equation}\label{eq:strong-existence-linear-assumption}
	\E\bigg[ \Big| \int_0^T \| f_s\|_{L^p_x}\, \d s \Big|^2 \bigg] <\infty \quad \forall\, T>0.
\end{equation}
Set $Q(0)=2\kappa I_d$. Then for any deterministic $\rho_0\in L^p_x$ there exists a global probabilistically strong, analytically weak $L^p_x$-valued solution $\rho:\Theta\times\R_{\geq 0}\times\R^d\to\R$ to
\begin{equation}\label{eq:linear-spde-no-drift}
	\d \rho + \d W\cdot \nabla \rho = [f - \nu \Lambda^\beta \rho + \kappa \Delta \rho]\, \d t
\end{equation}
which moreover satisfies the $\P$-a.s. pathwise bound \eqref{eq:apriori-linear}.
\end{proposition}

\begin{proof}
The idea of proof, based on a priori estimates and convergence of subsequences in suitable weak topologies, is classical; see for instance the references \cite{Pardoux, FlGuPr2010, Maurelli, Gal} for other proofs in slightly different settings.

For notational simplicity, we will only show existence of a strong solution defined on a finite time interval $[0,T]$, although arbitrarily large; the extension to $[0,+\infty)$ is straightforward, based on the extraction of a further subsequence by a diagonalization procedure.
We divide the proof in three steps.

\textit{Step 1: Approximations.} For any $N\ge 1$, consider the solution $\rho^N$ to
\begin{equation}\label{eq:proof-defn-approximants}
	\d \rho^N + \circ \d W^N \cdot \nabla \rho^N
	= \big[f^N + \Lambda^\beta \rho^N + \eps^N\Delta \rho^N \big]\, \d t, \quad
	\rho^N\vert_{t=0} = \rho^N_0.
\end{equation}
Here $\eps^N\to 0^+$ and the noises $W^N$ are defined by $W^N = p_{N^{-1}} \ast W$, where $p_t$ denotes the standard heat kernel;
in particular, the associated covariances $Q^N$ are given by formula \eqref{eq:covariance-basic} with $g$ replaced by $g^N(\xi):= g(\xi) e^{-|\xi|^2/N}$, so that $W^N$ are smooth.
Moreover $f^N$, $\rho_0^N$ are smooth approximations of $f$, $\rho_0$ such that
\begin{equation}\label{eq:proof-properties-approx1}
	\| \rho^N_0-\rho_0\|_{L^1_x\cap L^{\tilde p}_x} \to 0, \quad
	\int_0^T \E \big[\| f^N_r-f_r\|_{L^1_x\cap L^{\tilde p}_x} \big]\, \d r\to 0
	\quad \forall\, \tilde p\in [1,p]\setminus\{\infty\}
\end{equation}
and satisfying the $\P$-a.s. bounds
\begin{equation}\label{eq:proof-properties-approx2}
	\| \rho^N_0\|_{L^p_x} \leq \| \rho_0\|_{L^p_x}, \quad
	\int_0^T \| f^N_s\|_{L^p_x}\, \d s \leq \int_0^T \| f_s\|_{L^p_x}\, \d s;
\end{equation}
one can verify that such approximations always exist.

By their definition and Remark \ref{rem:smooth-noise}, it is clear that $W^N$ are smooth noises;
since everything is regular, by Proposition \ref{prop:apriori-estimates-linear} the solutions $\rho^N$ to \eqref{eq:proof-defn-approximants} strongly exist.
Condition \eqref{eq:proof-properties-approx2}, combined with \eqref{eq:apriori-linear}, implies that $\rho^N$ satisfies the $\P$-a.s. bound
\begin{equation}\label{eq:proof-uniform-bound}
	\| \rho^N_t\|_{L^p_x} \leq \| \rho_0\|_{L^p_x} + \int_0^t \| f_s\|_{L^p_x} \,\d s\quad \forall\, t\in [0,T];
\end{equation}
together with assumption \eqref{eq:strong-existence-linear-assumption}, we find
\begin{equation}\label{eq:proof-uniform-bound-2}
	\sup_N \E\bigg[\sup_{t\in [0,T]} \| \rho^N_t\|_{L^p_x}^2 \bigg] <\infty.
\end{equation}
Moreover the solutions $\rho^N$ solve the equivalent It\^o SPDE; thus for any $\varphi\in C_c^\infty$, it holds $\P$-a.s. for all $t\in [0,T]$,
\begin{equation}\label{eq:approximating-linear}\begin{split}
	\<\rho^N_t,\varphi \>
	& = \<\rho^N_0,\varphi \> + \int_0^t \< \rho^N_s \nabla \varphi, \d W^N_s\> + \int_0^t \< f^N_s, \varphi \>\,\d s \\
	& \quad + \int_0^t \< \rho^N_s, \Lambda^\beta \varphi \>\,\d s + \int_0^t \big\<\rho^N_s, \div \big([\eps^N I_d + Q^N(0)]\nabla \varphi \big) \big\>\,\d s.
\end{split}\end{equation}

\textit{Step 2: Compactness.}
In view of the uniform bounds \eqref{eq:proof-uniform-bound}-\eqref{eq:proof-uniform-bound-2}, we can now extract weak limits.
Let us shortly distinguish three cases of interest:
\begin{itemize}
\item[i)] $p\in (1,\infty)$: here \eqref{eq:proof-uniform-bound-2} implies uniform boundedness of $\{\rho^N \}_{N\ge 1}$ in $L^2_\Theta L^p_t L^p_x$, which is reflexive; thus $\{\rho^N \}_{N\ge 1}$ is precompact in the weak topology and we can extract a (not relabelled for simplicity) subsequence converging weakly in $L^2_\Theta L^p_t L^p_x$ to some limit $\rho$.
\item[ii)] $p=\infty$: here \eqref{eq:proof-uniform-bound-2} gives a uniform bound in $L^2_\Theta L^\infty_t L^\infty_x$; while not reflexive, this space still has weak-$\ast$ compact balls, thus we can extract a subsequence weakly-$\ast$ converging to some $\rho\in L^2_\Theta L^\infty_t L^\infty_x$.
\item[iii)] $p=1$: this case is the trickiest since $L^1_x$ is not reflexive nor weak-$\ast$ compact; indeed we will actually need to pass to considering $L^1_{\rm loc}$.
Nevertheless, one can follow the same arguments as in \cite[Proposition II.1, bottom p. 515]{DiPLio} (conceptually simple, but notationally heavy, as they require the introduction of another auxiliary approximation and thus a double-indexed sequence $\{\rho^{N,\eps}\}_{N\geq 1, \eps>0}$) to establish weak compactness of $\{\rho^N\}_N$ in $L^2_\Theta L^2_t L^1_\loc$, from which we can again extract a subsequence.
\end{itemize}
In the above three cases, since $\rho$ is also a weak limit in the sense of distributions, the lower semicontinuity of all the norms involved implies the pathwise bound \eqref{eq:apriori-linear}, i.e. \eqref{eq:proof-uniform-bound} holds with $\rho^N$ replaced by $\rho$.
Moreover, the process $\rho$ constructed in this way is still $\F_t$-predictable, since this properly defines closed sets w.r.t. to both weak and strong convergence in $L^2_\Theta L^p_t L^p_\loc$ for any $p\in [1,\infty]$; see e.g. the argument from the proof of \cite[Theorem 15]{FlGuPr2010}.

\textit{Step 3: Passage to the limit.} In order to show that the limit constructed in this way is an analytically weak solution to \eqref{eq:linear-spde-no-drift}, we need to take weak limits on both sides of \eqref{eq:approximating-linear};
more precisely, we need to verify that for any $\varphi\in C^\infty_c$ it holds (we omit for simplicity the time variable $s$ and $\d s$ in the integrals)
\begin{align}
	& \int_0^\cdot \langle \rho^N \nabla \varphi, \d W^N\rangle \rightharpoonup \int_0^\cdot \langle \rho \nabla \varphi, \d W\rangle, \label{eq:proof-goal-1}\\
	& \int_0^\cdot \langle f^N, \varphi \rangle + \big\<\rho^N, \div \big([\eps^N I_d + Q^N(0)]\nabla \varphi \big) \big\>
\rightharpoonup
\int_0^\cdot \langle f, \varphi \rangle + \langle \rho, \div ( Q(0)\nabla \varphi)\rangle, \label{eq:proof-goal-2}\\
	& \int_0^\cdot \langle \Lambda^\beta \varphi,\rho^N \rangle \rightharpoonup \int_0^\cdot \langle \Lambda^\beta \varphi,\rho \rangle. \label{eq:proof-goal-3}
\end{align}
We present the computations only for the case $p=1$, which is the most challenging one.

Proof of \eqref{eq:proof-goal-1}:
we split this step in two parts, as one can show respectively that
\begin{align*}
	\int_0^\cdot \langle \rho^N \nabla \varphi, \d W\rangle \rightharpoonup \int_0^\cdot \langle \rho \nabla \varphi, \d W\rangle, \quad
	\int_0^\cdot \langle \rho^N \nabla \varphi, \d (W-W^N)\rangle\to 0.
\end{align*}
The first limit follows from weak convergence arguments, see e.g. the argument from \cite[Theorem 15]{FlGuPr2010};
for the second one, using the definition of $W^N$, Lemma \ref{lem:stoch-integr-basic} and Lemma \ref{lem:basic-properties-covariance}-ii), it holds
\begin{align*}
	\E\bigg[\sup_{t\in  [0,T]} \Big|\int_0^t \langle \rho^N_r \nabla \varphi,\d (W_r-W^N_r)\rangle\Big|^2 \bigg]
	& \lesssim \int_0^T \E\big[ \| (\cQ-\cQ^N)^{1/2}\rho^N_r \nabla \varphi\|_{L^2_x}^2 \big]\, \d r\\
	& \lesssim \int_0^T \| g-g^N\|_{L^1_x} \E \big[ \| \rho^N_r \nabla \varphi\|_{L^1_x}^2 \big]\, \d r\\
	& \lesssim \| g-g^N\|_{L^1_x} \| \nabla\varphi\|_{L^\infty_x}^2 \sup_N\, \int_0^T \E\big[ \|\rho^N_r\|_{L^1_x}^2\big]\, \d r\\
	& \lesssim \| g-g^N\|_{L^1_x},
\end{align*}
where the last quantity converges to $0$ by dominated convergence.

Proof of \eqref{eq:proof-goal-2}:
the limit $\int_0^\cdot \langle f^N, \varphi \rangle \rightharpoonup \int_0^\cdot \langle f, \varphi \rangle$ follows directly from properties of weak convergence, similarly for
\begin{align*}
	\int_0^\cdot \big\<\rho^N, \div (\eps^N I_d\, \nabla \varphi) \big\>
	= \eps^N \int_0^\cdot \<\rho^N, \Delta \varphi \> \rightharpoonup 0.
\end{align*}
Moreover, using the strong convergence of $Q^N(0)\to Q(0)$ and weak convergence of $\rho^N$ to $\rho$, we obtain
$\int_0^\cdot \langle \rho^N, \div (Q^N(0)\nabla \varphi)\rangle\rightharpoonup \int_0^\cdot \langle \rho, \div (Q(0)\nabla \varphi)\rangle$.

Proof of \eqref{eq:proof-goal-3}:
for any $R>0$ we can apply the decomposition $\Lambda^\beta=\Lambda^\beta_{R,1}+\Lambda^\beta_{R,2}$ as described in Lemma \ref{lem:decomposition-fractional-laplacian} from Appendix \ref{app:useful}; in particular
\begin{equation}\label{eq:proof-existence-splitting}
	\int_0^\cdot \big\langle \Lambda^\beta\varphi, \rho^N_r-\rho_r \big\rangle\, \d r
	= \int_0^\cdot \big\langle \Lambda^\beta_{R,1} \varphi, \rho^N_r-\rho_r \big\rangle\, \d r + \int_0^\cdot \big\langle \varphi, \Lambda^\beta_{R,2}(\rho^N_r-\rho_r) \big\rangle\, \d r.
\end{equation}
Since $\Lambda^\beta_{R,1} \varphi$ is bounded and compactly supported, while $\rho^N\rightharpoonup \rho$ in $L^2_{\omega,t} L^1_{x, {\rm loc}}$,  we deduce that the first term on the r.h.s. of \eqref{eq:proof-existence-splitting} converges weakly to $0$ in $L^2_\omega C^0_t$;
on the other hand, again by Lemma \ref{lem:decomposition-fractional-laplacian}, it holds
\begin{align*}
	& \E\bigg[ \sup_{t\in [0,T]} \Big| \int_0^t \big\langle \varphi, \Lambda^\beta_{R,2}(\rho^N_r-\rho_r) \big\rangle\, \d r\Big| \bigg]\\
	& \lesssim \|\varphi\|_{L^\infty_x} R^{-\beta} \sup_N \E\bigg[\int_0^T \big(\|\rho^N_r\|_{L^1_x} + \|\rho_r\|_{L^1_x} \big)\, \d r\bigg]
	\lesssim R^{-\beta}.
\end{align*}
Since $R$ can be chosen arbitrarily large, combining these facts allows us to deduce that \eqref{eq:proof-goal-3} indeed holds.

Summarizing the above arguments, we can take weak limits in \eqref{eq:approximating-linear} to get
\begin{align*}
	\<\rho_t,\varphi \>
	& = \<\rho_0,\varphi \> + \int_0^t \< \rho_s \nabla \varphi, \d W_s\> + \int_0^t \< f_s, \varphi \>\,\d s + \int_0^t \big\< \rho_s, (\Lambda^\beta +2\kappa \Delta )\varphi \big\>\,\d s .
\end{align*}
Since the right hand side defines a continuous stochastic process, we see that $t\mapsto \<\rho_t,\varphi \>$ has a continuous modification; as the argument holds for any $\varphi$, we conclude that the selected limit $\rho$ admits a weakly continuous modification.
\end{proof}

\subsection{Pathwise uniqueness and stability}\label{subsec:uniqueness-linear}

Next we verify pathwise uniqueness in the class of $L^1_x$-valued, local in time solutions.
The proof relies on a non-standard argument, involving integral inequalities for the energy spectrum of solutions, and is inspired by the ones from \cite[Proposition 3]{BBF} and \cite[Lemma 2]{Gal}.
Among notable differences, here we can cover $L^1_x$-valued functions on $\R^d$, possibly only defined up to a stopping time;
as a byproduct, we can deduce a number of interesting observations on the Kraichnan model, see Remarks \ref{rem:michele} and \ref{rem:invariant-measure-linear} below.

\begin{proposition}\label{prop:pathwise-uniqueness-linear}
Let $\beta\in (0,2)$, $\nu\geq 0$, $\rho_0\in L^1_x$, $f\in L^1_t L^1_x$ and $W$ be a noise with the associated covariance $Q$ satisfying Assumption \ref{ass:covariance-basic}.
Let $\rho^1$, $\rho^2$ be $L^1_x$-valued solutions to \eqref{eq:intro-linear-spde} for the same $(\rho_0, f, W)$, defined respectively up to stopping times $\tau^1,\, \tau^2$, in the sense of Definition \ref{defn:linear-weak-solution}; assume that $\P$-a.s.
\begin{equation*}
	\int_0^{\tau^i} \| \rho^i_t \|_{L^1_x}^2\,\d t <+\infty, \quad i=1,2.
\end{equation*}
Then it holds
\begin{equation*}
	\sup_{t\in [0,\tau^1\wedge \tau^2]} \| \rho^1_t-\rho^2_t\|_{L^1_x} =0 \quad \P\text{-a.s.}
\end{equation*}
\end{proposition}

\begin{proof}
As the difference $\rho:=\rho^1-\rho^2$ is again a local $L^1_x$-valued solution of the homogeneous equation \eqref{eq:intro-linear-spde} with $\rho_0=0=f$, defined up to the random time $\tau:=\tau^1\wedge \tau^2$, in order to conclude it suffices to show that any such solution is $\P$-a.s. constantly $0$ on $[0,\tau]$.

Let us define an increasing sequence of stopping times $\tau_N$ by
\begin{equation*}
	\tau_N := \inf\Big\{t\in [0,\tau) : \int_0^{t} \| \rho_s\|_{L^1_x}^2\, \d s \geq N\Big\}	
	\quad \forall\, N\in\N,
\end{equation*}
with the convention that $\inf\emptyset=\tau$; then our task is reduced to show that for any fixed $N$, $\P$-a.s. $\rho\equiv 0$ on $[0,\tau^N]$.
We divide the rest of the proof in two main steps.

\textit{Step 1: Integral inequality for the spectral energy.}
For any $\xi\in \R^d$, define
\begin{equation*}
	A^N(\xi) = \E\bigg[\int_0^{\tau_N}|\hat\rho_s(\xi)|^2\, \d s\bigg]
\end{equation*}
where $\hat{\rho}_s$ denotes the Fourier transform of $\rho_s$.
It follows from the definition of $\tau_N$, the properties of Fourier transform of $L^1_x$-functions and dominated convergence that $A^N$ is a continuous, bounded function of $\xi$.
Our aim is to find an integral inequality for $A^N$, which will allow us to show that $A^N\equiv 0$.

We can compute $| \hat{\rho}(\xi)|^2$ using the structure of the SPDE: it holds
\begin{equation}\label{eq:uniqueness-linear-proof-1}
	\d | \hat{\rho}_t(\xi)|^2
	= 2 \Re \Big(\overline{\hat{\rho_t}(\xi)}\, \d \hat{\rho}_t(\xi) \Big) + \d [\hat{\rho}(\xi)]_t;
\end{equation}
by taking $\varphi(x) = (2\pi)^{-d/2} e^{-i\xi\cdot x}$ in \eqref{eq:defn-linear-weak-solution} (which is allowed by Remark \ref{rem:defn-linear-weak-solution}), we have
\begin{equation}\label{eq:uniqueness-linear-proof-2}
	\d \hat\rho_t(\xi)
	= -i(2\pi)^{-d/2} \langle \rho_t \xi e^{-i\xi\cdot x}, \d W_t\rangle - (\kappa |\xi|^2 + \nu |\xi|^\beta) \hat{\rho}_t(\xi)\, \d t.
\end{equation}
By Lemma \ref{lem:stoch-integr-basic} and our convention for the Fourier transform, it holds
\begin{equation}\label{eq:uniqueness-linear-proof-3}\begin{split}
	\frac{\d [\hat{\rho}(\xi)]_t}{\d t}
	& = (2\pi)^{-d}\, \big\| \cQ^{1/2} ( \rho_t \xi e^{-i\xi\cdot x}) \big\|_{L^2}^2
	= (2\pi)^{-d/2} \int_{\R^d} g(\eta) |P_\eta \xi|^2\, |\widehat{ \rho_t e^{-i\xi\cdot x}} (\eta)|^2\, \d \eta\\
	& = (2\pi)^{-d/2} \int_{\R^d} g(\eta) |P_\eta \xi|^2 |\hat\rho_t(\xi+\eta)|^2\, \d \eta.
\end{split}\end{equation}
Combining \eqref{eq:uniqueness-linear-proof-1} with \eqref{eq:uniqueness-linear-proof-2} and \eqref{eq:uniqueness-linear-proof-3}, we obtain
\begin{equation*}
	\d | \hat{\rho}_t(\xi)|^2
	= \d M_t(\xi) -2 (\kappa |\xi|^2 + \nu |\xi|^\beta)| \hat{\rho}_t(\xi)|^2\,\d t + (2\pi)^{-d/2} \int_{\R^d} g(\eta) |P_\eta \xi|^2 |\hat\rho_t(\xi+\eta)|^2\, \d \eta \,\d t,
\end{equation*}
where $M_t(\xi)$ is the martingale part.
Integrating over $(0,\tau_N)$ and taking expectation, which removes the martingale term, we arrive at
\begin{align*}
	0 & \leq \E\big[|\hat\rho_{\tau_N}(\xi)|^2\big]\\
	& = -2(\kappa|\xi|^2 +\nu |\xi|^\beta) \E\bigg[\int_0^{\tau_N}\! |\hat\rho_t(\xi)|^2\, \d t\bigg] + (2\pi)^{-d/2}\!\! \int_{\R^d} g(\eta) |P_{\eta}\xi|^2 \E\bigg[ \int_0^{\tau_N}\! |\hat\rho_t(\xi+\eta)|^2\, \d t\bigg] \d \eta\\
	& \leq -2\kappa |\xi|^2 A^N(\xi) + (2\pi)^{-d/2}\int_{\R^d} g(\eta) |P_{\eta}\xi|^2 A^N(\xi+\eta)\, \d \eta,
\end{align*}
where we used the fact that $\rho_0\equiv 0$.
Since $2\kappa I_d = Q(0)$, it holds
\begin{equation*}
	2\kappa |\xi|^2
	= \xi\cdot Q(0)\xi
	= (2\pi)^{-d/2} \int_{\R^d} g(\eta) |P_\eta \xi|^2\, \d \eta
\end{equation*}
and so finally we arrive at
\begin{equation}\label{eq:spectral-density-inequality}
	0 \leq (2\pi)^{-d/2} \int_{\R^d} g(\eta) |P_\eta\xi|^2 \, [A^N(\xi+\eta)-A^N(\xi)]\, \d \eta
	\quad \forall\, \xi\in\R^d.
\end{equation}
\textit{Step 2: Conclusion.}
Recall that for any fixed $t$ such that $\rho_t\in L^1_x$, by the Riemann-Lebesgue lemma it holds $|\hat{\rho}_t(\xi)|\to 0$ as $|\xi|\to\infty$; combined with the definition of $\tau^N$ and dominated convergence, this implies
\begin{equation}\label{eq:uniqueness-linear-proof-4}
	\lim_{|\xi|\to\infty} A^N(\xi)
	= \lim_{|\xi|\to\infty} \E\bigg[ \int_0^{\tau_N} |\hat\rho_t(\xi)|^2\, \d t \bigg]
	= 0.
\end{equation}
Together with the continuity of $A^N$, \eqref{eq:uniqueness-linear-proof-4} implies that $A^N$ admits a maximum, realized at some point $\bar{\xi}$. But then evaluating \eqref{eq:spectral-density-inequality} at $\bar{\xi}$ we find
\begin{equation*}
	0 \leq \int_{\R^d} g(\eta) |P_\eta \bar{\xi}|^2 [A^N(\bar\xi+\eta)-A^N(\bar\xi)]\, \d \eta
	\leq 0,
\end{equation*}
which implies that the integrand must be Lebesgue a.e. $0$; since $g$, $P_\eta$ and $A^N$ are all continuous, it must be identically zero everywhere. There are now three possible cases.

\textit{Case 1.}
It holds $\bar\xi=0$.
Recall that $\hat{\rho}_t(0)=(2\pi)^{-d/2} \int \rho_t(x)\, \d x$; since $\rho$ is an $L^1_x$-valued solution to \eqref{eq:linear-spde-no-drift} with $\rho_0= 0=f$, by applying Remark \ref{rem:defn-linear-weak-solution} with test function $\varphi\equiv 1$, we can deduce that $\P$-a.s. $\hat{\rho}_t(0)= \hat\rho_0(0)=0$ for all $t\in [0,\tau]$.
As a consequence, $A^N(0)=0$; recalling that by definition $A^N(\xi)\geq 0$ and we are assuming it achieves maximum at $\bar{\xi}=0$, it follows that $\P$-a.s. $\int_0^{\tau_N} |\hat{\rho}_t(\xi)|^2 \d t=0$ for all $\xi\in\R^d$.
By the continuity of $\xi\mapsto \hat\rho_t(\xi)$ at Lebesgue a.e. $t$ (coming from $\rho_t\in L^1_x$) and the weak continuity of $t\mapsto \rho_t$ (coming from Definition \ref{defn:linear-weak-solution}), we can conclude that $\P$-a.s. $\rho_t\equiv 0$ for all $t\in [0,\tau^N]$.

Hence, we may assume $\bar\xi\neq 0$ in the sequel.

\textit{Case 2.} There exists $\bar{\eta}\neq 0$ such that $g(\bar{\eta}) |P_{\bar{\eta}} \bar{\xi}|^2\neq 0$.
Then it must hold $A^N(\bar{\xi}+\bar{\eta})=A^N(\bar{\xi})$, namely $\bar{\xi}+\bar{\eta}$ also realizes the maximum of $A^N$; observing that
\begin{equation*}
	g(\bar{\eta})\, |P_{\bar{\eta}} \bar{\xi}|^2
	= g(\bar{\eta})\, |P_{\bar{\eta}} (\bar{\xi}+\bar{\eta})|^2
\end{equation*}
we can now repeat the argument to deduce
$A^N(\bar{\xi}+2\bar{\eta})=A^N(\bar{\xi}+\bar\eta)$.
Iterating, it holds $A^N(\bar{\xi}+m\bar{\eta})=A^N(\bar{\xi})$ for all $m\in\N$, which implies
\begin{equation*}
	A^N(\bar{\xi})
	= \sup_{\eta\in\R^d} A^N(\eta)
	= \lim_{m\to\infty} A^N(\bar{\xi}+m\bar\eta)=0
\end{equation*}
where again we used \eqref{eq:uniqueness-linear-proof-4}.
The conclusion now follows by arguing as in the previous case.

\textit{Case 3.} It holds $g(\eta) |P_{\eta} \bar{\xi}|^2 = 0$ for all $\eta$.
In this case, since $P_{\eta} \bar\xi \neq 0$ for Lebesgue a.e. $\eta\in\R^d$ and $g$ is continuous, we deduce that $g\equiv 0$; namely, the noise $W$ is the trivial $0$ noise.
But then $\rho$ is a weak $L^1_x$-valued solution on $[0,\tau^N]$ of the deterministic PDE
$\partial_t \rho = -\nu \Delta^\beta \rho $
starting at $\rho_0=0$, for which uniqueness is classical.
\end{proof}

\begin{remark}\label{rem:michele}
The above proof reveals an interesting structure concerning the evolution of the energy spectrum of solutions, similar to the ones arising in weak turbulence \cite{EscVel}.
For the sake of simplicity, we restrict to $\nu=0$.
Given a solution $\rho$, set $a_t(\xi)=\E[|\hat{\rho}_t(\xi)|^2]$; then going through the same computations as in Step 1 of the proof, one arrives at an integral system of ODEs for $\{a_t(\xi)\}_{t,\xi}$ of the form
\begin{equation}\label{eq:pde-energy-density}
	\frac{\d}{\d t} a_t(\xi)
	= (2\pi)^{-d/2} \int_{\R^d} g(\eta) |P_\eta\xi|^2 \, \big( a_t(\xi+\eta)-a_t(\xi)\big)\, \d \eta.
\end{equation}
By testing \eqref{eq:pde-energy-density} against a smooth function $\psi$ of $\xi$, after changing variable one finds
\begin{equation*}
	\frac{\d}{\d t} \langle a_t, \psi \rangle
	= (2\pi)^{-d/2} \int_{\R^{2d} } g(\eta-\xi) |P_{\eta-\xi} \xi|^2\, \big( a_t(\eta)-a_t(\xi)\big) \psi(\xi)\, \d \xi \d \eta.
\end{equation*}
Define the non-negative kernel $K(\eta,\xi):=(2\pi)^{-d/2} g(\eta-\xi) |P_{\eta-\xi} \xi|^2$; by the radial symmetry of $g$ and the definition of $P_\eta$ as an orthogonal projection, it follows that $K$ is symmetric, i.e. $K(\eta,\xi)=K(\xi,\eta)$.
By symmetry it then holds
\begin{equation}\label{eq:pde-density-symmetry}\begin{split}
	\frac{\d}{\d t} \langle a_t,\psi \rangle
	& = -\frac{1}{2} \int_{\R^{2d}} K(\eta,\xi) \big( a_t(\eta)-a_t(\xi)\big) \big( \psi(\eta)-\psi(\xi)\big)\, \d\xi\d\eta\\
	& = \int_{\R^{2d}} K(\eta,\xi)\, a_t(\xi) \big( \psi(\eta)-\psi(\xi)\big)\, \d\xi\d\eta.
\end{split}\end{equation}
The second equation in \eqref{eq:pde-density-symmetry} can be thought of as an alternative definition of weak solutions to the integral system \eqref{eq:pde-energy-density}, where now equality is not enforced pointwise in $\xi$ but only by testing against $\psi$.
Instead the first equation in \eqref{eq:pde-density-symmetry} shows that this system is \emph{monotone}; formally, by taking $\psi=a_t$, one would obtain
\begin{equation*}
	\frac{\d}{\d t} \| a_t\|_{L^2}^2
	= -\frac{1}{2} \int_{\R^{2d}} K(\eta,\xi)\, |a_t(\eta)-a_t(\xi)|^2\, \d \xi\d\eta \leq 0
\end{equation*}
which also provides an alternative (formal) route to establishing uniqueness of solutions.
\end{remark}

We are now ready to provide a complete existence and uniqueness statement, which includes as a particular subcase the main Theorem \ref{thm:main-linear-intro} from the introduction.

\begin{theorem}\label{thm:main-linear-extended}
Let $p\in[1,\infty]$, $\nu\geq 0$ and $\beta\in (0,2)$; let $W$ satisfy Assumption \ref{ass:covariance-basic} and let $f:\Omega\times \R_{\geq 0}\to L^1_x\cap L^p_x$ be a predictable process such that
\begin{equation*}
	\E\bigg[ \Big| \int_0^T \| f_s\|_{L^1_x\cap L^p_x}\, \d s \Big|^2 \bigg] <\infty
	\quad \forall\, T>0.
\end{equation*}
Then for any deterministic $\rho_0\in L^1_x\cap L^p_x$, there exists a global probabilistically strong, analytically weak $L^1_x\cap L^p_x$-valued solution to \eqref{eq:linear-spde-no-drift}, which satisfies the $\P$-a.s. pathwise bound
\begin{equation}\label{eq:apriori-linear-thm}
	\sup_{s\in [0,t]} \| \rho_t\|_{L^1_x\cap L^p_x}
	\leq \| \rho_0\|_{L^1_x\cap L^p_x} + \int_0^t \| f_s\|_{L^1_x\cap L^p_x}\, \d s
	\quad \forall\, t\geq 0.
\end{equation}
Moreover pathwise uniqueness and uniqueness in law hold among all solutions to \eqref{eq:linear-spde-no-drift} satisfying $\rho\in L^1_t L^1_x$ $\P$-a.s.
Finally, given two solutions $\rho^i$, $i=1,\,2$, defined on the same probability space and associated to different data $(\rho_0^i, f^i)$, we have the $\P$-a.s. stability estimate
\begin{equation}\label{eq:stability-linear-thm}
	\sup_{s\in [0,t]} \| \rho^1_s -\rho^2_s \|_{L^1_x\cap L^p_x}
	\leq \| \rho^1_0-\rho^2_0\|_{L^1_x\cap L^p_x} + \int_0^t \| f^1_s-f^2_s\|_{L^1_x\cap L^p_x}\, \d s
	\quad \forall\, t>0.
\end{equation}
\end{theorem}

\begin{proof}
The existence of solutions satisfying the pathwise bound \eqref{eq:apriori-linear-thm} follows from Proposition \ref{prop:strong-existence-linear}, while pathwise uniqueness is proved in Proposition \ref{prop:pathwise-uniqueness-linear};
strong existence and pathwise uniqueness imply uniqueness in law.
Given two solutions $\rho^i$ as above, by linearity of the SPDE their difference $\rho^1-\rho^2$ solves the equation associated to $(\rho_0^1-\rho_0^2, f^1-f^2)$;
uniqueness of solutions then implies that $\rho^1-\rho^2$ must also satisfy the pathwise bound \eqref{eq:apriori-linear-thm}, yielding the stability estimate \eqref{eq:stability-linear-thm}.
\end{proof}

We conclude this section with an interesting byproduct of Proposition \ref{prop:pathwise-uniqueness-linear}.

\begin{remark}\label{rem:invariant-measure-linear}
Consider the SPDE \eqref{eq:intro-linear-spde} with $f\equiv 0$, $\nu=0$ and nontrivial noise, i.e. $g\neq 0$; define the Banach space $E=\big\{ \rho \in L^1_x \, : \, \int_{\R^d} \rho(x)\,\d x=0 \big\}$.
Since the noise $W$ has independent increments, Theorem \ref{thm:main-linear-extended} and standard arguments imply that the SPDE gives rise to a Feller semigroup on $E$, given by
\begin{equation*}
	P_t F(\rho_0):= \E[F(\rho_t)]
\end{equation*}
where $F\in C_b(E)$ and $\rho_t$ is the unique solution starting from $\rho_0$.
Therefore it makes sense to talk about invariant measures associated to $\{P_t\}_{t\geq 0}$.

Suppose now that there existed an invariant measure $\mu$ on $E$, with finite second moment $\int_E \| \rho\|_E^2\, \mu(\d \rho)$.
Then setting $a(\xi):=\int_{E} |\hat\rho(\xi)|^2 \mu(\d \rho)$ and applying \eqref{eq:pde-energy-density}, by invariance of the measure one would obtain
\begin{equation*}
	\int_{\R^d} g(\eta) |P_\eta \xi|^2\, \big( a(\xi+\eta)-a(\xi)\big)\, \d \eta =0
	\quad \forall\, \xi\in \R^d
\end{equation*}
which going through the same arguments as in the proof of Proposition \ref{prop:pathwise-uniqueness-linear} ultimately results in $a\equiv 0$ (observe that, since $g\neq 0$ and we are on $E$, Cases 1. and 3. do not apply).
In other terms, the only invariant measure with finite second moment for \eqref{eq:intro-linear-spde} supported on $E$ is expected to be the Dirac centered at the trivial stationary solution $v=0$.

The situation can be drastically different for measures supported on larger classes of distributions; a formal (Gibbs type) invariant measure for \eqref{eq:intro-linear-spde} with $f\equiv 0$ and $ \nu=0$ is the white noise measure, corresponding to $a(\xi)=c$ for some constant $c>0$.
In this direction, for similar (nonlinear) systems, let us mention the classical work \cite{AlbCru} and more recently in the presence of transport noise \cite{FlaLuo2019}.

We expect uniqueness of the invariant measure $\mu=\delta_0$ in the setting of $E$-valued solutions to be very important in understanding the long-time behaviour of the semigroup $P_t$;
we leave this topic for future investigations.
\end{remark}

\section{Proof of Theorem \ref{thm:main-theorem}}\label{sec:wellposedness-nonlinear}

In this section we work exclusively on $\R^2$.
We split the proof in two parts: we first show the existence of probabilistically weak solutions to \eqref{eq:intro-abstract-spde} by the classical compactness method, then we apply a Girsanov transform argument to infer their uniqueness in law.

\subsection{A priori estimates and weak existence}\label{subsec:apriori-nonlinear}

We start by establishing suitable a priori estimates for smoothened SPDEs. Recall that $\cS$ is the space of Schwartz functions.

\begin{proposition}\label{prop:apriori-estimates-nonlinear}
Let $\omega_0:\R^2\to\R$ and $f:\R_{\geq 0} \times \R^2\to\R$ be deterministic smooth functions, such that $\int_{\R^2} \omega_0(x)\, \d x =0$\footnote{Since $\omega_0$ is smooth, $\int_{\R^2} \omega_0(x)\, \d x =0$ is the same as requiring $\omega_0\in \dot H^{-1}_x$, see Lemma \ref{lem:homogeneous-H1} in Appendix \ref{app:useful}.};
let $W$ be a smooth noise, $\mathcal{R}$ be the Fourier multiplier associated to some $r\in\cS$.
For any $\nu\geq 0$, $\eps>0$ and $\beta\in (0,2)$, consider the SPDE
\begin{equation}\label{eq:nonlinear-regularized-spde}
	\d \omega+ (\mathcal{R}\, \curl^{-1}\omega)\cdot \nabla\omega\,\d t + \circ\d W \cdot \nabla\omega
	= [f - \nu\Lambda^\beta \omega +\eps\nu\Delta \omega]\, \d t,
	\quad \omega\vert_{t=0}=\omega_0.
\end{equation}
Then there exists a unique global regular strong solution to \eqref{eq:nonlinear-regularized-spde} and for any $p\in [1,\infty]$, $\P$-a.s. it holds
\begin{equation}\label{eq:apriori-vorticity-Lp}
	\| \omega_t\|_{L^1_x\cap L^p_x}
	\leq \| \omega_0\|_{L^1_x\cap L^p_x} + \int_0^t \| f_s\|_{L^1_x\cap L^p_x}\, \d s
	\quad \forall\, t\geq 0,
\end{equation}
as well as
\begin{equation}\label{eq:apriori-vorticity-L2}
	\| \omega_t\|_{L^2_x}^2 + 2\nu \int_0^t \| \omega_s\|_{\dot H^{\beta/2}_x}^2\, \d s
	\leq \bigg( \| \omega_0\|_{L^2_x} + \int_0^t \| f_s\|_{L^2_x}\, \d s\bigg)^2
	\quad \forall\, t\geq 0.
\end{equation}
Moreover for any finite $T>0$, there exists a constant $C$, depending on $T$, $\|f\|_{L^1_t (L^2_x\cap \dot{H}^{-1}_x)}$, $\|r\|_{L^\infty_x}$ and $\| g\|_{L^1_x\cap L^\infty_x}$, increasing in all its entries, such that
\begin{equation}\label{eq:apriori-energy}
	\E\bigg[ \sup_{t\in [0,T]} \| \omega_t\|_{\dot H^{-1}_x}^2 \bigg]
	\leq C \big(1+ \|\omega_0\|_{\dot H^{-1}_x}^2 \big).
\end{equation}
\end{proposition}

\begin{proof}
Let us first discuss strong existence and uniqueness of regular solutions for the system \eqref{eq:nonlinear-regularized-spde}.
Recalling that $\curl^{-1}\omega=K\ast\omega$ for the Biot--Savart kernel $K$, it holds $\cR\, \curl^{-1} \omega= \tilde K\ast \omega$ for $\tilde K = \cR K= r\ast K$, which is now a kernel in $C_b^\infty$ thanks to the assumption $r\in \cS$ and properties of Biot--Savart kernel.

In the case $\nu=0$, strong well-posedness then follows from the results of \cite[Theorem 13]{CogFla}, which also yield the representation\footnote{Technically speaking, in \cite{CogFla} the authors only consider the case $f\equiv 0$, but it is rather clear that the contraction mapping argument developed therein in Section 3.2, based on the representation \eqref{eq:proof-flow-representation} itself, extends easily in the presence of a smooth, deterministic forcing $f$.}
\begin{equation}\label{eq:proof-flow-representation}
	\omega_t (X^\omega_t(x))
	= \omega_0(x) + \int_0^t f_s (X^\omega_s (x))\, \d s,
\end{equation}
where $X^\omega$ is the flow associated to the SDE
\begin{equation*}
	\d X^\omega_t(x)
	= b^\omega( X^\omega_t(x))\, \d t + \sum_k \sigma_k(X^\omega_t(x))\circ \d W^k_t,
	\quad b^\omega:= \tilde K\ast \omega.
\end{equation*}
By the above observations $\tilde K\ast\omega\in C^\infty_b$, which by the results from \cite{Kunita} implies smoothness of the flow $X^\omega$; jointly with the regularity of $\omega_0$ and $f$, and the representation \eqref{eq:proof-flow-representation}, this implies the desired regularity of $\omega$.

In the case $\nu>0$, we can exploit the additional presence of $\eps\nu \Delta\omega$ to infer strong existence and uniqueness from the same references mentioned in the proof of Proposition \ref{prop:apriori-estimates-linear}, e.g. all the local monotonicity requirements from \cite{FlGaLu21} are satisfied; regularity of $\omega$ follows as before either from using parabolic bootstrap or the mild formulation.

Estimate \eqref{eq:apriori-vorticity-Lp} can then be proved as in Proposition \ref{prop:apriori-estimates-linear}, up to replacing the drift $b$ with $\cR\,\curl^{-1} \omega$, which is divergence free; similarly, estimate \eqref{eq:apriori-vorticity-L2} follows from Remark \ref{rem:energy-estim-linear-fractional-laplacian}.

Thus it only remains to show \eqref{eq:apriori-energy}; for simplicity, we will work on a fixed finite interval $[0,T]$ and use $\lesssim$ without expliciting the dependence on the parameters determining the constant $C$ from \eqref{eq:apriori-energy}.
Let $Q$ be the covariance function of $W$, $Q(0)=2\kappa I_d$; since everything is smooth, the solution $\omega$ to \eqref{eq:nonlinear-regularized-spde} satisfies the equivalent It\^o form
\begin{equation}\label{eq:nonlinear-regularized-spde-ito}
	\d \omega+ (\mathcal{R}\, \curl^{-1}\omega)\cdot \nabla\omega\,\d t + \d W \cdot \nabla\omega
	= [f - \nu\Lambda^\beta \omega +(\eps\nu+\kappa)\Delta \omega]\, \d t.
\end{equation}
Let us set $u=\curl^{-1} \omega$; by estimate \eqref{eq:apriori-vorticity-L2}, we have the $\P$-a.s. bound
\begin{equation}\label{eq:proof-nonlinear-pathwise-bound}
	\sup_{t\in [0,T]} \| \omega_t\|_{L^2_x}
	\lesssim 1+ \|\omega_0\|_{L^2_x}.
\end{equation}
By applying the same arguments as in the proof of  Lemma \ref{lem:identity-enstrophy} in the Appendix \ref{app:useful}, we have
\begin{equation}\label{eq:fourier-estimates}\begin{split}
	\| u_t \|_{L^2_x}
  	& = \| \nabla^\perp (-\Delta)^{-1} \omega_t\|_{L^2_x}
  	= \| \nabla (-\Delta)^{-1} \omega_t\|_{L^2_x}
  	= \| (-\Delta)^{-1/2} \omega_t\|_{L^2_x},\\
  	\| \omega_t\|_{L^2_x}
  	& = \| \nabla u_t\|_{L^2_x}
  	= \| \nabla^2 (-\Delta)^{-1} \omega_t\|_{L^2_x}.
\end{split}\end{equation}
By It\^o formula in Hilbert spaces, it holds
\begin{equation*}
	\d\| u_t \|_{L^2_x}^2
	= \d \| (-\Delta)^{-1/2} \omega_t \|_{L^2_x}^2
	= 2 \langle (-\Delta)^{-1} \omega_t , \d \omega_t \rangle + \d [(-\Delta)^{-1/2} \omega]_t;
\end{equation*}
let $W$ have chaos expansion \eqref{eq:noise-series-representation}, then by \eqref{eq:nonlinear-regularized-spde-ito} we have
\begin{equation*}
	\d [(-\Delta)^{-1/2} \omega]_t
	= \sum_{k\in\N} \|(-\Delta)^{-1/2}(\sigma_k \cdot \nabla\omega_t) \|_{L^2_x}^2\,\d t
	= \sum_{k\in\N} \|(-\Delta)^{-1/2}\nabla\cdot (\sigma_k\, \omega_t) \|_{L^2_x}^2\,\d t
\end{equation*}
where in the last passage we used $\nabla\cdot\sigma_k=0$.
Consequently,
\begin{align*}
  	\d\| u_t\|_{L^2_x}^2
  	& = 2\Big( -\<(-\Delta)^{-1} \omega_t, (\mathcal{R} u_t)\cdot \nabla\omega_t\> + \big\<(-\Delta)^{-1} \omega_t, f_t - \nu\Lambda^\beta \omega_t +(\eps\nu+\kappa) \Delta \omega_t \big\>\Big) \, \d t \\
  	& \quad + \sum_{k\in\N} \|(-\Delta)^{-1/2}\nabla\cdot (\sigma_k \,\omega_t) \|_{L^2_x}^2\,\d t - 2 \<(-\Delta)^{-1} \omega_t, \nabla\omega_t \cdot \d W_t\> \\
  	& =: 2 \big( I^1_t + I^2_t\big)\, \d t +I^3_t\, \d t +\d I^4_t.
\end{align*}
We estimate the terms $I^i$ separately.
For $I^1$, since $\nabla\cdot \cR u=0$, integrating by parts yields
\begin{align*}
	|I^1_t|
	& = |\<(\mathcal{R} u_t)\cdot \nabla (-\Delta)^{-1} \omega_t, \omega_t\>|
	\le \|\mathcal{R} u_t\|_{L^4_x} \|\nabla(-\Delta)^{-1} \omega_t\|_{L^4_x} \|\omega_t\|_{L^2_x};
\end{align*}
by Ladyzhenskaya's inequality and the properties of the multiplier $\mathcal{R}$, it then holds
\begin{align*}
	|I^1_t|
	& \lesssim \|\mathcal{R} u_t\|_{L^2_x}^{1/2}\, \|\nabla \mathcal{R} u_t\|_{L^2_x}^{1/2}\, \|\nabla(-\Delta)^{-1} \omega_t\|_{L^2_x}^{1/2}\, \|\nabla^2 (-\Delta)^{-1} \omega_t \|_{L^2_x}^{1/2}\, \|\omega_t\|_{L^2_x} \\
  	& \lesssim \| u_t\|_{L^2_x}^{1/2}\, \|\nabla u_t\|_{L^2_x}^{1/2}\, \|(-\Delta)^{-1/2} \omega_t\|_{L^2_x}^{1/2}\, \| \omega_t \|_{L^2_x}^{3/2} \\
  	& \lesssim \| u_t\|_{L^2_x} \|\omega_t \|_{L^2_x}^2,
\end{align*}
where we used \eqref{eq:fourier-estimates} several times.
Next, by fractional integration by parts it holds
\begin{align*}
	I^2_t
	& = \big\<(-\Delta)^{-1} \omega_t, f_t - \nu\Lambda^\beta \omega_t +(\eps\nu+\kappa) \Delta \omega_t \big\>\\
	& = \big\< (-\Delta)^{-1/2}\omega_t, (-\Delta)^{-1/2} f_t\big\> -\nu \|\Lambda^{(\beta-2)/2} \omega_t\|_{L^2_x}^2 - (\eps\nu+\kappa)\|\omega_t \|_{L^2_x}^2\\
	& \le |\<(-\Delta)^{-1/2} \omega_t, (-\Delta)^{-1/2} f_t\>|
	\le \|u_t\|_{L^2_x} \| f_t\|_{\dot H^{-1}_x}
\end{align*}
where the last step is due to Cauchy's inequality and \eqref{eq:fourier-estimates}.
Since $\varphi\mapsto (-\Delta)^{-1/2}\nabla\cdot\varphi$ is a bounded operator on $L^2_x$, it holds
\begin{equation*}
	I^3_t
	\lesssim \sum_k \|\sigma_k\, \omega_t \|_{L^2_x}^2
	\lesssim \|\omega_t \|_{L^2_x}^2,
\end{equation*}
where the last step follows by arguing as in the proof of Lemma \ref{lem:stoch-integr-basic}.
Combining the above estimates and using \eqref{eq:proof-nonlinear-pathwise-bound} overall give us
\begin{equation}\label{eq:proof-nonlinear-useful}\begin{split}
	\d\| u_t\|_{L^2_x}^2
	& \lesssim \big(\| u_t\|_{L^2_x} \|\omega_t\|_{L^2_x}^2 + \|u_t\|_{L^2_x} \|f_t\|_{\dot H^{-1}_x} + \|\omega_t \|_{L^2_x}^2\big)\,\d t - 2 \<(-\Delta)^{-1} \omega_t, \nabla\omega_t\cdot \d W_t\> \\
  	& \lesssim (1+\| u_t\|_{L^2_x})\big( 1+\|\omega_0\|_{L^2}^2 + \| f_t\|_{\dot H^{-1}_x} \big)\,\d t + 2 \<\omega_t\, \nabla (-\Delta)^{-1} \omega_t, \d W_t\>,
\end{split}\end{equation}
where in the last passage we integrated by parts in the stochastic integral, using the fact that $W$ is divergence free.
Integrating in $t$, taking supremum in time and then expectation we arrive at
\begin{align*}
	\E\bigg[\sup_{s\in [0,t]}\| u_s\|_{L^2_x}^2 \bigg]
	& \lesssim \| u_0\|_{L^2_x}^2+ \int_0^t \bigg(1+\E\bigg[\sup_{s\in [0,r]}\| u_s\|_{L^2_x}^2 \bigg] \bigg)\big( 1+\|\omega_0\|_{L^2_x}^2 + \| f_r \|_{\dot H^{-1}_x} \big)\,\d r\\
  	& \quad + \E\bigg[\sup_{s\in [0,t]} \Big| \int_0^s \<\omega_r \nabla (-\Delta)^{-1} \omega_r, \d W_r\> \Big| \bigg].
\end{align*}
For the stochastic integral, we can now apply Lemma \ref{lem:stoch-integr-basic} and Remark \ref{rem:stoch-integr-basic} to find
\begin{equation}\label{eq:proof-nonlinear-useful2}\begin{split}
	& \E\bigg[\sup_{s\in [0,t]} \Big| \int_0^s \<\omega_r \nabla (-\Delta)^{-1} \omega_r, \d W_r\> \Big| \bigg]
  	\lesssim \E\bigg[ \Big(\int_0^t \| \omega_r \nabla (-\Delta)^{-1} \omega_r\|_{L^1_x}^2 \d r \Big)^{1/2} \bigg]\\
  	& \lesssim \E\bigg[ \Big(\int_0^t \| \omega_r \|_{L^2_x}^2 \| u_r\|_{L^2_x}^2 \d r\Big)^{1/2} \bigg]
  	\lesssim (1+\| \omega_0\|_{L^2_x})\, \E\bigg[ \Big(\int_0^t \| u_r\|_{L^2_x}^2 \d r\Big)^{1/2} \bigg],
\end{split}\end{equation}
where as before we applied Cauchy's inequality, \eqref{eq:proof-nonlinear-pathwise-bound} and \eqref{eq:fourier-estimates}.
To sum up, we find
\begin{equation*}
	\E\bigg[\sup_{s\in [0,t]}\| u_s\|_{L^2_x}^2 \bigg]
	\lesssim \| u_0\|_{L^2_x}^2+ \int_0^t \bigg(1+\E\bigg[\sup_{s\in [0,r]}\| u_s\|_{L^2_x}^2 \bigg] \bigg)\big( 1+\|\omega_0\|_{L^2_x}^2 + \| f_r \|_{\dot H^{-1}_x} \big)\,\d r
\end{equation*}
and the desired estimate \eqref{eq:apriori-energy} follows from Gr\"onwall's inequality.
\end{proof}

We are now ready to show existence of weak solutions;
in the next proof, we will use crucially some topological properties of the Fr\'echet spaces $H^s_\loc$, see Lemma \ref{lem:Hs-loc}, Corollary \ref{cor:ascoli-Hs-loc} and Remark \ref{rem:Hs-loc-interpolation} in Appendix \ref{app:useful} for more details.

\begin{proposition}\label{prop:weak-existence-nonlinear}
Let $p\in [2,\infty]$, $\omega_0\in L^1_x\cap L^p_x\cap \dot H^{-1}_x$ and $f\in L^1_t (L^1_x\cap L^p_x\cap \dot H^{-1}_x)$ be deterministic functions;
let $Q$ be a covariance function satisfying Assumption \ref{ass:covariance-basic} and $\mathcal{R}$ be a Fourier multiplier asssociated to $r\in L^\infty_x$.
Then for any $\nu\geq 0$ and any $\beta\in (0,2)$, there exists a probabilistically weak solution $(\omega,W)$ to the SPDE
\begin{equation}\label{eq:nonlinear-general-spde}
	\d \omega+ (\mathcal{R}\, \curl^{-1} \omega)\cdot \nabla\omega\,\d t + \d W \cdot \nabla\omega
	= [f - \nu\Lambda^\beta \omega + \kappa\Delta \omega]\, \d t
\end{equation}
where the noise $W$ has covariance function $Q$; moreover $\omega$ satisfies the pathwise estimates \eqref{eq:apriori-vorticity-Lp}-\eqref{eq:apriori-vorticity-L2} and the moment estimate \eqref{eq:apriori-energy}, for the same constant $C$.
\end{proposition}

\begin{proof}
In light of the a priori estimates from Proposition \ref{prop:apriori-estimates-nonlinear}, we follow a compactness argument in the style of Flandoli--Gatarek \cite{FlaGat}.
As before, we fix a finite interval $[0,T]$.

First, consider smooth approximations $\omega^N_0,\, f^N$ and $W^N$ similarly as in the proof of Proposition \ref{prop:strong-existence-linear}; also consider a sequence $r^N\in C^\infty_c$ such that $|r^N|\leq |r|$ for all $N$, $r^N\uparrow r$ Lebesgue almost everywhere, and let $\cR^N$ denote the associated Fourier multipliers.
Set $K^N := \cR^N K\in C^\infty_b$ and let $\omega^N$ be the solutions to the mollified It\^o SPDEs
\begin{equation}\label{eq:nonlinear-approx-spde}
	\d \omega^N+ (K^N\ast\omega^N)\cdot \nabla\omega^N\,\d t + \d W^N \cdot \nabla\omega^N
  	= \big[f^N - \nu\Lambda^\beta \omega^N +(\eps\nu+\kappa_N)\Delta \omega^N \big]\, \d t
\end{equation}
with smooth initial data $\omega_0^N$, where $\kappa_N$ is given by the relation $Q^N(0)=2\kappa_N I_d$.
By Proposition \ref{prop:apriori-estimates-nonlinear}, the solutions $\omega^N$ exist strongly, in particular we can take them all defined on the same probability space;
moreover they satisfy the estimates \eqref{eq:apriori-vorticity-Lp}, \eqref{eq:apriori-vorticity-L2} and \eqref{eq:apriori-energy}, where we can take the r.h.s. to be independent of $N$, thanks to our choice of the approximations $(\omega_0^N, f^N, W^N, r^N)$.
The rest of the proof is divided in two main parts.

\textit{Step 1: Tightness.}
Our first task is to derive equicontinuity estimates and moment bounds for $\{\omega^N\}_N$ in $C^0_t H^{-s}_x$, for sufficiently small $s$;
combined with Corollary \ref{cor:ascoli-Hs-loc}, Remark \ref{rem:Hs-loc-interpolation} and the uniform pathwise bound for $\| \omega^N_t\|_{L^2_x}$, this will imply that their laws are tight in $C^0_t H^{-s}_\loc$ for all $s>0$.

To this end, we need to estimate the time integrals of terms appearing in \eqref{eq:nonlinear-approx-spde} one by one.
We start with the nonlinear term; to this end, set $u^N=\curl^{-1}\omega^N$, so that $K^N\ast \omega^N= \cR^N u^N$.
By the Sobolev embedding $H^{1/2}_x\hookrightarrow L^4_x$, its dual $L^{4/3}_x\hookrightarrow H^{-1/2}_x$ and H\"older's inequality, we have
\begin{align*}
	\big\| (K^N\ast \omega^N_t)\cdot\nabla \omega^N_t \big\|_{H^{-3/2}_x}
	& = \big\| \nabla \cdot ( (\cR^N u^N_t)\, \omega^N_t) \big\|_{H^{-3/2}_x}
	\leq \big\| (\cR^N u^N_t)\, \omega^N_t \big\|_{H^{-1/2}_x}\\
	& \leq \| \cR^N u^N_t\|_{L^4_x} \| \omega^N_t\|_{L^2_x}
	\lesssim \| \cR^N u^N_t\|_{L^4_x} ,
\end{align*}
where in the last step we applied the uniform in $N$ pathwise bound \eqref{eq:apriori-vorticity-L2}.
By arguing as in the estimates for $I^1$ from Proposition \ref{prop:apriori-estimates-nonlinear}, one then finds
\begin{align*}
	\| \cR^N u^N_t\|_{L^4_x}
	\lesssim \| u^N_t\|_{L^2_x}^{1/2} \, \| \omega^N_t\|_{L^2_x}^{1/2}
	\lesssim \| \omega^N_t\|_{\dot H^{-1}_x}^{1/2}
\end{align*}
which together with Sobolev embeddings and the moment estimates \eqref{eq:apriori-energy} yields
\begin{equation}\label{eq:equicontinuity-1}
	\sup_N\, \E\bigg[\, \Big\| \int_0^\cdot \big[(K^N\ast \omega^N_t)\cdot\nabla \omega^N_t \big]\, \d t \Big\|_{C^{1/2}_t H^{-3/2}_x}^2 \bigg]
	\lesssim \sup_N\, \E\bigg[\int_0^T \| \omega^N_t\|_{\dot H^{-1}_x}\, \d t\bigg]
	< \infty.
\end{equation}
For the stochastic integrals, since $W^N$ are divergence free, by Lemma \ref{lem:stoch-integr-basic} it holds
\begin{align*}
	\E\bigg[ \Big\| \int_s^t \nabla\omega^N_r \cdot \d W^N_r \Big\|_{H^{-1}_x}^{2p} \bigg]
	& = \E \bigg[ \Big\| \nabla\cdot \int_s^t \omega^N_r\, \d W^N_r \Big\|_{H^{-1}_x}^{2p} \bigg] \\
	& \lesssim \E \bigg[ \Big\| \int_s^t \omega^N_r\, \d W^N_r \Big\|_{L^2_x}^{2p} \bigg]\\
	& \lesssim \kappa_N^p\, \E\bigg[\Big( \int_s^t \| \omega^N_r\|_{L^2_x}^2\, \d r \Big)^p \bigg]
	\lesssim |t-s|^p
\end{align*}
where in the last passage we applied the uniform pathwise estimates on $\|\omega^N\|_{L^2_x}$ and the fact that $\kappa_N\leq \kappa$ by construction.
Combining the above estimate with Kolmogorov's continuity theorem, one can conclude that for any $\gamma<1/2$ and any $p\in [1,\infty)$ it holds
\begin{equation}\label{eq:equicontinuity-2}
	\sup_N \, \E\bigg[ \Big\| \int_0^\cdot \nabla\omega^N_r\cdot \d W^N_r\Big\|_{C^\gamma_t H^{-1}_x} ^p\bigg] <\infty.
\end{equation}
Next, again by the uniform pathwise bounds on $\| \omega^N\|_{L^2_x}$, it holds
\begin{equation}\label{eq:equicontinuity-3}
	\sup_N\, \E\bigg[ \Big\| \int_0^\cdot [-\nu\Lambda^\beta + (\eps \nu +\kappa_N)\Delta ] \omega^N_r \d r \Big\|_{C^1_t H^{-2}_x} \bigg]
	\lesssim \sup_N\, \E \bigg[ \int_0^T \| \omega^N_r\|_{L^2_x}\, \d r \bigg]
	< \infty
\end{equation}
Finally, by construction the sequence $f^N\to f$ in $L^1_t L^2_x$, which implies that that $\int_0^\cdot f^N_r\, \d r \to \int_0^\cdot f_r\, \d r$ in $C([0,T];L^2_x)$ and is equicontinuous therein; moreover by assumption $f^N$ and $f$ are all deterministic.

Writing the SPDEs \eqref{eq:nonlinear-approx-spde} in integral form and combining the last observation with estimates \eqref{eq:equicontinuity-1}, \eqref{eq:equicontinuity-2} and \eqref{eq:equicontinuity-3}, we can conclude that the family $\{\omega^N\}_N$ is equicontinuous in $C^0_t H^{-2}_x$, with suitable moment bounds; together with the pathwise estimate on $\sup_N \| \omega^N\|_{L^\infty_t L^2_x}$, by Remark \ref{rem:Hs-loc-interpolation} this implies tightness of the sequence in $H^{-s}_\loc$, for any $s>0$.

\textit{Step 2: Passage to the limit in a new probability space.}
The next step is very classical, therefore we mostly sketch it; for a thorough presentation in a similar setting, see for instance \cite[Proof of Theorem 2.2]{FlGaLu21a}.

By construction $W^N\to W$ in $L^2_\Theta C^0_t L^2_\loc$; together with Step 1, this implies tightness of the laws of $\{(\omega^N,W^N)\}_N$ in $C^0_t H^{-s}_\loc\times C^0_t L^2_\loc$.
By Prohorov's theorem and Skorohod's representation theorem, we can extract a (not relabelled) subsequence and construct a new probability space $(\tilde\Theta, \tilde{\mathbb F}, \tilde\P)$, on which there exists a sequence of processes $\{(\tilde \omega^N,\tilde W^N)\}_N$ satisfying:
\begin{itemize}
\item[i)] for any $N$, $(\tilde \omega^N, \tilde W^N)$ has the same law as  $(\omega^N, W^N)$;
\item[ii)] $\tilde\P$-a.s., $(\tilde \omega^N, \tilde W^N)\to (\tilde \omega,\tilde W)$ in $C^0_t H^{-s}_\loc\times C^0_t L^2_\loc$.
\end{itemize}
Point i) above implies that $\tilde\omega^N$ is a solution to \eqref{eq:nonlinear-approx-spde}, adapted to the filtration generated by $W^N$, satisfying the same pathwise and moment bounds \eqref{eq:apriori-vorticity-Lp}, \eqref{eq:apriori-vorticity-L2} and \eqref{eq:apriori-energy}; moreover, $\tilde W^N$ are again $Q^N$-Brownian motions.
By Point ii) and Remark \ref{rem:Hs-loc-interpolation}, it then follows that, on a set of full probability $\tilde\P$, $\tilde\omega^N_t\rightharpoonup \tilde \omega_t$ in $L^2_x$ for all $t\in [0,T]$; by lowersemicontinuity of the relevant norms, it follows that $\tilde\omega$ satisfies the bounds \eqref{eq:apriori-vorticity-Lp}, \eqref{eq:apriori-vorticity-L2} and \eqref{eq:apriori-energy} as well.
Moreover, since $\tilde W^N\to \tilde W$, $\tilde W$ is a $Q$-Brownian motion.
With slightly more effort, using Points i) and ii) above, one can additionally show that $\tilde W$ is a $\tilde\F_t$-Brownian motion, where $\tilde \F_t=\sigma(\tilde W_r, \tilde \omega_r: r\leq t)$;
we claim that $(\tilde\Omega, \tilde\F, \tilde \F_t, \tilde \P; \tilde \omega, \tilde W)$ is a desired weak solution to \eqref{eq:nonlinear-general-spde}.

This can be accomplished by applying Points i) and ii) and passing to the limit as $N\to\infty$ in the weak form of the SPDEs \eqref{eq:nonlinear-approx-spde} satisfied by $\tilde\omega^N$.
For simplicity, let us drop the tildes in the notation of processes, so that for any $\varphi\in C_c^\infty$, $\tilde\P$-a.s. for all $t\in [0,T]$, it holds
\begin{equation}\label{eq:nonlinear-approx-spde-new}\begin{split}
	\<\omega^N_t,\varphi\>
	& = \<\omega^N_0,\varphi\> + \int_0^t \big\<\omega^N_s, (\cR^N u^N_s) \cdot \nabla \varphi \big\>\,\d s + \int_0^t \big\<\omega^N_s, \d W^N_s\cdot\nabla \varphi \big\> \\
	& \quad + \int_0^t \big[ \<f^N_s, \varphi\> - \nu \<\omega^N_s, \Lambda^\beta \varphi \> +(\kappa_N+ \eps\nu) \< \omega^N_s, \Delta \varphi \> \big]\,\d s.
\end{split}\end{equation}
The reasoning concerning the last integral in \eqref{eq:nonlinear-approx-spde-new} is almost identical to the one in the proof of Proposition \ref{prop:strong-existence-linear}, similarly for the stochastic integrals.

The nonlinear terms $\<\omega^N_s, (\cR^N u^N_s) \cdot \nabla \varphi \>$ are the only ones requiring less standard treatment; here however, thanks to Point ii) above and Remark \ref{rem:Hs-loc-interpolation}, we can apply at any fixed $t\in [0,T]$ Lemma \ref{lem:convergence-nonlinearity} from Appendix \ref{app:useful} to deduce that $\P$-a.s.
\begin{align*}
	\big\<\omega^N_s, (\cR^N \curl^{-1} \omega^N_s) \cdot \nabla \varphi \big\>
	\to \big\<\omega_s, (\cR \, \curl^{-1} \omega_s)\cdot\nabla\varphi \big\>.
\end{align*}
By dominated convergence, we can then establish the same for the integrals in time, which completes the proof.
\end{proof}

\begin{remark}\label{rem:apriori-estimates-nonlinear}
Our existence result from Proposition \ref{prop:weak-existence-nonlinear} crucially requires $\omega_0\in L^2_x$, as it is based on the a priori estimates for $\|u\|_{L^2_x}$ coming from Proposition \ref{prop:apriori-estimates-nonlinear}.
In light of the deterministic weak existence results by Delort \cite{Delort} and Schochet \cite{Schochet}, where it is enough to require $\omega_0\in L^1_x\cap \dot H^{-1}_x$, this assumption might seem superfluous; indeed, a stochastic counterpart of their result can be found in \cite{BrzMau} (on the torus $\T^2$), under the additional regularity condition
  \begin{equation}\label{eq:additional-regularity}
  \sum_k \| \sigma_k\|_{C^1_x}^2<\infty
  \end{equation}
on the noise, cf. \cite[Condition 2.1]{BrzMau}.
However, in order to prove uniqueness in law, we will need to work with a rougher noise, which prevents us from obtaining the a priori estimates from \cite[Lemma 4.5]{BrzMau}.

This issue is not related to the nonlinear structure of the PDE, instead, it is specific to the (ir)regularity of the noise. The work \cite{CogMau} contains a key intuition on how to exploit it effectively to obtain improved weak existence results.
To illustrate the idea in a somewhat heuristic fashion, consider a solution $\rho$ to the simpler linear SPDE $\d \rho + \circ\d W\cdot\nabla\rho=0$, which we may write in It\^o form as
\begin{equation*}
\d \rho_t + \sum_k \sigma_k\cdot\nabla \rho_t\, \d B^k_t = \kappa \Delta \rho_t\, \d t = \frac{1}{2} \nabla\cdot( Q(0)\nabla\rho_t)\, \d t;
\end{equation*}
assume it admits a weak non-negative solution $\rho$, uniformly bounded in $L^1_x$.
Let $G$ be a \emph{smooth} convolutional kernel; by an application of It\^o's formula in $L^2_x$, it holds
\begin{align*}
	\d \langle G\ast \rho_t,\rho_t\rangle
	= 2 \langle G\ast\rho_t,\d \rho_t \rangle + \sum_k \langle G\ast(\sigma_k\cdot\nabla \rho_t),\sigma_k\cdot\nabla \rho_t\rangle.
\end{align*}
Observe that
\begin{align*}
	2 \langle G\ast\rho_t,\d \rho_t \rangle
	& = \langle G\ast \rho_t, \nabla\cdot (Q(0)\nabla\rho_t) \rangle\, \d t + \d M_t\\
	& = \int_{\R^{2d}} D^2 G(x-y): Q(0) \rho_t(x) \rho_t(y)\, \d x \d y + \d M_t,
\end{align*}
where $M$ is a suitable martingale term and the last passage comes from integration by parts; we employed the notation $A:B=\Tr (A^T B)=\sum_{ij} A_{ij} B_{ij}$.
Similarly, since $\sigma_k$ are divergence free, integrating by parts we find
\begin{align*}
	\sum_k \langle G\ast(\sigma_k\cdot\nabla \rho_t),\sigma_k\cdot\nabla \rho_t\rangle
	& = \sum_k \int_{\R^{2d}} G(x-y) \sigma_k(y)\cdot\nabla\rho_t(y) \sigma_k(x) \cdot\nabla \rho_t(x)\, \d x \d y\\
	& = - \sum_k \int_{\R^{2d}} D^2G(x-y): \big( \sigma_k(y)\otimes \sigma_k(x)\big)\, \rho_t(x) \rho_t(y)\, \d x \d y\\
	& = - \int_{\R^{2d}} D^2G(x-y) : Q(x-y) \rho_t(x)\rho_t(y)\, \d x \d y.
\end{align*}
Combining the above computations, taking expectations which removes the martingale part $M_t$, one finds
\begin{equation}\label{eq:convolutional-balance}
	\frac{\d}{\d t} \E[\langle G\ast \rho_t,\rho_t\rangle]
	= \E\bigg[\int_{\R^{2d}} H(x-y) \rho_t(x)\rho_t(y)\, \d x \d y\bigg]
	= \E[\langle H\ast \rho_t,\rho_t\rangle]
\end{equation}
for the new kernel $H(z):=(Q(0)-Q(z)):D^2 G(z)$.
We derived \eqref{eq:convolutional-balance} for smooth $G$, but the right-hand side of \eqref{eq:convolutional-balance} can be meaningful even for singular kernels, depending on the regularity of $Q$.
Arguing by approximations, one thus expects relation \eqref{eq:convolutional-balance} to be true whenever $G$, $H$ and $\rho$ are such that all terms appearing are well-defined.

Consider the case of the Green function $G=\Delta^{-1}$, $G(z)\sim |z|^{2-d}$ ($-\log |z|$ for $d=2$); then the left-hand side of \eqref{eq:convolutional-balance} would correspond to $\E \big[\| \rho_t\|_{\dot H^{-1}_x}^2 \big]$, while
\begin{align*}
	|H(z)| \leq |G(z)|\, |Q(0)-Q(z)| \lesssim |z|^{-d} |Q(0)-Q(z)|.
\end{align*}
Under \eqref{eq:additional-regularity}, Proposition \ref{prop:regularity-Q-sigma} implies that $|Q(z)-Q(0)| \lesssim \|Q \|_{W^{2,\infty}_x} |z|^2$ and consequently $H(z)\lesssim |z|^{2-d}$.
Since $\rho$ is non-negative, we overall expect an a priori estimate of the form
\begin{align*}
	\frac{\d}{\d t}\E\big[\| \rho_t\|_{\dot H^{-1}_x}^2\big]
	& \lesssim \| Q\|_{W^{2,\infty}_x} \E \int_{\R^{2d}} |x-y|^{2-d} \rho_t(x)\rho_t(y)\, \d x \d y \\
	&\sim \| Q\|_{W^{2,\infty}_x}\, \E \langle G \ast \rho_t, \rho_t\rangle
	= \| Q\|_{W^{2,\infty}_x}\, \E\big[\| \rho_t\|_{\dot H^{-1}_x}^2\big];
\end{align*}
the last two passages work for $d\geq 3$, while for $d=2$ there is a more direct bound with $\| Q\|_{W^{2,\infty}_x} \| \rho_t\|_{TV}^2$.
Overall, for $Q\in W^{2,\infty}_x$, which in particular holds true under \eqref{eq:additional-regularity}, one can expect to obtain closed estimates for $\E\big[ \| \rho_t\|_{\dot{H}^{-1}_x}^2 \big]$.

The key intuition from \cite{CogMau} is that, for rough Kraichnan noise (namely satisfying Assumption \ref{ass:covariance-basic} with $g(\xi)=(1+|\xi|^2)^{-(d+\gamma)/2}$, $\gamma\in (0,2)$), one can carefully analyze the kernel $H$ to show that (for $d=2$)
\begin{align*}
	|H(z)|= -c (1+|z|^2)^{-\gamma/2} + O(|z|^{-2}) \quad \text{as } |z|\to 0
\end{align*}
for some positive constant $c>0$; this allows to get an improved estimate of the form
\begin{align*}
	\frac{\d}{\d t}\E\big[\| \rho_t\|_{\dot H^{-1}_x}^2\big] + c\, \E\big[\| \rho_t\|_{H^{-\gamma/2}_x}^2\big]
	\lesssim \E\big[\| \rho_t\|_{\dot H^{-1}_x}^2\big]
\end{align*}
which yields closed estimates in $\dot H^{-1}_x$ and shows a regularization effect (represented by the term $\|\rho_t\|_{H^{-\gamma/2}_x}^2$) taking place, due to the irregularity of the noise.
\end{remark}

We presented the computations in Remark \ref{rem:apriori-estimates-nonlinear} in detail, in order to highlight the subtle and nontrivial effect in the well-posedness theory for the nonlinear SPDE, coming from the roughness of our noise.
Classical strategies like Yudovich's proof of uniqueness for $\omega_0\in L^1_x\cap L^\infty_x$, which was readapted to the stochastic setting in \cite{BrFlMa} (on $\T^2$), might not be available anymore; on the other hand, the rough Kraichnan regime ($\gamma\in (0,2)$) allows for new arguments compared to the standard setting.

We conclude this section with a consequence of Proposition \ref{prop:weak-existence-nonlinear}, which will be useful later.

\begin{corollary}\label{cor:stability-linear-energy}
Consider a deterministic sequence $\{(\rho_0^n, f^n)\}_n$ such that $\rho_0^n\to \rho_0$ in $L^1_x\cap L^2_x\cap \dot H^{-1}_x$ and $f^n\to f$ in $L^1_t(L^1_x\cap L^2_x\cap \dot H^{-1}_x)$; let $\nu$, $\beta$ and $W$ be fixed parameters taken as in Proposition \ref{prop:weak-existence-nonlinear}. For each $n$, consider the unique strong solution to the linear SPDE
\begin{align*}
	\d \rho^n + \d W\cdot\nabla \rho^n = [f^n-\nu \Lambda^\beta\rho^n+\kappa\Delta\rho^n]\, \d t;
\end{align*}
similarly, denote by $\rho$ the solution associated to $(\rho_0,f)$. Then for any $T\in (0,+\infty)$ it holds
\begin{align*}
	\lim_{n\to\infty} \E\bigg[\sup_{t\in [0,T]} \| \rho^n_t-\rho_t\|_{L^2_x}^2 + \sup_{t\in [0,T]} \| \rho^n_t-\rho_t\|_{\dot H^{-1}_x}^2 + \nu \int_0^T \| \rho^n_t -\rho_t \|_{\dot H^{\beta/2}_x}^2\, \d t\bigg] = 0.
\end{align*}
\end{corollary}

\begin{proof}
Strong existence and uniqueness of $\rho^n$, $\rho$ comes from Theorem \ref{thm:main-linear-extended}; by linearity of the SPDE, $\rho^n-\rho$ is the solution associated to $(\rho^n_0-\rho_0, f^n-f)$, therefore it suffices to show the statement for $\rho^n_0\to 0$ and $f^n\to 0$.
Remark \ref{rem:energy-estim-linear-fractional-laplacian} yields the stronger $\P$-a.s. pathwise estimate
\begin{equation}\label{eq:proof-cor-pathwise}
	\lim_{n\to\infty} \bigg[ \sup_{t\in [0,T]} \| \rho^n_t\|_{L^2_x}^2  + \nu \int_0^T \| \rho^n_t \|_{\dot H^{\beta/2}_x}^2\, \d t \bigg]
	\lesssim \lim_{n\to\infty} \bigg( \| \rho^n_0\|_{L^2_x} + \int_0^T \|f^n_t \|_{L^2_x}\, \d t \bigg)^2 = 0 ,
\end{equation}
so we only need to show convergence of the $\dot H^{-1}_x$-norm.
Next observe that, as the family $\| f^n\|_{L^1_t (L^2_x\cap \dot H^{-1}_x)}$ is uniformly bounded, estimate \eqref{eq:apriori-energy} yields a uniform bound
\begin{equation}\label{eq:proof-cor-unif-bound}
	\sup_{n\in\N}\, \E\bigg[\sup_{t\in [0,T]} \| \rho^n_t\|_{\dot H^{-1}_x}^2 \bigg] \leq C;
\end{equation}
going through the same computations as in the proof of Proposition \ref{prop:apriori-estimates-nonlinear}, this time for $u^n=\curl^{-1} \rho^n$, one then finds (cf. the first equation in \eqref{eq:proof-nonlinear-useful}, keeping in mind that the terms coming from the nonlinearity are absent here)
\begin{align*}
	\d\| \rho^n_t\|_{\dot H^{-1}_x}^2
	\lesssim \big(\|\rho^n_t\|_{\dot H^{-1}_x} \|f^n_t\|_{\dot H^{-1}_x} + \|\rho^n_t \|_{L^2_x}^2\big)\,\d t -2 \<(-\Delta)^{-1} \rho^n_t, \nabla\rho^n_t\cdot \d W_t\>.
\end{align*}
Integrating from 0 to $t$ and taking the supremum over $t\in [0,T]$ on both sides, estimating the stochastic integral as in \eqref{eq:proof-nonlinear-useful2}, we then arrive at
\begin{align*}
	\E\bigg[\sup_{t\in [0,T]} \| \rho^n_t\|_{\dot H^{-1}_x}^2\bigg]
	& \lesssim \| \rho^n_0 \|_{\dot H^{-1}_x}^2+ \E\bigg[\int_0^T \big( \|\rho^n_t\|_{\dot H^{-1}_x} \|f^n_t\|_{\dot H^{-1}_x} + \|\rho^n_t \|_{L^2_x}^2\big)\,\d t\bigg]\\
	& \quad + \Big(\| \rho^n_0\|_{L^2_x} + \int_0^T \|f^n_t\|_{L^2_x}\, \d t\Big) \E\bigg[ \Big(\int_0^T \| \rho^n_t\|_{\dot H^{-1}_x}^2\, \d t\Big)^{1/2} \bigg]\\
	& \lesssim_T \| \rho^n_0 \|_{\dot H^{-1}_x}^2+ \| \rho^n_0\|_{L^2_x} + \!\int_0^T\! \|f^n_t\|_{L^2_x \cap\dot H^{-1}_x} \,\d t + \!\bigg( \| \rho^n_0\|_{L^2_x} + \!\int_0^T\! \|f^n_t\|_{L^2_x}\, \d t \bigg)^2 ,
\end{align*}
where in the last passage we applied the pathwise bound \eqref{eq:proof-cor-pathwise} on $\| \rho^n_t\|_{L^2_x}$, the uniform moment bound \eqref{eq:proof-cor-unif-bound} on $\| \rho^n_t\|_{\dot H^{-1}_x}$ and the fact that $\rho^n_0$, $f^n$ are deterministic.
All terms on the r.h.s. of the last inequality above by assumption converge to $0$ as $n\to\infty$, yielding the conclusion.
\end{proof}

\subsection{Uniqueness in law and weak stability}\label{subsec:uniqueness-nonlinear}

We now turn to the proof of uniqueness in law; as mentioned in the introduction, it is based on the use of Girsanov transform, see \cite[Theorem 10.14]{DaPZab} for its classical statement on Hilbert spaces.
We shall start by some abstract considerations, presented in Propositions \ref{prop:weak-uniqueness-nonlinear} and \ref{prop:entropy-bounds} below, working for a general class of multipliers $\cR$; we then specialize to our equations of interest in Theorems \ref{thm:main-logeuler}-\ref{thm:main-hypoNS}, finally yielding the proof of Theorem \ref{thm:main-theorem}.

We need to introduce some conventions. Recall that, for a noise $W$ with covariance $Q$, the associated Cameron-Martin space is given by $\cH = \cQ^{1/2}(L^2_x)$, with norm $\|\varphi \|_\cH= \|\cQ^{-1/2} \varphi\|_{L^2_x}$.
Given a stochastic process $h$ $\P$-a.s. belonging to $L^2_t \cH$, we define $\cI(h)_\cdot:=\int_0^\cdot h_s\, \d s$; if moreover $h$ is $\F_t$-progressively measurable and $W$ is an $\F_t$-Brownian motion, we denote by $\cE_T(h)$ the associated exponential martingale evaluated at time $T$, namely
\begin{equation}\label{eq:exponential-martingale}
	\cE_T(h):= \exp\bigg(\int_0^T \langle \cQ^{-1} h_s, \d W_s\rangle_{L^2_x} - \frac{1}{2} \int_0^T \| h_s \|_{\cH}^2\, \d s \bigg).
\end{equation}
From now on, we will always assume $\omega_0$ and $f$ to be deterministic and satisfying the assumptions of Proposition \ref{prop:weak-existence-nonlinear} for some $p\in [2,\infty]$.
Correspondingly, we work with a weak solution $(\Theta,\F,\F_t,\P;\omega,W)$ to the nonlinear SPDE \eqref{eq:nonlinear-general-spde}, defined on the whole finite interval $[0,T]$, whose existence is thus granted; on the same space, thanks to Theorem \ref{thm:main-linear-extended}, we can construct a strong solution $\rho$ to the linear SPDE \eqref{eq:linear-spde-no-drift}, which would correspond to the choice $\cR\equiv 0$.

We denote by $C^w_t L^2_x=C^w([0,T];L^2_x)$ the Banach space of measurable bounded maps $\varphi:[0,T]\to L^2_x$, which are continuous in the weak topology of $L^2_x$; it is a separable Banach space endowed with supremum norm $\| f\|_{C^w_t L^2_x}:=\sup_{t\in [0,T]} \| f_t\|_{L^2_x}$.
Whenever talking about the law of either $\omega$ or $\rho$, this will be meant in the sense of being a $C^w_t L^2_x$-valued random variable.

\begin{proposition}\label{prop:weak-uniqueness-nonlinear}
Let $\omega_0$, $f$, $\nu$, $\beta$, $Q$, $\cR$ satisfy the assumptions of Proposition \ref{prop:weak-existence-nonlinear} for some $p\in [2,\infty]$; let $\omega$ be a weak solution to \eqref{eq:nonlinear-general-spde}, defined on the interval $[0,T]$.
If
\begin{equation}\label{eq:weak-uniqueness-assumption1}
\int_0^T \| \mathcal{R}\, \curl^{-1}\omega_s\|_{\cH}^2\, \d s <\infty \quad \P\text{-a.s.},
\end{equation}
then the law of $\omega$ is uniquely determined. If additionally $\rho$ satisfies
\begin{equation}\label{eq:weak-uniqueness-assumption2}
\int_0^T \| \mathcal{R}\, \curl^{-1}\rho_s\|_{\cH}^2\, \d s <\infty \quad \P\text{-a.s.},
\end{equation}
then $\mu^{NL}:=\L(\omega)$ and $\mu^L:=\L(\rho)$ are equivalent and  $\P$-a.s. it holds
\begin{equation}\label{eq:formula-density-girsanov}
	\frac{\d \mu^{NL}}{\d \mu^L}(\rho) = \E\big[ \cE_T \big(\cR\, \curl^{-1} \rho\big) \big\vert \F^\rho_T \big], \quad
	\frac{\d \mu^L}{\d \mu^{NL}}(\omega) = \E\big[ \cE_T \big(-\cR\, \curl^{-1} \omega\big) \big\vert \F^\omega_T \big]
\end{equation}
where $\F^\rho$ and $\F^\omega$ denote respectively the natural filtrations of $\rho$ and $\omega$.
\end{proposition}

\begin{proof}
The statement is just a convenient rephrasing in our setting of the content of Propositions 9.1 (concerning uniqueness in law) and 9.2 (concerning formula \eqref{eq:formula-density-girsanov}) from \cite{Ferrario}, although applied to a different SPDE; the argument in \cite{Ferrario} is applied to many equations of interest, including SDEs with multiplicative noise, clearly showing the robustness of the method. As the proof here is almost identical to \cite{Ferrario}, rather than giving full details, let us explain the main ideas and difficulties.

Since $\omega$ is a weak solution to \eqref{eq:nonlinear-general-spde}, upon regrouping the terms in the SPDE it holds
\begin{equation}\label{eq:weak-uniqueness-proof1}
	\d \omega_t + \d \Big( W_t + \cI(\mathcal{R}\, \curl^{-1} \omega)_t\Big)\cdot \nabla\omega
	= [f_t - \nu\Lambda^\beta \omega_t + \kappa\Delta \omega_t]\, \d t;
\end{equation}
set $\tilde{W}_t := W_t + \cI(\mathcal{R}\, \curl^{-1} \omega)_t$.
If one were allowed to apply Girsanov's theorem to $\tilde W$, namely to verify that $\E_{\P} [\cE_T (-\cR\, \curl^{-1} \omega )]=1$, then we could construct a new probability measure $\Q\ll \P$ on $\Theta$ with Radon-Nikodym derivative $\d \Q/\d \P= \cE_T (-\cR\, \curl^{-1} \omega)$ such that $\L_{\Q}(\tilde W) = \L_{\P}(W)$ and relation \eqref{eq:weak-uniqueness-proof1} still holds $\Q$-a.s.

By \eqref{eq:weak-uniqueness-proof1}, $\omega$ actually solves the \emph{linear} SPDE \eqref{eq:linear-spde-no-drift} driven by $\tilde W$; by Theorem \ref{thm:main-linear-extended} its law is uniquely determined and coincides with $\mu^L$. Overall, assuming $\E_{\P} [\cE_T (-\cR\, \curl^{-1} \omega )]=1$, we could conclude that $\L_\Q(\tilde W, \omega)=\L_\P(W,\rho)$ with an explicit formula for $\d \Q/\d \P$; the expression for $\d \mu^L/\d \mu^{NL}$ as given in \eqref{eq:formula-density-girsanov} then follows from the general formula for the density of $\d (\mu\circ \psi^{-1})/\d (\nu\circ \psi^{-1})$ given $\d \mu/\d \nu$, valid for any measurable map $\psi$ and any $\mu\ll \nu$.
The equivalence of $\mu^{NL}$ and $\mu^L$ would then follow from similar considerations, upon inverting the roles of $\omega$ and $\rho$, namely starting from the strong solution $\rho$ on $(\Theta,\P)$ and then constructing a new $\tilde\Q\ll \P$ such that $\L_{\P}(\omega)=\L_{\tilde\Q}(\rho)$.

The problem is that in practice conditions like $\E_{\P} [\cE_T (-\cR\, \curl^{-1} \omega )]=1$ are very hard to verify, usually requiring criteria like Novikov's one to be verified and involving exponential moment bounds, which are not available in our setting.
If one is only interested in \emph{uniqueness in law}, rather than all the consequences guaranteed by Girsanov's theorem, one can rather introduce an increasing sequence of stopping times $\tau^n\uparrow T$ such that the stopped processes $\omega^n_\cdot=\omega_{\cdot\wedge \tau^n}$ satisfy the Novikov condition, apply the Girsanov transform therein, and finally observe that $\L_{\P}(\omega)$ is uniquely determined by the knowledge of $\{\L_{\P} (\omega^n)\}_n$.
This is loosely the idea of proof of \cite[Proposition 9.1]{Ferrario}; the use of suitable stopping times in the Girsanov theorem already appeared in \cite[Proposition V.3.10]{KarShr} and later in  \cite{Lehec}.
To be precise, the correct replacement for condition [A4] in \cite[Proposition 9.1]{Ferrario} is given respectively by the strong existence and pathwise uniqueness for the SPDE for $\rho$ (which plays the role of $Z$), guaranteed by Theorem \ref{thm:main-linear-extended} (which upon minor modifications holds true also on $[t_0,T]\subset [0,T]$);
instead [A5] is replaced by the lowersemicontinuity (thus measurability) of the map $\varphi\mapsto \| \cQ^{-1/2} \cR\, \curl^{-1} \varphi\|_{L^2}$ in $L^1_x\cap L^2_x\cap \dot H^{-1}_x$, once this is extended to take possibly value $+\infty$.

The proof of \cite[Proposition 9.2]{Ferrario} is more intricate and builds on the original one by Liptser and Shiryaev, cf. \cite[Chapter 7]{LipShi}, so we will not delve into too much details. Let us mention that the replacements for assumptions [A4] and [A5] are the same explained above, while conditions (9.9) and (9.10) from \cite{Ferrario} correspond to \eqref{eq:weak-uniqueness-assumption1}-\eqref{eq:weak-uniqueness-assumption2} here.
Essentially, already under solely \eqref{eq:weak-uniqueness-assumption1}, one can actually show that $\mu^{NL}\ll \mu^L$, with an explicit formula for the density $\d \mu^{NL}/\d \mu^L$; this in general needs not to be given by \eqref{eq:formula-density-girsanov}, rather $\int_0^T \< \cQ^{-1} \cR\, \curl^{-1} \omega, \d W_s \>$ should be replaced by a new term $\mathcal{J}_T(\rho)$ (cf. the formula for $\mathcal{J}_T(Z)$ in \cite[Section 9]{Ferrario}). Under the additional \eqref{eq:weak-uniqueness-assumption2}, the two formulas however coincide and all conclusions follow.
\end{proof}

\begin{remark}\label{rem:weak-uniqueness-local}
For simplicity, we stated Proposition \ref{prop:weak-uniqueness-nonlinear} in the case of solutions defined on $[0,T]$, but one can easily verify that the first part of the argument applies to local solutions; in particular, if $\omega$ is a weak solution such that $\int_0^\tau \| \cR\, \curl^{-1}\,\omega_s\|_{\cH}^2\, \d s <\infty$, for some stopping time $\tau$, then necessarily the law of the stopped process $\omega^\tau_\cdot=\omega_{\cdot\wedge \tau}$ is uniquely determined.
\end{remark}

In fact, upon further knowledge on the solutions $\omega$, $\rho$, one can further strengthen the previous result by providing \emph{entropy bounds}.
To this end, let us recall a few basic facts on entropy; see \cite{Lehec} and the references therein for more details.
Given a measurable space $(X,\cA)$ and two measures $\mu$, $\nu$ on it, the relative entropy of $\mu$ given $\nu$ is defined by
\begin{equation*}
	H(\mu |\nu) = \begin{cases}
	\int_X \log \big(\frac{\d \mu}{\d \nu}\big)\, \d \mu \quad & \text{if } \mu\ll \nu,\\
	+\infty & \text{otherwise}.
\end{cases}\end{equation*}
Entropy behaves well under transformations: for any measurable map $\psi :(X,\cA)\to(Y,\mathcal{B})$, it holds
\begin{equation}\label{eq:entropy-pushforward}
	H(\mu\circ \psi^{-1} | \nu\circ \psi^{-1}) \leq H(\mu| \nu).
\end{equation}

\begin{proposition}\label{prop:entropy-bounds}
Let $\omega$, $\rho$ be as in the first part of Proposition \ref{prop:weak-uniqueness-nonlinear} and suppose additionally that $\E\big[\int_0^T  \| \mathcal{R}\curl^{-1}\omega_s\|_{\cH}^2\, \d s \big] <\infty$; then
\begin{equation}\label{eq:entropy-bound-1}
	H\big(\mu^{NL}| \mu^L \big) \leq \frac{1}{2}\, \E\bigg[\int_0^T  \| \mathcal{R}\curl^{-1}\omega_s\|_{\cH}^2 \,\d s \bigg]< \infty.
\end{equation}
If additionally $\E\big[\int_0^T  \| \mathcal{R}\curl^{-1}\rho_s\|_{\cH}^2\, \d s \big] <\infty$, then
\begin{equation}\label{eq:entropy-bound-2}
	H\big(\mu^L| \mu^{NL} \big) \leq \frac{1}{2}\,\E \bigg[\int_0^T  \| \mathcal{R}\curl^{-1}\rho_s\|_{\cH}^2 \,\d s \bigg]< \infty.
\end{equation}
\end{proposition}

\begin{proof}
We present the proof of estimate \eqref{eq:entropy-bound-1}, the other being identical upon exchanging the roles of $\omega$ and $\rho$.

Assume for the moment that a stronger uniform bound $\sup_{s\in [0,T]} \| \mathcal{R}\curl^{-1}\omega_s\|_{\cH} \leq C$ is available, for some deterministic constant $C>0$.
Then applying Girsanov transform as in the proof of Proposition \ref{prop:weak-uniqueness-nonlinear}, we can construct a new probability $\Q\sim \P$ such that on $[0,T]$ it holds
\begin{equation*}
	\L_{\Q}\big( W+ \cI( \cR \curl^{-1}\omega) , \omega\big)
	= \L_{\P} ( W , \rho),
	\quad \frac{\d \Q}{\d \P} = \cE_T(-\cR \curl^{-1}\omega),
	\quad \frac{\d \P}{\d \Q} = \cE_T(-\cR \curl^{-1}\omega)^{-1}.
\end{equation*}
By virtue of \eqref{eq:entropy-pushforward}, it then holds
\begin{align*}
	H\big(\mu^{NL}| \mu^L \big)
	& \leq H(\P|\Q)
	= -\E_\P \big[ \log \cE_T(-\cR \curl^{-1}\omega) \big]\\
	& = \E_\P\bigg[ \int_0^T \langle \cQ^{-1} \cR \curl^{-1}\omega_s, \d W_s\rangle + \frac{1}{2} \int_0^T \| \cR \curl^{-1}\omega_s\|_{\cH}^2\, \d s \bigg]\\
	& = \frac{1}{2}\, \E_{\P} \bigg[\int_0^T  \| \mathcal{R}\curl^{-1}\omega_s\|_{\cH}^2\,\d s \bigg]
\end{align*}
yielding the conclusion for such $\omega$.
The general case then follows by using stopping time approximations and the lower-semicontinuity of the entropy, see \cite[Proposition 1]{Lehec} for more details.
\end{proof}

We are now ready to specialize the abstract Propositions \ref{prop:weak-uniqueness-nonlinear} and \ref{prop:entropy-bounds} to our cases of interest.
In the next statements, $\cP(C^w_t L^2_x )$ denotes the set of all probability measures on $C^w_t L^2_x = C^w([0,T];L^2_x)$, endowed with the topology of weak convergence of measures.

\begin{theorem}\label{thm:main-logeuler}
Let $\omega_0\in L^1_x\cap L^2_x\cap \dot H^{-1}_x$ and $f\in L^1_t (L^1_x\cap L^2_x\cap \dot H^{-1}_x)$ be deterministic, $\gamma>1/2$; let $W$ be the noise with associated covariance $Q$ determined by
\begin{equation}\label{eq:covariance-logeuler}
	g(\xi)= (1+|\xi|^2)^{-1} \log^{-2\gamma}(e+|\xi|).
\end{equation}
Then the corresponding $2$D logEuler SPDE given by System \ref{system:logEuler} admits a global weak solution $\omega$ in the sense of Definition \ref{defn:solution-nonlinear-SPDE}, which is unique in law.
Moreover the solution map $(\omega_0,f)\mapsto \L(\omega)$ is continuous as a function from $L^1_x\cap L^2_x\cap \dot H^{-1}_x \times L^1_t (L^1_x\cap L^2_x\cap \dot H^{-1}_x)$ to $\cP(C^w_t L^2_x)$.
\end{theorem}

\begin{proof}
Clearly $g\in L^1_\xi \cap L^\infty_\xi$, thus $W$ satisfies Assumption \ref{ass:covariance-basic} and global existence of weak solutions holds by virtue of Proposition \ref{prop:weak-existence-nonlinear}, for the choice $\cR=T_\gamma$; in order to establish uniqueness in law, we need to verify that the assumptions of Proposition \ref{prop:weak-uniqueness-nonlinear} are satisfied.
We claim that $\cQ^{-1/2}\, T_\gamma\, \curl^{-1}$ is a bounded operator from $L^2_x\cap \dot H^{-1}_x$ to $L^2_x$; once this is shown, the validity of \eqref{eq:weak-uniqueness-assumption1}-\eqref{eq:weak-uniqueness-assumption2} immediately follows from the a priori bounds on $\omega$, $\rho$ given by Proposition \ref{prop:weak-existence-nonlinear}.

In order to show the claim, let us first observe that by identity \eqref{eq:covariance-projector} and our choice \eqref{eq:covariance-logeuler} of $g$, it holds $\cQ = T_{2\gamma} (I-\Delta)^{-1} \Pi$, where $\Pi$ is the Leray-Helmholtz projector and all the operators commute; moreover $\Pi\, \curl^{-1}=\curl^{-1}$ by construction.
Therefore
\begin{align*}
	\cQ^{-1/2} T_\gamma \curl^{-1} = (I-\Delta)^{1/2} T_\gamma^{-1} \Pi T_\gamma \curl^{-1} = (I-\Delta)^{1/2} \curl^{-1};
\end{align*}
using its representation as a Fourier multiplier and Parseval's theorem, it then holds
\begin{align*}
	\| \cQ^{-1/2} T_\gamma \curl^{-1} \varphi\|_{L^2_x}^2
	\sim \int_{\R^d} \frac{1+|\xi|^2}{|\xi|^2} |\hat{\varphi}(\xi)|^2\, \d \xi = \| \varphi\|_{\dot H^{-1}_x}^2 + \| \varphi\|_{L^2_x}^2
\end{align*}
yielding the claim.

It remains to show continuity of the solution map $(\omega_0,f)\mapsto \L(\omega)$; namely, for any continuous bounded map $F:C^w_t L^2_x \to \R$ and any sequence $(\omega_0^n,f^n)\to (\omega_0,f)$, we need to show that the corresponding solutions $\omega^n$, $\omega$ satisfy $\langle F, \L(\omega^n)\rangle\to \langle F, \L(\omega)\rangle$ as $n\to \infty$.
Corresponding to the deterministic data $(\omega^n_0,f^n)$ (resp. $(\omega_0,f)$), let $\rho^n$ be the strong solutions, defined on the same probability space $(\Theta,\F,\P)$ and driven by the same $W$, to the linear SPDEs
\begin{equation*}
	\d \rho^n + \d W\cdot\nabla \rho^n = (f^n+\kappa \Delta\rho^n) \, \d t,
\end{equation*}
which amount to removing the logEuler nonlinearity.
By the first formula in \eqref{eq:formula-density-girsanov}, it holds
\begin{align*}
	\langle F, \L(\omega^n)\rangle = \E_\P \big[ F(\rho^n)\, \cE( T_\gamma \curl^{-1} \rho^n) \big],
\end{align*}
similarly for $\langle F, \L(\rho^n) \rangle$. By the boundedness of $F$ and dominated convergence, in order to conclude it then suffices to verify that
\begin{align*}
	F(\rho^n)\to F(\rho)\quad \P\text{-a.s.}, \quad \cE( T_\gamma \curl^{-1} \rho^n) \to \cE( T_\gamma \curl^{-1} \rho) \quad \text{in } L^1_\Theta.
\end{align*}
The first claim follows from the continuity of $F$ in $C^w_t L^2_x$ and the stability of the solution map $(\omega_0,f)\mapsto \rho$, given by Theorem \ref{thm:main-linear-extended}.
For the second one, observing that $\cE( T_\gamma \curl^{-1} \rho^n)$ is a sequence of nonnegative random variables with
\begin{align*}
	\E_{\P} \big[\cE( T_\gamma \curl^{-1} \rho^n) \big] = \E_{\P} \big[\cE( T_\gamma \curl^{-1} \rho) \big]=1,
\end{align*}
it suffices to show that $\cE( T_\gamma \curl^{-1} \rho^n)\to \cE( T_\gamma \curl^{-1} \rho)$ in probability.
By formula \eqref{eq:exponential-martingale} and continuity of the exponential, it then suffices to show that
\begin{align*}
	\P- \lim_{n\to\infty} \int_0^T \big\langle \cQ^{-1}T_\gamma \curl^{-1} (\rho^n_s-\rho_s), \d W_s \big\rangle = 0, \quad
	\P-\lim_{n\to\infty} \int_0^T \| T_\gamma \curl^{-1} (\rho^n_s-\rho_s)\|_{\cH}^2\, \d s = 0
\end{align*}
which by Lemma \ref{lem:stoch-integr-basic} will both follow if we show that
\begin{equation}\label{eq:nonlinear-stability-proof}
	\lim_{n\to\infty} \int_0^T \E\big[ \| T_\gamma \curl^{-1} (\rho^n_s-\rho_s)\|_{\cH}^2\big]\, \d s = 0.
\end{equation}
Recalling that by the previous computation $\| T_\gamma \curl^{-1}\varphi\|_{\cH} \lesssim \| \varphi\|_{L^2_x\cap \dot{H}^{-1}_x}$, the last claim \eqref{eq:nonlinear-stability-proof} now follows from Corollary \ref{cor:stability-linear-energy}, completing the proof.
\end{proof}

\begin{theorem}\label{thm:main-hypoNS}
Let $\omega_0$, $f$ be as in Theorem \ref{thm:main-logeuler}, $\nu>0$, $\beta\in (0,2)$ and let $W$ be the noise with associated covariance $Q$ determined by
\begin{equation}\label{eq:covariance-hypoNS}
	g(\xi)=(1+|\xi|^2)^{-1-\beta/2}.
\end{equation}
Then the corresponding $2$D hypodissipative Navier--Stokes SPDE given by System \ref{system:hypoNS} admits a global weak solution $\omega$ in the sense of Definition \ref{defn:solution-nonlinear-SPDE}, which is unique in law.
Moreover the solution map $(\omega_0,f)\mapsto \L(\omega)$ is continuous as a function from $L^1_x\cap L^2_x\cap \dot H^{-1}_x \times L^1_t (L^1_x\cap L^2_x\cap \dot H^{-1}_x)$ to $\cP(C^w_t L^2_x)$.
\end{theorem}

\begin{proof}
The proof is mostly identical to the one of Theorem \ref{thm:main-logeuler}, so let us only comment on the main differences.
In this case we have $\cR=I$, $\cQ=(1-\Delta)^{-1-\beta/2} \Pi$, so that $\cQ^{-1/2} \cR \curl^{-1}$ is the Fourier multiplier associated to $(1+|\xi|^2)^{(1/2+\beta/4)} |\xi|^{-2} \xi^\perp$.
Using Parseval, it is then easy to check that it is a bounded linear operator from $H^{\beta/2}_x\cap \dot H^{-1}_x$ to $L^2_x$, so that the requirements \eqref{eq:weak-uniqueness-assumption1}-\eqref{eq:weak-uniqueness-assumption2} are satisfied thanks to the a priori estimates \eqref{eq:apriori-vorticity-L2}-\eqref{eq:apriori-energy} (applied for both $\omega$ and $\rho$, respectively $\cR=I$ and $\cR=0$).
Uniqueness in law thus follows from Proposition \ref{prop:weak-uniqueness-nonlinear}; the part concerning stability is then identical to before, based on stability at the linear level for $\rho^n$ and the associated densities $\cE( \curl^{-1} \rho^n)$; the only difference is that, rather than claim \eqref{eq:nonlinear-stability-proof}, due to the different choice of $\cQ$ and $\cR$ one now needs to verify that
\begin{equation*}
	\lim_{n\to\infty} \int_0^T \E\big[ \| \cQ^{-1/2} \curl^{-1} (\rho^n_s-\rho_s)\|_{L^2_x}^2\big] \d s = 0
\end{equation*}
which again follows from the bounds on the operator and Corollary \ref{cor:stability-linear-energy}.
\end{proof}

We are finally ready to prove our main result.

\begin{proof}[Proof of Theorem \ref{thm:main-theorem}]

Part a) of the statement corresponds to Theorem \ref{thm:main-logeuler}, up to justifying the pathwise regularity of the noise $W$.
By choosing $g$ as in \eqref{eq:covariance-logeuler} for some $\gamma>1/2$, the $C^0_t L^2_\loc$ regularity now comes from Proposition \ref{prop:pathwise-regularity-noise}-a); in the case $\gamma>3/4$ we get $C^0_t C^0_\loc$ regularity by virtue of Proposition \ref{prop:pathwise-regularity-noise}-b).

Part b) similarly follows from  Theorem \ref{thm:main-hypoNS}, the choice of $g$ given by \eqref{eq:covariance-hypoNS} and Proposition \ref{prop:pathwise-regularity-noise}-c) applied with $\beta=\alpha/2$.
\end{proof}

We conclude this section by discussing an extension of the above result to the torus $\T^2$, and how this relates to answering positively an open problem formulated in \cite{Fla}.

\begin{remark}\label{rem:readaptation-torus}
As evident from the above, our proofs crucially rely on two ingredients: the existence of weak solutions satisfying $\P$-a.s. the condition $\cQ^{-1/2} \cR \curl^{-1} \omega \in L^2_t L^2_x$, so to reduce ourselves to the linear SPDE by applying Girsanov to \eqref{eq:weak-uniqueness-proof1}, and the pathwise uniqueness and available bounds for the latter, mainly deduced by studying its energy spectrum as in Proposition \ref{prop:pathwise-uniqueness-linear}.
All the results transfer to the torus $\T^2$ without problems, up to replacing Fourier transform with its discrete counterpart and all the integrals in Fourier space with their corresponding sums over $\Z^2$.
In fact, the case of $\T^2$ is simpler to treat, due to the nicer behaviour of the Biot--Savart kernel: one can reduce themselves to work only with functions of zero mean, namely $\int_{\T^2} \omega(x)\, \d x=0$, in which case $\curl^{-1}$ is a bounded operator from $L^2(\T^2)$ to $H^1(\T^2)$; therefore there is no need to find estimates for $\| \omega\|_{\dot H^{-1}_x}$ like in Proposition \ref{prop:apriori-estimates-nonlinear} anymore, one has directly pathwise bounds for $\|\curl^{-1}\omega\|_{H^1_x}$ as well as compact embeddings, and the tightness argument from Proposition \ref{prop:weak-existence-nonlinear} becomes simpler.
On the other hand, pathwise uniqueness for the linear SPDE can be established similarly, see the proof of \cite[Theorem 3]{Gal}.
\end{remark}

\begin{remark}\label{rem:flandoli}
On the torus $\T^2$, Flandoli considered in \cite{Fla} a regularized version of the $2$D Euler equations, corresponding in our setup to $\cR_\gamma:=(1-\Delta)^{-\gamma/2}$ (cf. \cite[eq. (2)]{Fla}) and posed the following question \cite[beginning p.17]{Fla}:

\emph{``Do there exist a noise $W$ and an exponent $\gamma\leq 1$ such that uniqueness holds for the SPDE with initial condition $\omega_0\in L^2_x$?''}

Our results answer affirmatively to the above, in terms of uniqueness in law, for any $\gamma\in (0,1]$. Indeed, take a noise $W$ with covariance $Q$ associated to\footnote{We use $k$ in place of $\xi$ to stress that $k\in\Z^2$;
recall that on $\T^2$, we can work with functions $\omega$ such that $\hat\omega(0)=0$, so we don't need to worry about the behaviour of any Fourier multiplier around $0$.} $g(k)= |k|^{-2} \log^{-2} (e+|k|)$;
existence of weak solutions $\omega$ with paths in $L^\infty_t L^2_x$ holds, thus in order for the above strategy to work, it suffices to check that $\cQ^{-1/2} \cR_\gamma\, \curl^{-1}$ is a bounded operator from $L^2(\T^2)$ to $L^2(\T^2;\R^2)$.
This is certainly the case, since the operator corresponds to a Fourier multiplier of the form
\begin{align*}
	\hat\omega(k)\mapsto r(k)\, \hat\omega (k), \quad r(k):=\frac{\log (e + |k|)}{(1+|k|^2)^{\gamma/2}}\, \frac{k^\perp}{|k|}
\end{align*}
and $r\in \ell^\infty(\Z^2 \setminus \{0\})$ for all $\gamma>0$.
\end{remark}

\section{Proof of Proposition \ref{prop:pde-nonuniquess}}\label{sec:nonuniqueness-deterministic}

In this section we focus on the deterministic hypodissipative 2D Navier--Stokes equation
\begin{equation}\label{eq:2D-hypodissip-NSE}
	\partial_t u + u\cdot\nabla u+ \nabla P = F- \Lambda^\beta u, \quad  \nabla\cdot u =0
\end{equation}
where $\beta\in (0,2)$ and $F\in L^1_tL^2_x$.
It is well known that for $\beta=2$ (standard 2D Navier-Stokes), existence and uniqueness of Leray--Hopf solutions hold for any $u_0\in L^2_x$.
However, for any $\beta\in (0,2)$, for some carefully chosen forcing $F\in L^1_tL^2_x$, Albritton and Colombo \cite{AlbCol} recently showed the existence of two distinct Leray--Hopf solutions $\bar{u}$, $u$ to \eqref{eq:2D-hypodissip-NSE}, both with $u_0\equiv 0$.

The goal of this section is to go through the proof from \cite{AlbCol} and verify that for $\beta\in (0,1)$, the solutions $\bar{u}$, $u$ constructed in \cite{AlbCol} satisfy the desired regularity claimed in Proposition \ref{prop:pde-nonuniquess}.
To this end, let us first briefly describe the main ideas from \cite{AlbCol}, which build on the works by Vishik \cite{Vishik1, Vishik2} and their revisitations in \cite{ABCDGJK, AlBrCo} and is based on \emph{linear instability} techniques, the construction of an \emph{unstable background} $\bar{u}$ and the analysis of the associated \emph{unstable manifold}.

The construction from \cite{AlbCol} is rather explicit. The authors start by transforming \eqref{eq:2D-hypodissip-NSE} in vorticity form (for $f:=\curl\, F$)
\begin{equation}\label{eq:2D-hypodissip-NSE-vort}
	\partial_t \omega + u\cdot\nabla\omega = f- \Lambda^\beta \omega,\quad  u = \mbox{curl}^{-1} \omega
\end{equation}
and rewrite the PDE in similarity variables $\xi:= t^{-1/\alpha}x$, $\tau:= \log t$, for a suitable parameter $\alpha>\beta$ to be chosen later.
Introducing the \emph{similarity profiles}
\begin{align*}
	\Omega(\tau, \xi) := e^\tau \omega(e^\tau, e^{\tau/\alpha}\xi), \quad
	\bar f(\tau, \xi) := e^{2\tau} f(e^\tau, e^{\tau/\alpha}\xi),
\end{align*}
the PDE \eqref{eq:2D-hypodissip-NSE-vort} becomes
\begin{equation}\label{eq:2D-hypodissip-NSE-simil-varia}
	\partial_\tau \Omega - \Big(1+ \frac{\xi}{\alpha} \cdot\nabla_\xi\Big) \Omega+ U\cdot\nabla \Omega + e^{\gamma \tau} \Lambda^\beta \Omega = \bar f, \quad U= \curl^{-1} \Omega
\end{equation}
where $\gamma := 1-\beta/\alpha \in (0,1)$.
The idea from \cite{AlbCol} is to treat the dissipation term $e^{\gamma \tau} \Lambda^\beta$ as a perturbation of the 2D Euler equation in similarity variables:
\begin{equation}\label{eq:2D-Euler-simil-varia}
	\partial_\tau \Omega - \Big(1+ \frac{\xi}{\alpha} \cdot\nabla_\xi\Big) \Omega+ U\cdot\nabla \Omega = \bar f, \quad U= \curl^{-1} \Omega.
\end{equation}

Any steady state $\bar\Omega$ for \eqref{eq:2D-Euler-simil-varia} corresponds to a self-similar solution $\bar{\omega}$ for the $2$D Euler equations; in particular, one would like to look at an \emph{unstable steady state} $\bar\Omega$, such that the linearization of \eqref{eq:2D-Euler-simil-varia} around $\bar\Omega$ admits an eigenfunction with eigenvalue $\lambda\in\mathbb{C}$ satisfying $ {\rm Re} \lambda >0$ (in the terminology of \cite[Proposition 2.1]{AlbCol}, $\bar{\Omega}$ is an \emph{unstable vortex}).
The forcing $\bar{f}$ is then chosen exactly so that $\bar{\Omega}$ solves \eqref{eq:2D-hypodissip-NSE-simil-varia}.
The second self-similar solution $\Omega$ to \eqref{eq:2D-hypodissip-NSE-simil-varia} instead is constructed by looking at the \emph{unstable manifold} associated to the background $\bar{\Omega}$, namely enforcing the solution ansatz
\begin{align*}
	\Omega= \bar\Omega+ \Omega^{\rm lin} + \Omega^{\rm per},
\end{align*}
where $\Omega^{\rm lin}$ is an exponentially growing solution of the linearization of \eqref{eq:2D-Euler-simil-varia} around $\bar{\Omega}$ and $\Omega^{\rm per} = o_{\tau\to -\infty} (|\Omega^{\rm lin}|)$.

Once this program is carried out, going back to the physical variables $(x,t)$, one overall finds a forcing $f$ and two distinct solutions $\bar\omega$, $\omega$ to \eqref{eq:2D-hypodissip-NSE-vort}, both starting at $0$, given by
\begin{equation}\label{eq:unstable-solut+forcing}
	\bar\omega_t(x)= t^{-1}\bar\Omega(t^{-\frac{1}{\alpha}} x), \quad
	\omega_t(x)= t^{-1}\Omega(\log t, t^{-\frac{1}{\alpha}} x), \quad
	f= \partial_t \bar\omega + \Lambda^\beta \bar\omega.
\end{equation}

We are now going to show that, for suitable values of $0<\beta<\alpha<2$, the solutions $\bar\omega$, $\omega$ above  belong to the class $L^2_t \big( L^1_x\cap \dot{H}^{-1}_x \cap H^{\beta/2}_x \big)$.
In particular, Proposition \ref{prop:pde-nonuniquess} is a consequence of the following result.

\begin{proposition}\label{prop:regularity-nonunique-solut}
If $0<\beta< \alpha< 1$, then the solutions $\bar\omega$ and $\omega$ constructed in \cite{AlbCol} satisfy
\begin{equation}\label{eq:regularity-nonunique-solut}
	\bar{\omega}, \,\omega \in L^2_t \big(L^1_x \cap \dot H^{-1}_x \cap H^{\beta/2}_x \big), \quad
	\bar u,\, u \in L^\infty_t L^2_x \cap L^2_t H^{\beta/2}_x,
\end{equation}
where $\bar u= \curl^{-1} \bar\omega,\, u= \curl^{-1}\omega$. Moreover, the forcing term verifies
\begin{equation}\label{eq:regularity-forcing}
	f\in L^1_t\big(L^1_x\cap L^2_x \cap \dot H^{-1}_x \big).
\end{equation}
\end{proposition}

\begin{proof}
For simplicity, in the following we restrict ourselves to the time interval $t\in [0,1]$.
Theorem 1.1 in \cite{AlbCol} implies $\bar{u},\, u\in L^\infty_t L^2_x\cap L^2_t H^{\beta/2}_x$ and $F\in L^1_t L^2_x$;
since $\bar\omega=\curl\, \bar u$, $\omega=\curl\, u$ and $f=\curl\, F$, we have $\bar\omega,\, \omega\in L^2_t \dot H^{-1}_x$ and $f\in L^1_t \dot H^{-1}_x$.
Thus we are left with verifying that
\begin{equation}\label{eq:pde-nonuniq-goal}
	\bar{\omega}, \,\omega \in L^2_t \big(L^1_x \cap H^{\beta/2}_x \big), \quad
	f\in L^1_t\big(L^1_x\cap L^2_x\big).
\end{equation}
As described above, the unstable background vorticity $\bar\omega$ is given by \eqref{eq:unstable-solut+forcing}, where by \cite[Proposition 2.1]{AlbCol} $\bar\Omega$ is a smooth, compactly supported function and $\alpha$ is a parameter of our choice satisfying $0<\beta<\alpha<2$ (cf. \cite[beginning of Section 2]{AlbCol}); moreover, the forcing term $f$ is defined as in \eqref{eq:unstable-solut+forcing} (see \cite[Lemma 2.3]{AlbCol} for $f=h$).

We divide the verification of \eqref{eq:pde-nonuniq-goal} in a few steps.

\emph{Step 1}.
We start by verifying the properties of $\bar\omega$.
In light of \eqref{eq:unstable-solut+forcing} and scaling relations for $L^p_x$-norms in $d=2$, for any $p\in [1,\infty]$ it holds
\begin{align*}
	\int_0^1 \| \bar\omega_t\|_{L^p_x}^2\, \d t
	= \| \bar\Omega\|_{L^p_x}^2 \int_0^1 t^{-2+\frac{4}{\alpha p}}\, \d t<\infty \quad \text{ for } \alpha<\frac{4}{p}
\end{align*}
so that $\bar\omega\in L^2_t (L^1_x\cap L^2_x)$ as soon as $\alpha<2$.
Next, recall the scaling relations for the fractional Laplacian: for any $\delta\in \R$ and any function $f$, it holds
\begin{equation}\label{fractional-Laplacian}
	\Lambda^\delta (f(\lambda\, \cdot))(x) = \lambda^\delta (\Lambda^\delta f)(\lambda x).
\end{equation}
Then we have
\begin{align*}
	\int_0^1 \| \bar\omega_t\|_{\dot H^{\beta/2}_x}^2\,\d t
	= \int_0^1 \| \Lambda^{\beta/2} \bar\omega_t\|_{L^2_x}^2\,\d t
	= \| \Lambda^{\beta/2} \bar\Omega \|_{L^2_x}^2 \int_0^1 t^{-2-\frac{\beta}{\alpha}+\frac{2}{\alpha}}\, \d t <\infty
\end{align*}
as soon as $2+\beta/\alpha - 2/\alpha <1$, that is, $\alpha+\beta<2$;
in the above we used the fact that $\bar\Omega$ is smooth and compactly supported, so that $\| \Lambda^{\beta/2} \bar\Omega \|_{L^2_x}<\infty$.
To sum up, we proved the regularity properties of $\bar\omega$. \smallskip

\emph{Step 2}. Next, we deal with the forcing term $f$.
Set $\bar\Omega_\nabla (\xi) := \nabla\bar\Omega(\xi) \cdot \xi$, which is also a smooth, compactly supported function; observe that by \eqref{eq:unstable-solut+forcing} and \eqref{fractional-Laplacian} it holds
\begin{align*}
	\partial_t \bar\omega_t
	= -t^{-2} \bar\Omega(t^{-\frac{1}{\alpha}} x) -\frac{1}{\alpha} t^{-2} \bar\Omega_\nabla (t^{-\frac{1}{\alpha}} x), \quad
	\Lambda^\beta \bar\omega_t
	= t^{-1-\frac{\beta}{\alpha}} (\Lambda^\beta \bar\Omega)(t^{-\frac{1}{\alpha}} x).
\end{align*}
Then again by scaling of $L^p_x$-norms we find
\begin{align*}
	\int_0^1 \| \partial_t \bar\omega_t\|_{L^p_x} \, \d t
	&\lesssim \int_0^1 t^{-2+\frac{2}{\alpha p}} \big(\| \bar\Omega\|_{L^p_x} + \| \bar\Omega_\nabla \|_{L^p_x} \big)\, \d t
	<\infty \quad \text{holds for } \alpha< \frac{2}{p},\\
 	\int_0^1 	\| \Lambda^\beta \bar\omega_t \|_{L^p_x}\, \d t
	&\lesssim \int_0^1 t^{-1-\frac{\beta}{\alpha}+\frac{2}{\alpha p}} \| \Lambda^\beta \bar\Omega\|_{L^p_x}\, \d t
	<\infty \quad \text{holds for } \beta<\frac{2}{p};
\end{align*}
in particular, $f\in L^1_t (L^1_x\cap L^2_x)$ as soon as $\alpha\vee\beta <1$.
%

\emph{Step 3}. It remains to verify that $\omega\in L^2_t \big(L^1_x\cap H^{\beta/2}_x \big)$.
Recall that $\omega$ is given in function of $\Omega$ by relation \eqref{eq:unstable-solut+forcing}, where $\Omega=\bar{\Omega}+\Omega^{\rm lin} + \Omega^{\rm per}$ (cf. \cite[Eq. (1.9)-(1.11)]{AlbCol}); to verify the requirements on $\omega$, we shall verify them on $\bar{\Omega}$, $\Omega^{\rm lin}$ and $\Omega^{\rm per}$ separately.
The term $\bar{\Omega}$ was already treated in Step 1.; let us first give a useful observation for treating the others.

Given a function $\gamma=\gamma(t,x)$, denote by $\Gamma=\Gamma(\tau,\xi)$ the corresponding function in self-similar variables, so that $\gamma(t,x)=e^{-\tau}\, \Gamma(\tau, \xi)$ with $\tau= \log t$, $\xi= t^{-1/\alpha} x$; then
\begin{align*}
	\int_0^1 \| \gamma(t,\cdot)\|_{L^p_x}^2\, \d t
	& = \int_0^1 t^{-2} \bigg(\int_{\R^2} |\Gamma(\log t, t^{-\frac{1}{\alpha}} x)|^p\,\d x \bigg)^{\frac{2}{p}}\, \d t \\
	& = \int_0^1 t^{-2+ \frac{4}{\alpha p}} \| \Gamma(\log t,\cdot)\|_{L^p_\xi}^2\, \d t
	= \int_{-\infty}^0 e^{\tau\big( \frac{4}{\alpha p} - 1\big)} \| \Gamma(\tau,\cdot)\|_{L^p_\xi}^2\, \d \tau;
\end{align*}
the convergence of the integral is dictated by the asymptotic decay of $\| \Gamma(\tau,\cdot)\|_{L^p_\xi}$ as $\tau\to -\infty$.

We now start by checking $\omega\in L^2_t (L^1_x\cap L^2_x)$.
By \cite[Eq. (2.12)]{AlbCol}, it holds $\Omega^{\rm lin}(\tau,\xi) = {\rm Re} (e^{\lambda \tau} \eta)$ with $(\lambda,\eta)$ given as in \cite[Proposition 2.1]{AlbCol};
specifically, for a fixed parameter $a\geq a_0$, with $a_0=8/\alpha>0$ as at the beginning of \cite[Section 3]{AlbCol}, we can find a corresponding $\lambda\in\mathbb{C}$ with ${\rm Re} \lambda=a$ and  $\eta\in L^2_m\subset L^2_\xi$, $\eta$ compactly supported (cf. items 1), 3) of \cite[Proposition 2.1]{AlbCol})\footnote{Here $L_m^2$ consists of $m$-fold rotationally symmetric $L^2_\xi$-functions.}.
Therefore, for any $p\in [1,2]$ it holds
\begin{align*}
	\int_{-\infty}^0 e^{\tau\big(\frac{4}{\alpha p}-1\big)} \|\Omega^{\rm lin}(\tau,\cdot)\|_{L^p_\xi}^2\, \d \tau
	\lesssim \| \eta\|_{L^1_\xi\cap L^2_\xi}^2 \int_{-\infty}^0 e^{\tau\big( \frac{4}{\alpha p} + 2a-1\big)}\, \d \tau
\end{align*}
which is finite for $\alpha<4/p$, since $a>0$; in particular the cases $p=1,2$ are covered for $\alpha<2$.

The term $\Omega^{\rm per}$ can be controlled with the estimates from the proof of \cite[Theorem 1.1]{AlbCol}, see pp. 14--15, especially estimate (4.2) therein (valid for $\Omega^{{\rm per},(k)}$ but then also for $\Omega^{\rm per}$ once we take $k\to\infty$).
In particular, the $X$-norm, as defined in \cite[Eq. (3.3)]{AlbCol} for the choice $Q=1$, controls the $L^1_\xi\cap L^2_\xi$-norm, thus it holds
\begin{equation}\label{eq:estim-Omega-per}
	\| \Omega^{\rm per}(\tau,\cdot)\|_{L^1_\xi\cap L^2_\xi}
	\leq e^{\tau (a+\delta_0)} \quad \forall\, \tau\leq \bar{\tau}.
\end{equation}
Here $\bar\tau$ is a suitable finite value, fixed at the end of proof of \cite[Theorem 1.1]{AlbCol}; the parameter $\delta_0$ is defined at the beginning of \cite[Section 3]{AlbCol} as $\delta_0:=\min\{\gamma/4,1/8,a_0/2\}$ where $\gamma:= 1-\beta/\alpha>0$ (cf. \cite[Eq. (1.13)]{AlbCol}).
Therefore for any $p\in [1,2]$ we obtain
\begin{align*}
	\int_{-\infty}^{\bar \tau } e^{\tau\big(\frac{4}{\alpha p}-1\big)} \| \Omega^{\rm per}(\tau,\cdot)\|_{L^p_\xi}^2\, \d \tau
	\leq \int_{-\infty}^{\bar\tau} e^{\tau\big(\frac{4}{\alpha p}+2a+2\delta_0-1\big)}\, \d \tau
\end{align*}
which is again finite for any $\alpha<4/p$.
Overall, this verifies that for $\alpha<1$ it holds
\begin{align*}
	\int_{(-\infty,0)} e^{\tau\big(\frac{4}{\alpha p}-1 \big)} \|\Omega(\tau,\cdot)\|_{L^1_\xi\cap L^2_\xi}^2\, \d \tau <\infty
\end{align*}
and so correspondingly that $\omega\in L^2_t (L^1_x\cap L^2_x)$.

Finally, we check that $\omega\in L^2_t \dot H^{\beta/2}_x$; we will show a bit more, namely that $\omega-\bar\omega\in L^2_t \dot H^1_x$, where $\omega-\bar\omega=:\tilde\omega$ corresponds in self-similar variables to $\Omega^{{\rm lin}}+\Omega^{{\rm per}}=: \tilde\Omega$.
By \eqref{eq:unstable-solut+forcing}, the previous observations and scaling relations we have
\begin{align*}
	\|\tilde\omega(t,\cdot) \|_{\dot H^1_x}
	= \|\nabla \tilde\omega(t,\cdot) \|_{L^2_x}
	= t^{-1-1/\alpha} \|\nabla \tilde\Omega(\log t,\cdot) (t^{-1/\alpha} \cdot) \|_{L^2_x}
	= t^{-1} \|\nabla \tilde\Omega(\log t,\cdot) \|_{L^2_\xi};
\end{align*}
hence
\begin{align*}
	\int_0^1 \|\tilde\omega(t,\cdot) \|_{\dot H^1_x}^2\,\d t
	= \int_0^1 t^{-2} \|\nabla \tilde\Omega(\log t,\cdot) \|_{L^2_\xi}^2\,\d t
	= \int_{-\infty}^0 e^{-\tau} \|\nabla\tilde\Omega(\tau,\cdot) \|_{L^2_\xi}^2\,\d \tau.
\end{align*}
We treat the terms $\Omega^{\rm lin}$ and $\Omega^{\rm per}$ separately.
As before, $\Omega^{\rm lin}(\tau,\xi) = {\rm Re} (e^{\lambda \tau} \eta)$, where by \cite[Proposition 2.1]{AlbCol}, ${\rm Re}\lambda =a\ge a_0 =8/\alpha$ and $\eta\in H^8_x$;
therefore $\|\nabla \Omega^{\rm lin} (\tau,\cdot) \|_{L^2_x} \le e^{a\tau} \|\eta \|_{H^1_x}$ and
\begin{align*}
	\int_{-\infty}^0 e^{-\tau} \|\nabla \Omega^{\rm lin}(\tau,\cdot) \|_{L^2_x}^2\,\d \tau
	\le \int_{-\infty}^0 e^{-\tau} e^{2a\tau} \|\eta \|_{H^1_x}^2 \,\d \tau <+\infty
\end{align*}
due to the choice of $a$, which for $\alpha\in (0,2)$ satisfies $a\geq 8/\alpha>4$.
Next, by the definition of norm $\|\cdot \|_X$ in \cite[Eq. (3.3)]{AlbCol} and estimate (4.2) in \cite{AlbCol}, we have $\|\Omega^{\rm per}(\tau,\cdot)\|_{H^1_x} \le e^{(a+\delta_0)\tau}$; consequently
\begin{align*}
	\int_{-\infty}^0 e^{-\tau} \|\nabla \Omega^{\rm per}(\tau,\cdot) \|_{L^2_x}^2\,\d \tau
	\le \int_{-\infty}^0 e^{-\tau} e^{2(a+\delta_0)\tau} \,\d \tau <+\infty.
\end{align*}
Collecting all the above estimates completes the proof.
\end{proof}


\section{Open problems and future directions}\label{sec:open}

We have shown that, in the presence of a sufficiently active, Kraichnan transport noise, the well-posedness theory of fluid dynamics equations can be improved considerably, yielding existence and uniqueness results for a class of initial vorticities which are merely $L^2_x$-integrable.
Our techniques work for logEuler equations with $\gamma>1/2$, where the corresponding deterministic best results available in the literature require at least $\omega_0\in H^1_x$ or $\omega_0\in L^\infty_x$ (cf. \cite{ChaeWu}), and in the case of hypodissipative Navier--Stokes equations, where by Proposition \ref{prop:pde-nonuniquess} we have explicit counterexample to uniqueness in $L^2_x$.

Nevertheless, the main question, whether there is hope to achieve similar results for $2$D Euler, remains open; we formulate it as a conjecture, for simplicity in the absence of forcing.

\begin{conjecture}\label{conj:euler}
There exists a noise $W$ satisfying Assumption \ref{ass:covariance-basic} such that the stochastic 2D Euler equations
\begin{equation}\label{eq:stoch-euler}
	\d \omega + \curl^{-1} \omega\cdot \nabla \omega\, \d t + \circ \d W\cdot \nabla \omega = 0
\end{equation}
are well-posed in law, for all initial data $\omega_0$ belonging to $L^1_x\cap L^p_x\cap \dot H^{-1}_x$, for some $p<\infty$.
\end{conjecture}

Several variants of Conjecture \ref{conj:euler} can be formulated; for instance, rather than looking at $\omega_0\in L^1_x\cap L^p_x\cap \dot H^{-1}_x$, one could consider $\omega_0\in \dot H^s_x\cap \dot H^{-1}_x$ for some $s<1$.
In general, one would like to look at initial data in a \emph{supercritical class} $X$, with respect to the scaling invariance of the deterministic Euler equations $\omega(x)\mapsto \omega^\lambda (x)=\omega(\lambda x)$; such class should satisfy $\| \omega^\lambda\|_X\to \infty$ as $\lambda\to 0$ and thus cannot be included in the Yudovich regime.

Some partial results in this direction are already available.

\begin{corollary}\label{cor:euler-conjecture}
The following hold:
\begin{itemize}
\item[i)] For any $p\in [2,\infty]$, any noise $W$ satisfying Assumption \ref{ass:covariance-basic} and any $\omega_0\in L^1_x\cap L^p_x\cap \dot H^{-1}_x$, there exists a weak solution $\omega$ to \eqref{eq:stoch-euler}, with weakly continuous paths and almost surely in $L^\infty_t (L^1_x\cap L^p_x\cap \dot H^{-1}_x)$.
\item[ii)] Consider now the noise $W$ with $g$ given by \eqref{eq:covariance-logeuler} and the corresponding SPDE \eqref{eq:stoch-euler}. Then for any $\omega_0\in L^1_x\cap L^2_x\cap \dot H^{-1}_x$, there can be at most one solution on $[0,T]$ satisfying
\begin{equation}\label{eq:requirement-euler-uniqueness}
	\omega\in L^\infty_t (L^1_x\cap L^2_x\cap \dot H^{-1}_x), \quad
	\int_0^T \| T_\gamma^{-1} \omega_t\|_{L^2_x}^2\, \d t <\infty\quad \P\text{-a.s.}
\end{equation}
in the sense that, if two weak solutions both satisfy \eqref{eq:requirement-euler-uniqueness}, then their laws on $[0,T]$ necessarily coincide.
\end{itemize}
\end{corollary}

\begin{proof}
Part i) is a direct consequence of Proposition \ref{prop:weak-existence-nonlinear}, for the choice $\cR=I$, $\nu=0=f$.

Part ii) instead follows from the first part of Proposition \ref{prop:weak-uniqueness-nonlinear}, with $\cQ=(I-\Delta)^{-1} T_{2\gamma} \Pi$ and $\cR=I$; indeed, arguing by Fourier multipliers as in the proof of Theorem \ref{thm:main-logeuler}, one can see that condition \eqref{eq:requirement-euler-uniqueness} is sufficient to verify that
\begin{equation*}
 	\int_0^T \| \cQ^{-1/2} \curl^{-1}\omega_s\|_{L^2_x}^2\, \d s <\infty\quad \P\text{-a.s.}\quad \qedhere
\end{equation*}
\end{proof}

Corollary \ref{cor:euler-conjecture} tells us that, in order to answer positively Conjecture \ref{conj:euler} (or its variants), it would suffice to find a supercritical class $X$ of initial data which propagate the \emph{logarithmic regularity estimate} dictated by condition \eqref{eq:requirement-euler-uniqueness}.
In particular, the uniqueness question turns into the problem of constructing sufficiently nice solutions, or rather (by compactness arguments) finding a priori estimates for them.
As the noise $W$ given by \eqref{eq:covariance-logeuler} is very rough, propagation of regularity is highly non trivial, even when looking at the linear SPDE for $\rho$ without Euler nonlinearity, for which well-posedness is already known.
In that case, we believe that a careful study of the evolution for the energy spectrum $a_t(\xi)=\E[|\hat\rho_t(\xi)|^2]$, given by \eqref{eq:pde-energy-density}, could provide some useful insight on the problem.
Indeed a natural, partially stronger version of \eqref{eq:requirement-euler-uniqueness} can be formulated in Fourier as
\begin{equation}\label{eq:log-derivative-requirement}
	\int_0^T \int_{\R^2} \E\big[|\hat\omega_t(\xi)|^2 \big] \log^{2\gamma}(e+|\xi|)\, \d \xi\d t<\infty.
\end{equation}

Let us mention that the propagation of a ``logarithmic derivative'', as encoded by \eqref{eq:log-derivative-requirement}, is a phenomenon well-understood in the study of solutions to transport equations; however not in the stochastic setting, but rather for DiPerna--Lions flows.
Initial studies of the (very low) regularity propagated by such rough flows are due to Crippa and De Lellis \cite{CriDeL}; L\'eger \cite{Leger} then obtained bounds in Fourier space, closer to \eqref{eq:log-derivative-requirement}. Bru\'e and Nguyen recently showed that such regularity can be equivalently reformulated in real space by a suitable Gagliardo-Nirenberg seminorm in \cite{BruNgu1} and provided sharp regularity estimates for the associated solutions to the continuity equations in \cite{BruNgu2}.
We hope that in the future a combination of the ideas from this paper and the analysis techniques and regularity estimates from DiPerna-Lions theory will yield a positive resolution to Conjecture \ref{conj:euler}.

Either way, let us discuss Point ii) in Corollary \ref{cor:euler-conjecture} from a different perspective.
Recall Vishik's nonuniqueness result for $L^p_x$-valued solutions to the forced $2$D Euler equations, first presented in \cite{Vishik1,Vishik2} and then revisited in \cite{ABCDGJK}; going through similar computations as in Proposition \ref{prop:pde-nonuniquess}, one can verify that such non-unique solutions do satisfy the (deterministic analogue of the) regularity requirement \eqref{eq:requirement-euler-uniqueness}.
In other words, even though uniqueness of solutions to stochastic 2D Euler equations is open, we can still deduce that a sufficiently active noise $W$ prevents the arising of the same non-uniqueness scenario as for the forced deterministic equation.
This fact has a natural heuristical explanation: as the noise acts at the Lagrangian level as a very active environment, mixing particles around, it prevents any stationary or self-similar structure from forming or preserving itself for sufficiently long time.
This applies in particular to the existence of an \emph{unstable background} and its associated \emph{unstable manifold}, which are at the heart of Vishik's approach.

In a different direction, it would be interesting in the future to get a better understanding of the long-time behaviour of the (linear and nonlinear) SPDEs for which we have successfully obtained weak well-posedness in this paper;
we start with a simple observation.

\begin{corollary}\label{cor:Feller-property}
Consider any of the following systems:
\begin{enumerate}[i)]
\item the unforced stochastic $2$D logEuler equation, with $\gamma>1/2$ fixed and $W$ determined by \eqref{eq:covariance-logeuler};
\item the unforced stochastic $2$D hypodissipative Navier--Stokes, with $\nu>0$, $\beta\in (0,1)$ fixed and $W$ determined by \eqref{eq:covariance-hypoNS};
\end{enumerate}
Set $X:=L^1_x\cap L^2_x\cap \dot H^{-1}_x$; denote by $\mathcal{B}_b(X)$ the space of all bounded Borel measurable functions $F:X\to\R$ and by $C_b(X)$ the one of continuous bounded $F$.
For any $\omega_0$ and $t>0$, denote by $\L(\omega(t,\omega_0))$ the law at time $t$ of the weakly unique solution starting at $\omega_0$. Finally, define a family of linear operators on $\mathcal{B}_b(X)$ by
\begin{align*}
	(P_t F)(\omega_0):= \int_X F(v) \L(\omega(t,\omega_0))(\d v).
\end{align*}
Then $\{P_t\}_t$ is a Markov semigroup on $\mathcal{B}_b(X)$, in the sense that the relation $P_{t+s}=P_t\circ P_s$ holds, which is moreover Feller, in the sense that $P_t$ maps $C_b(X)$ into itself for all $t\geq 0$.
\end{corollary}

\begin{proof}
Although so far we always considered deterministic $\omega_0$, standard disintegration arguments allow to similarly establish weak existence and uniqueness in law for random $\F_0$-measurable data $\omega_0$ (note that the Banach space $X$ is separable).
The Markov property for $\{P_t\}_t$ is then a consequence of the SPDE being autonomous and the independent increments property of the driving noise $W$; it can be rigorously justified by first looking at more regular systems (like the ones used in the proof of Proposition \ref{prop:weak-existence-nonlinear}) for which it holds classically, and then passing to the limit in the approximations.
The Feller property instead is a consequence of the final parts of Theorems \ref{thm:main-logeuler}-\ref{thm:main-hypoNS}, which imply continuity of the map $\omega_0\mapsto \L(\omega(t,\omega_0))$ in the topology of weak convergence of measures.
\end{proof}

A similar result holds for the linear SPDE in any dimension $d\geq 2$, cf. Remark \ref{rem:invariant-measure-linear}.

In light of Corollary \ref{cor:Feller-property}, one could hope to apply ergodic theory techniques to understand the long-time behaviour of solutions (see \cite{DaPZab96} for a general overview); standard arguments would require to verify the \emph{Strong Feller} and \emph{irreducibility} properties of the semigroup, which we leave for future investigations.

The problem seems relatively trivial for $\nu>0$, as the presence of the fractional viscosity $\nu\Lambda^\beta$ should make any solution dissipate and thus vanish in the long-time limit in the strong $L^2_x$-norm.
The situation is much more interesting for $\nu=0$, namely for the inviscid $2$D logEuler equations and the incompressible Kraichnan noise; whether convergence to $0$ still holds, up to possibly replacing convergence in the strong topology with the weak one, here is less clear.
Let us mention some facts hinting at this possibility being realistic:
\begin{enumerate}[a)]
\item The only known stationary solutions for the SPDE are the constant ones, contrary to the deterministic case; for instance for $2$D Euler equations there are shear, radial and cellular flows among others. This again reflects the idea that a sufficiently active noise $W$ acts as an ever changing, erratic background which destroys any stationary configuration.
\item Although our techniques provide well-posedness of the SPDE, they do not imply preservation of the $L^2_x$ norm of solutions, which is formally expected; a priori estimates and approximation techniques only yield $\P$-a.s. inequalities of the form $\| \omega_t\|_{L^2_x}\leq \| \omega_0\|_{L^2_x}$, which could be strict, in which case \emph{anomalous dissipation} takes place.
In the presence of Stratonovich, transport-type noise, this phenomenon has been rigorously shown for dyadic models in \cite{BFM11}; it is however open for fluid dynamics SPDEs.
Anomalous dissipation taking place with positive probability would imply that the map $t\mapsto \E \big[\| \omega_t\|_{L^2_x}^2 \big]$  is decreasing, thus possibly converging to $0$ as $t\to\infty$.
\item For the incompressible Kraichnan model, by Remark \ref{rem:invariant-measure-linear} we know that the only possible invariant measure on $E= \big\{\varphi\in L^1_x: \int_{\R^d} \varphi(x) \d x = 0 \big\}$ is $\delta_0$, suggesting that $\{0\}$ should be the only global attractor for the linear system.
For the regularized Kraichnan model, under suitable assumptions on $W$, one can in fact infer the much stronger property that \emph{solutions get mixed exponentially fast}, which in particular implies their weak convergence to 0; see \cite[Section 5]{GesYar} and the references therein.
\end{enumerate}

In conclusion, our understanding of how transport type Kraichnan noise affects both the well-posedness and long-time behaviour of nonlinear PDEs is still very limited, especially in the context of inviscid equations.

\appendix

\section{Some useful lemmas}\label{app:useful}

The next lemma is classical, see for instance \cite[eq. (2.3)]{LFNLTZ} in the 2D setting.

\begin{lemma}\label{lem:identity-enstrophy}
Let $D$ be either $\T^d$ or $\R^d$, $d=2,3$, $\varphi\in H^1(D;\R^d)$ divergence free; then
\begin{equation*}
	\| \nabla \varphi\|_{L^2_x} = \| \curl\, \varphi\|_{L^2_x}.
\end{equation*}
\end{lemma}

\begin{proof}
We only treat $D=\R^3$, the other cases being similar; we can assume $\varphi\in C^\infty_c$, the general case follows by density. Computing the Fourier transforms of $\nabla \varphi$ and $\curl\, \varphi$, by Parseval identity it holds
\begin{align*}
	\| \nabla \varphi\|_{L^2_x}^2
	= (2\pi)^{-d} \int_{\R^3} |\xi|^2 |\hat\varphi(\xi)|^2 \d \xi, \quad
	\| \curl\, \varphi\|_{L^2_x}^2
	= (2\pi)^{-d} \int_{\R^3} |\xi\times \hat\varphi(\xi)|^2 \d \xi.
\end{align*}
By the divergence free condition, $\hat\varphi(\xi)\cdot \xi=0$ for all $\xi$, which implies $|\xi\times \hat\varphi(\xi)|=|\xi| \, |\hat\varphi(\xi)|$ and thus the conclusion.
\end{proof}

\begin{lemma}\label{lem:Lp-norms-appendix}
For any $\varphi\in L^1_x\cap L^\infty_x$, it holds $ \| \varphi\|_{L^\infty_x} =\lim_{p\to\infty} \| \varphi\|_{L^p_x}$.
\end{lemma}

\begin{proof}
By interpolation it holds $\| \varphi\|_{L^p_x} \leq \| \varphi\|_{L^1_x}^{1/p} \| \varphi\|_{L^\infty_x}^{1-1/p}$, which implies
\begin{equation*}
	\limsup_{p\to\infty} \| \varphi\|_{L^p_x} \leq \| \varphi\|_{L^\infty_x}.
\end{equation*}
Conversely, by definition of $\| \varphi\|_{L^\infty_x}$, for any $\eps>0$ there exists a set $A_\eps$ with ${\rm Leb} (A_\eps)>0$ such that $|\varphi(x)|\geq \| \varphi\|_{L^\infty_x} -\eps$ on $A_\eps$; therefore
\begin{equation*}
\liminf_{p\to\infty} \| \varphi\|_{L^p_x}
\geq \liminf_{p\to\infty}\, {\rm Leb} (A_\eps)^{1/p} (\| \varphi\|_{L^\infty_x} -\eps) = \| \varphi\|_{L^\infty_x} -\eps.
\end{equation*}
By arbitrariness of $\eps>0$, the conclusion follows.
\end{proof}

\begin{lemma}\label{lem:decomposition-fractional-laplacian}
Let $\beta\in (0,2)$; then for any $R>0$, the fractional Laplacian admits a decomposition $\Lambda^\beta=\Lambda^\beta_{R,1}+\Lambda^\beta_{R,2}$ with the following properties:
\begin{itemize}
\item[i)] $\Lambda^\beta_{R,i}$ are symmetric linear operators;
\item[ii)] $\Lambda^\beta_{R,1}$ maps $C^\infty_c$ into bounded functions with compact support;
\item[iii)] for any $p\ge 1$, $\Lambda^\beta_{R,2}$ maps $L^p_x$ into itself and $\| \Lambda^\beta_{R,2} \|_{L^p_x\to L^p_x} \lesssim_{d,\beta,p} R^{-\beta}$.
\end{itemize}
\end{lemma}

\begin{proof}
Recall the integral representation \eqref{eq:fractional-laplacian1} of $\Lambda^\beta$; for fixed $R>0$ and $\varphi\in C^\infty_c$, set
\begin{equation*}
	\Lambda^\beta_{R,1} \varphi (x) := C_{d,\beta}\, {\rm P.v.}\! \int_{|x-y|\leq R} \! \frac{\varphi(x)-\varphi(y)}{|x-y|^{d+\beta}}\, \d y, \ \
	\Lambda^\beta_{R,2} \varphi (x) := C_{d,\beta} \! \int_{|x-y|> R} \! \frac{\varphi(x)-\varphi(y)}{|x-y|^{d+\beta}}\, \d y.
\end{equation*}
The symmetry of $\Lambda^\beta_{R,i}$ is clear from their definition. The fact that $\Lambda^\beta_{R,1}\varphi$ is a uniformly bounded function comes from the identity
\begin{equation*}
	\Lambda^\beta_{R,1} \varphi (x) := C_{d,\beta} \int_{|x-y|\leq R} \frac{\varphi(x)-\varphi(y)-\nabla\varphi(x)\cdot(x-y)}{|x-y|^{d+\beta}}\, \d y
\end{equation*}
and the estimate
\begin{equation*}
	|\Lambda^\beta_{R,1} \varphi (x)|
	\lesssim \| \varphi\|_{C^2_b} \int_{|x-y|\leq R} |x-y|^{2-d-\beta}\, \d y
	\lesssim \| \varphi\|_{C^2_b} R^{2-\beta}.
\end{equation*}
It is also clear from the above that ${\rm supp}\, \Lambda^\beta_{R,1} \varphi \subset {\rm supp}\, \varphi + B_R$.

It remains to prove iii), which can be accomplished by a direct computation; by H\"older's inequality,
  $$\aligned
  |\Lambda^\beta_{R, 2} \varphi(x)|^p &\lesssim \bigg|\int_{|x-y|> R} \! \frac{\varphi(x)-\varphi(y)}{|x-y|^{d+\beta}}\, \d y \bigg|^p \\
  &\le \bigg[\int_{|x-y|> R} \! \frac{1}{|x-y|^{d+\beta}}\, \d y\bigg]^{p-1} \int_{|x-y|> R} \! \frac{|\varphi(x)-\varphi(y)|^p}{|x-y|^{d+\beta}}\, \d y \\
  &\lesssim R^{-\beta(p-1)} \int_{|x-y|> R} \! \frac{|\varphi(x)|^p +|\varphi(y)|^p}{|x-y|^{d+\beta}}\, \d y.
  \endaligned $$
Therefore, by symmetry,
\begin{align*}
	\int_{\R^d} |\Lambda^\beta_{R, 2} \varphi(x)|^p\, \d x
	& \lesssim R^{-\beta(p-1)} \int_{\R^d} \int_{|x-y|> R}\! \frac{|\varphi(x)|^p +|\varphi(y)|^p}{|x-y|^{d+\beta}}\,\d y\d x \\
	& = 2 R^{-\beta(p-1)} \int_{\R^d} |\varphi(x)|^p\,\d x \int_{|x-y|> R}\! \frac{1}{|x-y|^{d+\beta}}\,\d y \\
    &\lesssim R^{-\beta p} \|\varphi \|_{L^p_x}^p.
\end{align*}
By standard density arguments, the operator $\Lambda^\beta_{R,2}$ extends to the whole $L^p_x$ with the desired operator norm estimate.
\end{proof}

Next, let us discuss a bit more in detail the Fr\'echet topology of suitable local function spaces;
we focus on $H^\alpha_\loc$, which will be relevant for the proof of Proposition \ref{prop:weak-existence-nonlinear}, but similar considerations also apply e.g. for $C^n_\loc$.

We can endow $H^\alpha_\loc$ with the following family of seminorms $\| \cdot\|_{s,n}$:
\begin{equation}\label{eq:seminorms-Hs-loc}
	\| f\|_{\alpha,n} := \sup\bigg\{\, \frac{|\langle f,\varphi \rangle|}{\| \varphi\|_{H^{-\alpha}_x} } \, : \, \varphi\in C^\infty_c, \ {\rm supp }\, \varphi\subset \overline{B(0,n)} \bigg\}
\end{equation}
which clearly induce its topology; such family makes $H^\alpha_\loc$ a separable metric vector space, with distance
\begin{equation}\label{eq:metric-Hs-loc}
	d_\alpha (f^1,f^2):=\sum_{n=1}^\infty 2^{-n}\, \big(1\wedge \| f^1 -f^2 \|_{\alpha,n} \big).
\end{equation}
The usefulness of $d_\alpha$ comes from the following fact.

\begin{lemma}\label{lem:Hs-loc}
Let $\alpha\in \R$, then $H^\alpha_x$ compactly embeds into $H^{\alpha-\eps}_\loc$ for any $\eps>0$; moreover, for any $f^1,\, f^2\in H^\alpha_x$ it holds $d_\alpha(f^1,f^2)\leq \| f^1-f^2\|_{H^\alpha_x}$.
\end{lemma}

\begin{proof}
The second claim follows immediately from the observation that, by the definition \eqref{eq:seminorms-Hs-loc}, $\| f^1-f^2\|_{\alpha,n}\leq \| f^1-f^2\|_{H^\alpha_x}$, and the definition of $d_s$ given in \eqref{eq:metric-Hs-loc}.

Concerning the first claim, we need to show that given any bounded sequence $\{f^n\}_n$ in $H^\alpha_x$, we can extract a subsequence such that $d_{\alpha-\eps}(f^n, f)\to 0$; this follows from a standard diagonalization argument.
Indeed, as the sequence is bounded and $H^\alpha_x$ is Hilbert, upon extraction of subsequence we can assume that it converges weakly in $H^\alpha_x$ to some $f$; then for any $\psi\in C^\infty_c$, the sequence $\{\psi f^n\}_n$ converges to $\psi f$ in $H^\alpha_x$, but being compactly supported it is also precompact in $H^{\alpha-\eps}_x$, thus converges strongly therein.
Upon choosing properly the functions $\psi$, this shows that $\| f^n-f\|_{\alpha,n}\to 0$ for any $n$ and thus the conclusion.
\end{proof}

We can readily deduce Ascoli--Arzel\`a compactness type results.

\begin{corollary}\label{cor:ascoli-Hs-loc}
Let $\alpha\in \R$, $A\subset C([0,T];H^\alpha_x)$ be a family of uniformly bounded, equicontinuous functions (in the time variable, seen as $H^\alpha_x$-valued paths).
Then $A$ is precompact in $C([0,T];H^{\alpha-\eps}_\loc)$, for any $\eps>0$.
\end{corollary}

\begin{proof}
By the assumptions and Lemma \ref{lem:Hs-loc}, the set $\{f(t):f\in A\}$ is bounded in $H^\alpha_x$ and thus precompact in $H^{\alpha-\eps}_\loc$; on the other hand, since the family is equicontinuous in $H^\alpha_x$ and
\begin{align*}
	d_{\alpha-\eps}(f_t,f_s) \leq \| f_t-f_s\|_{H^{\alpha-\eps}_x} \leq \| f_t-f_s\|_{H^\alpha_x},
\end{align*}
it is also equicontinuous in $H^{\alpha-\eps}_\loc$.
The conclusion follows from Ascoli--Arzel\`a's theorem in abstract metric spaces.
\end{proof}

\begin{remark}\label{rem:Hs-loc-interpolation}
By standard interpolation arguments, one can easily verify that for any $\alpha_1< \alpha_2$ and any $\theta\in (0,1)$ it holds
\begin{align*}
	d_{\theta \alpha_1 + (1-\theta)\alpha_2}(f^1,f^2) \leq d_{\alpha_1} (f^1,f^2)^\theta\, \|f^1-f^2 \|_{H^{\alpha_2}_x}^{1-\theta}.
\end{align*}
In particular, if $f^n\to f$ in $H^{\alpha_1}_\loc$ and $\{f^n\}_n$ is a bounded sequence in $H^{\alpha_2}_x$, then $f^n$ converges strongly to $f$ in $H^\alpha_\loc$ for all $\alpha\in (\alpha_1, \alpha_2)$, as well as weakly in $H^{\alpha_2}_x$ (by reflexivity of this space).
Similarly, in the setting of Corollary \ref{cor:ascoli-Hs-loc}, if now $\{f^n\}_n$ is an equicontinuous sequence in $C([0,T];H^{\alpha_1}_x)$, which is also uniformly bounded in $L^\infty(0,T;H^{\alpha_2}_x)$\,\footnote{These two assumptions together readily imply that the maps $t\mapsto f^n_t$ are weakly continuous in $H^{\alpha_2}_x$, see e.g. \cite[Lemma 3.5]{FlGaLu21a} for a similar argument; in particular they are uniquely defined for all $t\in [0,T]$, not just up to Lebesgue negligible sets.},
then one can extract a (not relabelled) subsequence such that $f^n\to f$ in $C([0,T];H^\alpha_\loc)$ for all $\alpha\in (\alpha_1, \alpha_2)$;
moreover $f^n_t \rightharpoonup f_t$ weakly in $H^{\alpha_2}_x$, for all $t\in [0,T]$.
\end{remark}

We conclude this appendix by collecting some useful facts on $\curl^{-1}$.

\begin{lemma}\label{lem:homogeneous-H1}
Let $f\in C_c^\infty(\R^2)$, then $f\in \dot H^{-1}_x$ if and only if $\int_{\R^d} f(x)\,\d x =0$.
\end{lemma}

\begin{proof}
Since $f$ is a Schwartz test function, so is its Fourier transform $\hat{f}$, hence it is rapidly decaying at infinity and behaves like $\hat{f}(0)=\int_{\R^d} f(x)\,\d x$ near the origin. We conclude from these facts that
\begin{align*}
	\| f \|_{\dot H^{-1}_x}^2 = \int_{\R^d} \frac{|\hat{f}(\xi)|^2}{|\xi|^2}\, \d \xi<\infty
\end{align*}
is equivalent to $\hat{f}(0)=0$.
\end{proof}

\begin{lemma}\label{lem:convergence-nonlinearity}
Let $d\in\{2,3\}$, $\{f^n\}_n$ be a bounded sequence in $L^1_x\cap L^2_x\cap \dot H^{-1}_x$ such that $f_n\rightharpoonup f$ in $L^2_x$ and $\cR$ be the Fourier multiplier associated to $r\in L^\infty_x$.
Then for any $\varphi\in C^\infty_c$ it holds
\begin{align*}
	\langle \cR \, \curl^{-1} f^n\cdot \nabla \varphi, f^n\rangle \to \langle \cR\, \curl^{-1} f\cdot \nabla \varphi, f\rangle.
\end{align*}
Similarly, if we are given a sequence $r^n$ such that $|r^n|\leq |r|$, $r^n\uparrow r$ Lebesgue a.e., and we denote by $\cR^n$ the associated Fourier multipliers, then
\begin{align*}
	\langle \cR^n \, \curl^{-1} f^n\cdot \nabla \varphi, f^n\rangle \to \langle \cR\, \curl^{-1} f\cdot \nabla \varphi, f\rangle
\end{align*}
\end{lemma}

\begin{proof}
First observe that, thanks to our assumptions, $f^n\rightharpoonup f$ in $L^p_x$ for all $p\in (1,2]$.
Recall that $\curl^{-1} f^n = K\ast f^n$, where $K(x)=C_2\, x^\perp/|x|^2$ is the Biot--Savart kernel; set $K^1(x):= K(x)\mathbbm{1}_{|x|\leq 1}$, $K^2(x):=K(x)-K^1(x)=K(x)\mathbbm{1}_{|x|> 1}$, so that $K^1 \in L^r_x$ for all $r\in [1,2)$ and $K^2\in L^r_x$ for all $r\in (2,\infty]$.

Since $f^n\rightharpoonup f$ in $L^2_x$ and $K^1\in L^1_x$, $K^1\ast f^n\rightharpoonup K^1\ast f$ in $L^2_x$; by a similar argument, $K^2\ast f^n\rightharpoonup K^2\ast f$ in $L^r$ for any $r\in (2,\infty)$.
This implies $\langle \curl^{-1} f^n,\varphi\rangle \to \langle \curl^{-1} f,\varphi\rangle$ for all $\varphi\in C^\infty_c$;
by assumption $\{\curl^{-1} f^n\}_n$ is a bounded sequence in $L^2_x$, thus $\curl^{-1} f^n \rightharpoonup \curl^{-1} f$ in $L^2_x$.
In fact, thanks to Lemma \ref{lem:identity-enstrophy} everything is bounded in $H^1_x$ and weak convergence holds therein.
Since $\cR$ is a continuous operator from $H^1_x$ to itself, the same holds for $\cR\, \curl^{-1} f^n\rightharpoonup \cR\, \curl^{-1} f$;
then by Lemma \ref{lem:Hs-loc}, strong convergence in $L^2_\loc$ holds as well.

The conclusion in the first case then follows by a combination of $f^n\rightharpoonup f$ in $L^2_x$ and $\cR\, \curl^{-1} f^n\cdot\nabla \varphi\to \cR\, \curl^{-1} f\cdot\nabla \varphi$, where the latter holds since $\varphi\in C^\infty_c$.

The proof in the case of varying $\cR^n$ is similar; just observe that $\{\cR^n \curl^{-1} f^n\}_n$ is still bounded in $H^1_x$, and that under our assumptions $\cR^n \to \cR$ pointwise in $H^1_x$, as can be checked by computing $\| (\cR-\cR^n)\psi\|_{H^1_x}^2$ in Fourier space.
Combining these facts, one can still deduce that $\cR^n \curl^{-1} f^n\rightharpoonup \cR \curl^{-1} f$ in $H^1_x$, thus also strongly in $L^2_\loc$; the rest of the argument follows identically.
\end{proof}

\section{Equivalence of It\^o and Stratonovich formulations}\label{app:equiv-ito-strat}

We present here the

\begin{proof}[Proof of Proposition \ref{prop:equivalence-ito-stratonovich}]
Let us provide a detailed proof for $\rho$ satisfying condition \textit{i}), and later discuss the differences in the proof of case \textit{ii}).
Suppose that $\rho$ is a Stratonovich weak solution and fix $\varphi\in C^\infty_c$; define the sequence of semimartingales
\begin{align*}
	X^n_t := \sum_{k\leq n} \int_0^t \<\sigma_k\cdot\nabla\varphi, \rho_s\> \circ \d B^k_s, \quad t\in [0,\tau],
\end{align*}
which are known to converge in the sense of semimartingales, by \eqref{eq:solution-abstract-stratonovich}.
Without loss of generality, we can assume the associated family of increasing stopping times $\{\tau^m\}_m$ to be such that $\P$-a.s.,
\begin{equation}\label{eq:equivalence-ito-strat-proof1}
	\int_0^{\tau_m} \big( \|b_s\cdot\nabla\varphi\|_{L^{p'}_x} + \| f_s\, \varphi\|_{L^1_x} \big)\, \d s + \sup_{s\in [0,\tau^m]} \| \rho_s\|_{L^p_x} \leq m;
\end{equation}
otherwise, just replace $\tau^m$ by $\tau^m\wedge \tilde\tau^m$, for $\tilde{\tau}^m$ stopping times such that \eqref{eq:equivalence-ito-strat-proof1} holds.

It follows from the definition of $X^n$ and properties of Stratonovich integrals that $X^n=V^n+M^n$, where $M^n=\sum_{k\leq n} \int_0^\cdot \<\sigma_k\cdot\nabla\varphi, \rho_s\>\, \d B^k_s$;
we claim that $M^n$ converge as continuous local martingales on $[0,\tau)$ to $M:=\int_0^\cdot \< \rho_s\nabla\varphi,\d W \>$.
Indeed, it holds
\begin{align*}
	[M^n-M]_{\tau_m}
	& = \sum_{k=n+1}^{+\infty} \int_0^{\tau_m} |\<\sigma_k\cdot\nabla\varphi, \rho_s\>|^2\, \d s\\
	& = \sum_{k=n+1}^{+\infty} \int_0^{\tau_m} \int_{\R^{2d}} (\sigma_k\cdot \nabla\varphi(x))(\sigma_k\cdot \nabla\varphi(y)) \rho_s(x)\rho_s(y)\, \d x \d y \d s\\
	& = \int_0^{\tau_m} \int_{\R^{2d}} Q_n^\perp (x,y) : \nabla\varphi(x)\otimes \nabla\varphi(y)\, \rho_s(x)\rho_s(y)\, \d x \d y \d s
\end{align*}
for $Q^\perp_n(x,y):= \sum_{k>n} \sigma_k(x)\otimes \sigma_k(y)$.
By Lemma \ref{lem:covariance-series-representation}, $Q^\perp_n$ converges to $0$ uniformly on compact sets, while $\rho_s\nabla\varphi $ is compactly supported and in $L^1_x$, uniformly over $s\in [0,\tau_m]$ in view of \eqref{eq:equivalence-ito-strat-proof1}.
By dominated convergence we can thus deduce that
\begin{align*}
	\lim_{n\to\infty} \E\big( [M^n-M]_{\tau_m} \big) = 0\quad \forall\, m\in\N
\end{align*}
proving the claim.
Since $X^n$ converge to the l.h.s. of \eqref{eq:solution-abstract-stratonovich} in the sense of semimartingales, and the quadratic covariation is insensitive to the bounded variation component, it must hold, for all $j\in\N$,
\begin{align*}
	\big[\< \rho_\cdot, \varphi\>, B^j\big]_\cdot
	= \lim_{n\to\infty} \big[X^n,B^j\big]_\cdot
	= \lim_{n\to\infty} \big[M^n,B^j\big]_\cdot
	= \big[M,B^j\big]_\cdot
	= \int_0^\cdot \< \sigma_j\cdot\nabla \varphi, \rho_s \>\, \d s
\end{align*}
on $[0,\tau)$; in the last passage we used the fact that $\d [B^k,B^j]_t = \delta_{j,k}\, \d t $ since $\{B^k\}_k$ are independent standard Brownian motions.

Applying the same argument for $\varphi$ replaced by $\sigma_k\cdot\nabla\varphi\in C^\infty_c$, we deduce that
\begin{equation}\label{eq:equivalence-ito-strat-proof2}
	\frac{\d}{\d t} \big[ \< \rho_\cdot, \sigma_k\cdot\nabla\varphi\>, B^k\big]
	= \< \sigma_k\cdot\nabla(\sigma_k\cdot\nabla \varphi), \rho_t\>
	= - \< \sigma_k\cdot\nabla \varphi, \sigma_k\cdot\nabla\rho_t\>
\end{equation}
where the last passage is allowed by integration by parts, since $\sigma_k$ are divergence free and $\P$-a.s. $\rho\in L^1([0,\tau);W^{1,1}_\loc)$.
We can now use \eqref{eq:equivalence-ito-strat-proof2} to compute the bounded variation part $V^n$ of $X^n$:
\begin{equation}\label{eq:equivalence-ito-strat-proof3}
	V^n_\cdot
	= \frac{1}{2}\sum_{k\leq n} \big[ \<\sigma_k\cdot\nabla\varphi, \rho \>, B^k\big]_\cdot
	= -\frac{1}{2}\sum_{k\leq n} \int_0^\cdot  \<\sigma_k\cdot\nabla\varphi,\sigma_k\cdot\nabla \rho_s\>\, \d s.
\end{equation}
In order to conclude that $\rho$ is an It\^o solution, namely to identify the limit of $X^n$ with the respective It\^o integral and It\^o-Stratonovich corrector, it only remains to show that $V^n$ are converging on $[0,\tau)$ to
\begin{align*}
	\kappa \int_0^\cdot \< \Delta\varphi, \rho_s\>\,\d s
	= -\kappa \int_0^\cdot \<\nabla\varphi, \nabla\rho_s \>\,\d s
	= -\frac{1}{2} \int_0^\cdot \< \nabla\varphi, Q(0) \nabla \rho_s\>\,\d s;
\end{align*}
but this follows from similar arguments as the ones presented above, based on the uniform convergence on compact sets of $\sum_{k\leq n} \sigma_k(x)\otimes \sigma_k(x) \to Q(0)$, localization argument and dominated convergence.

Conversely, suppose now that $\rho$ is a solution in the sense of Definition \ref{defn:solution-abstract-ito}, again satisfying \textit{i}).
In this case, \eqref{eq:solution-abstract-ito} directly yields the semimartingale decomposition of $\langle \varphi,\rho\rangle$, from which relation \eqref{eq:equivalence-ito-strat-proof2} immediately follows.
For fixed $\varphi\in C^\infty_c$, we can then define the sequence $X^n$ as before, and compute its decomposition $X^n=V^n+M^n$, again given by $M^n=\sum_{k\leq n} \int_0^\cdot \<\sigma_k\cdot\nabla\varphi, \rho_s\>\, \d B^k_s$ and \eqref{eq:equivalence-ito-strat-proof3}.
The verification that $M^n+V^n\to M+V$ on $[0,\tau)$ in the sense of semimartingales, for $M=\int_0^\cdot \< \rho_s\nabla\varphi_s, \d W\>$ and $V=\kappa \int_0^\cdot \< \Delta\varphi, \rho_s\>\, \d s$ is then similar to the previous one, concluding the proof of this implication.

Let us finally comment on how the proof must be modified if condition \textit{ii}) holds instead.
The main argument is the same, based in both implications on showing that $M^n+V^n\to M+V$ as above; the part concerning $M^n\to M$ remains unchanged, which leads to the same formula for $V^n$.
The main difference now is that since $\rho$ is not known to be in $W^{1,1}_\loc$, we cannot enforce the final integration by parts in \eqref{eq:equivalence-ito-strat-proof2} and rather we need to show that $V^n\to V$ by only manipulating the formula
\begin{align*}
	V^n_\cdot = \frac{1}{2}\sum_{k\leq n} \int_0^\cdot \< \sigma_k\cdot\nabla(\sigma_k \cdot\nabla\varphi), \rho_s\>\, \d s;
\end{align*}
we can decompose it as $V^n=V^{n,1}+V^{n,2}$, for
\begin{align*}
	V^{n,1}_\cdot = \frac{1}{2}\sum_{k\leq n} \int_0^\cdot \< \sigma_k\otimes \sigma_k : D^2 \varphi, \rho_s\>\, \d s,
	\quad V^{n,2}_\cdot = \frac{1}{2}\sum_{k\leq n} \int_0^\cdot \< (\sigma_k\cdot\nabla \sigma_k)\cdot\nabla\varphi, \rho_s\>\, \d s.
\end{align*}
The convergence of $V^{n,1}_\cdot$ to $V$ is similar as before, based on Lemma \ref{lem:covariance-series-representation}; thus it remains to show that $V^{n,2}\to 0$. Formally this should hold since, for any $j=1,\ldots, d$,
\begin{align*}
	\sum_{k\in\N} (\sigma_k\cdot\nabla\sigma_k)_j(x)
	= \sum_{k\in\N} (\sigma_k(x)\cdot\nabla\sigma_k(y))_j\Big\vert_{x=y}
	= - \sum_{i=1}^d \partial_i Q^{ij}(x-y) \Big\vert_{x=y}
\end{align*}
and $\partial_i Q^{ij}(0)= 0$ since $Q$ is even.
In order to justify this rigorously, we need some uniform bounds on $V^{n,2}$ as $n\to\infty$; by condition \textit{ii}), Lemma \ref{lem:covariance-series-representation} and Proposition \ref{prop:regularity-Q-sigma}, it holds
\begin{align*}
	\sup_{x\in\R^d} \sum_{k>n} |\sigma_k(x)\cdot\nabla \sigma_k(x)|
	\leq \frac{1}{2} \sup_{x\in\R^d} \sum_{k>n} |\sigma_k(x)|^2 + \frac{1}{2} \sup_{x\in\R^d} \sum_{k>n} |\nabla \sigma_k(x)|^2 <\infty;
\end{align*}
moreover one can show that $\sum_{k>n} |\sigma_k(x)\cdot\nabla \sigma_k(x)|$ converges to $0$ as $n\to\infty$ uniformly on compact sets.
Dominated convergence and localization arguments then imply the conclusion.
\end{proof}

\section{Pathwise regularity of the noise}\label{app:pathwise-regularity}

The aim of this appendix is to give a fairly self-contained proof of Proposition \ref{prop:pathwise-regularity-noise}.
For simplicity, we will only prove that all the statements of Proposition \ref{prop:pathwise-regularity-noise} (the space where $W$ belongs and the $\P$-a.s. convergence of the series) hold at fixed time $t=1$; extending them to getting uniform convergence on some finite interval $[0,T]$ is then a routine argument.
Indeed, if $W$ is a $\cQ$-Brownian motion taking values in a Banach space $E$, then by Fernique's theorem it admits arbitrarily high finite moments; by Brownian scaling and Kolmogorov's continuity theorem, one can deduce that its paths are $\P$-a.s. in $C^{1/2-\eps}_t E$ and then reach the conclusion.
For simplicity, we will also manipulate $W_1(x)$ as if it were $\R$-valued, the general case of $\R^d$ following from arguing componentwise.

Therefore, our task reduces to statements of the form ``$\P$-a.s. $W_1$ belongs to $E$ and the associated series representation converges therein'', where $E$ is an appropriately chosen Banach space implying the convergence in the topologies given in Proposition \ref{prop:pathwise-regularity-noise} (resp. $L^2_\loc$, $C^0_\loc$ and $C^{\alpha-\eps}_\loc$).
To this end, let us recall a general result, providing easily verifiable conditions to ensure that the above holds.

\begin{theorem}[Theorem 2.12 from \cite{DaPZab}]\label{thm:white-noise-expansion}
Let $W_1$ be a centered Gaussian variable on a separable Banach space $E$ and let $\cH$ be its Cameron-Martin space; let $\{\sigma_k\}_k$ be a CONS of $\cH$ and consider the sequence $\{B^k_1\}_k$ as defined by \eqref{eq:chaos-expansion-finite-noises}.
Then the series $\sum_{k=1}^\infty \sigma_k B^k_1$ converges $\P$-a.s. to $W_1$ in $E$.
\end{theorem}

Our task is then reduced to verifying that $W_1$ belongs to $E$ (possibly up to modification); to this end, let us recall the classical Garsia-Rodemich-Rumsay lemma.

\begin{lemma}\label{lem:GRR}
Let $L$ be a Banach space, $a\in\R^d$, $r>0$ and $h:B(a,r)\to L$ a continuous function.
Suppose that there exist strictly increasing, continuous functions $p,\Psi:\R_{\geq 0}\to \R_{\geq 0}$ such that $\Psi(0)=p(0)=0$ and
\begin{equation}\label{eq:GRR-hypothesis}
	\int_{B(a,r)\times B(a,r)} \Psi\bigg( \frac{\| h(x)-h(y)\|_L}{p(|x-y|)} \bigg)\, \d x \d y \leq K;
\end{equation}
then there exists a dimensional constant $C_d$ such that
\begin{equation}\label{eq:GRR-conclusion}
	\| h(x)-h(y)\|_L
	\leq 8 \int_0^{2|x-y|} \Psi\big(C_d\, K\, s^{-2d} \big)\, p'(s)\, \d s
	\quad \forall\, x,y\in B(a,r).
\end{equation}
\end{lemma}

The statement above was first presented in $d=1$ in \cite[Lemma 1.1]{GRR} and admits several variants, see \cite[Exercise 2.41]{StrVar} or \cite[Theorem B.1.1]{DaPZab96}.
For simplicity we restricted to continuous functions $h$, but one can allow for merely integrable ones, in which case the conclusion becomes the existence of a continuous representative $\tilde{h}$, such that $h=\tilde h$ Lebesgue a.e., satisfying \eqref{eq:GRR-conclusion}; the proof is based on first mollifying $h$, see \cite[Proposition 3.4 and Theorem 3.5]{DaPZab} for a detailed presentation in a particular case.
Similarly, when applying the result to possibly discontinuous stochastic processes, one can deduce the existence of a continuous modification.

In the special case of centered Gaussian variables, moment bounds on $\| W_1(x)-W_1(y)\|_L$ always imply exponential ones by Fernique's theorem; therefore one can take $\Psi$ growing exponentially fast at infinity in Lemma \ref{lem:GRR}, upon choosing $p$ satisfying the above conditions and such that $(W_1(x)-W_1(y))/p(x-y)^{1/2}$ has uniformly bounded moments.
Lemma \ref{lem:GRR} then admits the following specialization, cf. \cite[Theorem 2.1]{GRR}; the proof therein is only given in $d=1$, but the higher dimensional generalization is immediate in view of Lemma \ref{lem:GRR}.

\begin{corollary}\label{cor:GRR}
Let $\{W_1(x):x\in B(a,r)\}$ be a collection of $L$-valued, centered Gaussian variables, $r\leq 1$;
suppose there exists a strictly increasing $p:\R_{\geq 0}\to \R_{\geq 0}$ such that $p(0)=0$,
\begin{align*}
	\sup_{x,y\in B(a,r): |x-y|\leq s} \E\big[\| W_1(x)-W_1(y)\|_L^2\big] \leq p(s)
\end{align*}
and
\begin{equation}\label{eq:GRR-cor-condition}
	\int_0^r |\log s|^{1/2}\, p'(s)\, \d s<\infty.
\end{equation}
Then $W_1$ admits a modification belonging to $C(B(a,r);L)$.
\end{corollary}

We are now ready to present the

\begin{proof}[Proof of Proposition \ref{prop:pathwise-regularity-noise}]
As mentioned before, in each case of interest, it suffices to show that $\P$-a.s. $W_1\in E$, for a suitably chosen Banach space $E$.

\textit{Proof of a).}
Define the measure $\mu(\d x)= (1+|x|)^{-d-2}\, \d x$ and consider be the associated weighted $L^2$-space, namely $E=L^2(\R^d,\mu)$;
it is clear that $E\hookrightarrow L^2_\loc$ with continuous embedding.
Recall that by homogeneity of the noise, it holds $\E[|W_1(x)|^2]=\Tr Q(0) =2d \kappa$ for all $x\in \R^d$; therefore
\begin{align*}
	\E\big[\| W_1\|_E^2\big] = \int_{\R^d} \E[|W_1(x)|^2]\, \mu(\d x) \sim \int_{\R^d} (1+|x|)^{-d-2}\, \d x <\infty.
\end{align*}
We conclude by Theorem \ref{thm:white-noise-expansion} that the series $\sum_{k=1}^\infty \sigma_k B^k_1$ converges $\P$-a.s. to $W_1$ in $E$, thus also in $L^2_{\loc}$.

\textit{Proof of b).}
This case is a bit more elaborate; we begin by showing that, under condition \eqref{eq:noise-logarithmic-regularity}, there exists $\delta\in (0,1)$ such that
\begin{equation}\label{eq:pathwise-regularity-proof1}
	\E\big[|W_1(x)-W_1(y)|^2\big] \lesssim \big|\log |x-y| \big|^{^{1-2\gamma}}
	\quad \forall\, x,y, |x-y|\leq \delta.
\end{equation}
By homogeneity, it holds $\E[|W_1(x)-W_1(y)|^2] = 2\Tr [Q(0)- Q(x-y)]$, so we only need to estimate the latter;
expressing $Q$ in terms of its antitransform, for any $R>0$ we have
\begin{align*}
	|Q(0)-Q(z)|
	& \sim \bigg| \int_{\R^d} (1-e^{i\xi\cdot z}) g(\xi) P_\xi\, \d \xi \bigg|
	\lesssim \int_{\R^d} |1-e^{i\xi\cdot z}| g(\xi)\, \d \xi\\
	& \lesssim \int_{|\xi|\leq R} |\xi|\,|z| g(\xi)\, \d \xi + \int_{|\xi|>R} g(\xi)\, \d \xi =: I_1^R + I_2^R.
\end{align*}
By assumption \eqref{eq:noise-logarithmic-regularity} and a few integral computations, for sufficiently large $R$ it holds
\begin{align*}
	I_1^R
	& \lesssim |z| \int_{|\xi|\leq R} \frac{|\xi|}{1+ |\xi|^d} \log^{-2\gamma} (e+|\xi|)\, \d \xi
	\sim |z| \int_0^R \log^{-2\gamma} (e+r)\, \d r
	\lesssim |z| R \log^{-2\gamma} R,\\
	I_2^R & \lesssim \int_{|\xi|> R}(1+ |\xi|^d)^{-1} \log^{-2\gamma} (e+|\xi|)\, \d \xi
	\sim \int_R^{+\infty} r^{-1} \log^{-2\gamma} (e+r)\, \d r
	\lesssim \log^{1-2\gamma} R.
\end{align*}
Taking $\delta=R^{-1}$, for all $|z|\leq \delta$ overall we find
\begin{align*}
	|Q(0)-Q(z)|
	\lesssim \log^{-2\gamma} R  + \log^{1-2\gamma} R
	\lesssim \big|\log |z|\big|^{1-2\gamma}
\end{align*}
and thus the claim \eqref{eq:pathwise-regularity-proof1}.
We are now in the position to apply Corollary \ref{cor:GRR} for $p(s)\sim |\log s|^{1-2\gamma}$, which under the condition $\gamma>3/4$ satisfies
\begin{align*}
	\int_0^\delta |\log s|^{1/2}\, p'(s)\, \d s
	\sim \int_0^\delta s^{-1} \, \big|\log |s|\big|^{1/2-2\gamma}\, \d s <\infty;
\end{align*}
exploiting again Fernique's theorem and the homogeneity of the noise, it follows that (up to modification)
\begin{align*}
	\sup_{x\in\R^d} \E \big[ \| W_1\|_{C(B(x,\delta))}^2 \big] < \infty.
\end{align*}
By taking a countable collection of points $\{x_n\}_n$ such that $B(x_n,\delta)$ cover $\R^d$, we conclude that $W_1$ is $\P$-a.s. continuous on $\R^d$; similarly, given any $R>0$, by taking $N\sim R^\delta$ balls of radius $\delta$ covering $B(0,R)$, we can deduce that
\begin{equation}\label{eq:pathwise-regularity-proof2}
	\E\Big[\| W_1\|_{C(B(0,R))}^2 \Big] \lesssim_\delta R^d .
\end{equation}
Let us define the Banach space $E$ of continuous function $f:\R^d\to\R$ growing at most like $(1+|x|)^{d+1}$ at infinity, with norm $\| f\|_E=\sup_x (1+|x|)^{-d-1} |f(x)|$;
we claim that $W_1$ belongs $\P$-a.s. to $E$.
Indeed by \eqref{eq:pathwise-regularity-proof2}, it holds
\begin{align*}
	\E\big[ \| W_1\|_E^2 \big]
	& \lesssim \sum_{n=0}^{+\infty} \E\bigg[ \sup_{n\leq |x|\leq n+1} (1+|x|)^{-2d-2} |W(1,x)|^2 \bigg]\\
	& \lesssim \sum_{n=0}^{+\infty} n^{-2d-2}\, \E\Big[ \|W_1\|_{B(0,n)}^2 \Big]
	\lesssim \sum_{n=0}^{+\infty} n^{-2} <\infty
\end{align*}
which proves the claim; the conclusion again follows from Theorem \ref{thm:white-noise-expansion} and the continuous embedding $E\hookrightarrow C^0_\loc$.

\textit{Proof of c).}
This case is similar to b) above and actually simpler, so we only shortly sketch it.
Using homogeneity and the antitransform of $Q$, in this case one can show that, for any $\eps>0$, it holds
\begin{align*}
	\E\big[|W_1(x)-W_1(y)|^2 \big] \lesssim_\eps |x-y|^{2\alpha-\eps}\quad \forall\, x, y\in\R^d;
\end{align*}
then using hypercontractivity of Gaussian moments and Lemma \ref{lem:GRR} (alternatively, \cite[Theorem 3.5]{DaPZab}), together with homogeneity of the noise, one finds
\begin{align*}
	\sup_x \E\Big[\, \llbracket W_1\rrbracket_{C^{\alpha-\eps}(B(x,1))}^2 \Big]\lesssim_\eps 1
\end{align*}
where $\llbracket \cdot \rrbracket_{C^{\alpha-\eps}(O)}$ denotes the $(\alpha-\eps)$-H\"older seminorm computed on an open set $O$.
Given the above, one can then come up with a suitable weighted $(\alpha-\eps)$-H\"older Banach space $E$, with at most polynomial growth at infinity, such that $\P$-a.s.  $W_1\in E$ and $E\hookrightarrow C^{\alpha-\eps}_\loc$ continuously.
Theorem \ref{thm:white-noise-expansion} and the arbitrariness of $\eps>0$ then yield the conclusion.
\end{proof}

\end{document}